\documentclass[reqno]{amsart}
\usepackage[T2A,T1]{fontenc}
\usepackage[lutf8]{luainputenc}
\usepackage[USenglish]{babel}
\usepackage{amsmath,amssymb,amsthm,upgreek,dsfont}
\usepackage{colortbl,color,xcolor} 
\usepackage{graphicx,tikz,pgfplots} 
\usepackage{enumitem}
\pgfplotsset{compat=1.6}
\usepackage{subfigure,float,longtable}
\usepackage[bookmarks=true,bookmarksnumbered,plainpages=false,linktocpage,colorlinks=true,citecolor=green!80!black,linkcolor=red!70!black,filecolor=magenta,urlcolor=magenta,hidelinks,breaklinks,pdfauthor={Andreas Horst, Thomas Jahn, Felix Voigtlaender},pdftitle={}]{hyperref}
\usepackage[nameinlink,capitalize,noabbrev]{cleveref}
\usepackage[babel]{csquotes}
\usepackage{todonotes}

\usetikzlibrary{calc,intersections,through,backgrounds,arrows,patterns,shapes.geometric,shapes.misc,spy,3d,cd}

\newcommand{\ii}{\mathrm{i}}
\newcommand{\ee}{\mathrm{e}}
\newcommand{\dd}{\mathrm{d}}
\newcommand{\TT}{\mathds{T}} 
\newcommand{\NN}{\mathds{N}}                                     
\newcommand{\RR}{\mathds{R}}                                     
\newcommand{\ZZ}{\mathds{Z}}                                     

\newcommand{\ex}{\:\exists\:}                                    
\newcommand{\abs}[1]{\left\lvert#1\right\rvert}                  
\newcommand{\abss}[1]{\lvert#1\rvert}                            
\newcommand{\mnorm}[1]{\left\lVert#1\right\rVert}                
\newcommand{\setn}[1]{\left\{#1\right\}}                         
\newcommand{\setcond}[2]{\left\{#1 \::\: #2\right\}}  
\newcommand{\defeq}{\mathrel{\mathop:}=}                         
\newcommand{\lr}[1]{\!\left(#1\right)}                           
\newcommand{\p}{\partial}                                        

\newcommand{\skpr}[2]{\left\langle#1 \,,\, #2\right\rangle}               
\newcommand{\eqdef}{=\mathrel{\mathop:}}     
\newcommand{\floor}[1]{\left\lfloor #1\right\rfloor} 
\newcommand{\ceil}[1]{\left\lceil #1\right\rceil}
\newcommand{\priortwo}{Besov--Bernoulli} 
\let \eps \varepsilon
\let \piup \uppi
\newcommand{\one}{\mathds{1}}
\newcommand{\lf}{\gamma}
\newcommand{\hf}{\beta}

\newcommand{\CalS}{\mathcal{S}}
\newcommand{\EE}{\mathds{E}}
\newcommand{\PP}{\mathds{P}}
\newcommand{\VV}{\mathds{V}}

\DeclareMathOperator{\erf}{erf}
\DeclareMathOperator{\esssup}{ess\,sup}

\definecolor{pyblue}{HTML}{1f77b4}

\theoremstyle{plain} 
\newtheorem{Satz}{Theorem}[section]

\newtheorem{Lem}[Satz]{Lemma}
\newtheorem{Prop}[Satz]{Proposition}

\theoremstyle{definition} 
\newtheorem{Def}[Satz]{Definition}

\newtheorem{Bem}[Satz]{Remark}

\crefname{Satz}{Theorem}{Theorems}
\crefname{Prop}{Proposition}{Propositions}
\crefname{Lem}{Lemma}{Lemmas}
\crefname{Kor}{Corollary}{Corollaries}
\crefname{Bem}{Remark}{Remarks}
\crefname{Bsp}{Example}{Examples}
\crefname{Def}{Definition}{Definitions}

\numberwithin{equation}{section}

\makeatletter
\renewcommand*{\eqref}[1]{%
  \hyperref[{#1}]{\textup{\tagform@{\ref*{#1}}}}%
}
\makeatother

\makeatletter
\@ifundefined{textcommabelow}{%
  \DeclareTextCommandDefault\textcommabelow[1]
    {\hmode@bgroup\ooalign{\null#1\crcr\hidewidth\raise-.31ex
     \hbox{\check@mathfonts\fontsize\ssf@size\z@
     \math@fontsfalse\selectfont,}\hidewidth}\egroup}%
}{}
\makeatother

\begin{document}
\allowdisplaybreaks

\title{Besov regularity of random wavelet series}


\author{Andreas Horst}
\address[A. Horst]{Technical University of Denmark, Department of Applied Mathematics and Computer Science, Matematiktorvet 303B, 2800 Kgs. Lyngby, Denmark}
\email{ahor@dtu.dk}
\thanks{}


\author{Thomas Jahn}
\address[T. Jahn]{Katholische Universität Eichstätt--Ingolstadt, Mathematical Institute for Machine Learning and Data Science (MIDS), Auf der Schanz 49, 85049 Ingolstadt, Germany}
\email{thomas.jahn@ku.de}
\thanks{}

\author{Felix Voigtlaender}
\address[F. Voigtlaender]{Katholische Universität Eichstätt--Ingolstadt, Mathematical Institute for Machine Learning and Data Science (MIDS), Auf der Schanz 49, 85049 Ingolstadt, Germany}
\email{felix.voigtlaender@ku.de}
\thanks{}




\date{}

\begin{abstract}
We study the Besov regularity of wavelet series on $\RR^d$ with randomly chosen coefficients.
More precisely, each coefficient is a product of a random factor and a parameterized deterministic factor (decaying with the scale $j$ and the norm of the shift $m$).
Compared to the literature, we impose relatively mild conditions on the moments of the random variables in order to characterize the almost sure convergence of the wavelet series in Besov spaces $B^s_{p,q}(\RR^d)$ and the finiteness of the moments as well as of the moment generating function of the Besov norm.
In most cases, we achieve a complete characterization, i.e., the derived conditions are both necessary and sufficient.
\end{abstract}

\maketitle


\section{Introduction}

We study the Besov regularity of functions $f:\RR^d\to\RR$ (or tempered distributions $f\in \CalS^\prime(\RR^d)$) given by Karhunen--Loève-type expansions of the form
\begin{equation}\label{eq:expansion-0}
f=\sum_{j,t,m}2^{j\alpha}(j+1)^\theta \lr{1+\frac{\mnorm{m}_\infty}{2^j}}^{\hf}\xi_{j,t,m}\Psi_{j,t,m}
\end{equation}
and
\begin{equation}\label{eq:expansion-0-bernoulli}
f=\sum_{j,t,m}2^{j\alpha} \lr{1+\frac{\mnorm{m}_\infty}{2^j}}^{\hf}\lambda_{j,t,m}\xi_{j,t,m}\Psi_{j,t,m}
\end{equation}
with respect to a sufficiently \enquote{nice} compactly supported orthonormal wavelet basis $(\Psi_{j,t,m})_{j,t,m}$ of $L_2(\RR^d)$.
The functions given by \cref{eq:expansion-0,eq:expansion-0-bernoulli} are deemed \emph{random} as the factors $\xi_{j,t,m}$ are chosen to be independent and identically distributed random variables, and the factors $\lambda_{j,t,m}$ are Bernoulli-distributed random variables with success probabilites depending on $j$ and $m$.
The functions in \eqref{eq:expansion-0-bernoulli} can be seen as a model for random \emph{sparse} wavelet series. (A similar approach for periodic wavelet series has been investigated in \cite{Abramovich1998}.)
In addition to purely mathematical interest, studying such random wavelet series is motivated by their use to define prior distributions in Bayesian inversion \cite{LassasSaSi2009,DashtiStuart2012}; see also \cref{chap:related-work}.

Our first main result is to show that under mild conditions on the moments of the random variables $\xi_{j,t,m}$, the series \eqref{eq:expansion-0} almost surely belongs to the Besov space $B^s_{p,q}(\RR^d)$ \emph{if and only if} the parameters $s$, $p$, $q$ of the Besov space and the parameters $\alpha$, $\hf$, and $\theta$ of the prior \eqref{eq:expansion-0} satisfy
 \begin{equation}\label{eq:property-a-0-1}
\hf<-\frac{d}{p},\qquad s+\frac{d}{2}+\alpha<0
\end{equation}
or 
 \begin{equation}\label{eq:property-a-0-2}
\hf<-\frac{d}{p},\qquad s+\frac{d}{2}+\alpha=0,\qquad \theta  <-\frac{1}{q},
\end{equation}
see \cref{thm:target-statement-sufficiency,thm:target-statement-necessity,thm:target-statement-necessity-functions}.
(When $p=\infty$, the condition $\hf<-\frac{d}{p}$ has to be replaced by $\hf\leq 0$.
When $q=\infty$, the condition $\theta  <-\frac{1}{q}$ has to be replaced by $\theta\leq 0$.)
A similar analysis for the series in \cref{eq:expansion-0-bernoulli} is given in \cref{thm:powers-bernoulli,thm:necessity-bernoulli,thm:necessity-bernoulli-functions}.
Furthermore, given $r>0$ and assuming that \eqref{eq:property-a-0-1} or \eqref{eq:property-a-0-2} holds, we show that under (different) mild conditions on the common distribution of the random variables $\xi_{j,t,m}$, the moments $\EE[\mnorm{f}_{B^s_{p,q}(\RR^d)}^r]$ and the moment generating function $\EE[\exp(c\mnorm{f}_{B^s_{p,q}(\RR^d)}^r)]$ are finite (for some $c>0$); see \cref{thm:powers,thm:mgf}.
Thus, somewhat surprisingly, the condition that guarantees that $\mnorm{f}_{B^s_{p,q}(\RR^d)}$ is finite almost surely already implies the finiteness of $\EE[\mnorm{f}_{B^s_{p,q}(\RR^d)}^r]$ or of $\EE[\exp(c\mnorm{f}_{B^s_{p,q}(\RR^d)}^r)]$, if only the common distribution of the random variables $\xi_{j,t,m}$ is sufficiently \enquote{nice}.

A similar analysis for the moments $\EE[\mnorm{f}_{B^s_{p,q}(\RR^d)}^r]$ of the series in \cref{eq:expansion-0-bernoulli} is given in \cref{thm:powers-bernoulli}.
Recall that $c\mnorm{f}_{B^s_{p,q}(\RR^d)}^r\leq \exp(c\mnorm{f}_{B^s_{p,q}(\RR^d)}^r)$, so the condition $\EE[\exp(c\mnorm{f}_{B^s_{p,q}(\RR^d)}^r)]<\infty$ implies $\EE[\mnorm{f}_{B^s_{p,q}(\RR^d)}^r]<\infty$ which in turn implies $\mnorm{f}_{B^s_{p,q}(\RR^d)}\stackrel{a.s.}{<}\infty$.
Therefore, necessary conditions for the latter are also necessary for the former.
In other words, if $f$ is of the form \eqref{eq:expansion-0} and the distribution of the random variables $\xi_{j,t,m}$ is sufficiently well concentrated, we achieve a \emph{complete characterization} in terms of the parameters $d$, $s$, $p$, $q$, and $\alpha$, $\hf$, $\theta$ for each of the following conditions
\begin{itemize}[leftmargin=*,align=left,noitemsep]
\item{$f\in B^s_{p,q}(\RR^d)$ almost surely,}
\item{$\EE[\mnorm{f}_{B^s_{p,q}(\RR^d)}^r]<\infty$,}
\item{$\EE[\exp(c\mnorm{f}_{B^s_{p,q}(\RR^d)}^r)]<\infty$ for some $c>0$.}
\end{itemize}

\subsection{Motivation and related work}\label{chap:related-work}
While our investigations in the present paper are of theoretical nature, they are motivated by the Bayesian approach to inverse problems, where random Fourier series and random wavelet series are used as prior distributions for a stochastic interpretation of inverse problems.
There, the finiteness of the exponential moments plays a role in the study of discretization invariance and well-posedness of the Bayesian inverse problems; see \cite{LassasSaSi2009,DashtiStuart2012}.
Lassas, Saksman, and Siltanen prove in \cite[Lemma~2]{LassasSaSi2009} that for $s\in\RR$ and $p\in [1,\infty)$, the series $f(x)=\sum_{\ell=1}^\infty\ell^{-\frac{s}{d}+\frac{1}{2}-\frac{1}{p}}\xi_\ell\Psi_\ell(x)$ almost surely belongs to $B^{\tilde{s}}_{p,p}(\TT^d)$ iff $\EE[\exp(\frac{1}{2}\mnorm{f}_{B^{\tilde{s}}_{p,p}(\TT^d)}^p)]<\infty$ iff $\tilde{s}<s-\frac{d}{p}$.
Here $(\Psi_j)_{j\in\NN}$ denotes an enumeration (by increasing scale) of a compactly supported orthonormal wavelet basis of $L_2(\TT^d)$ on the torus $\TT^d$, and $(\xi_j)_{j\in\NN}$ denotes a family of i.i.d.\ real-valued random variables with probability density function proportional to $\exp(-\abs{x}^p)$.
The series $f=\sum_{\ell=1}^\infty \ell^{-\frac{s}{d}+\frac{1}{2}-\frac{1}{p}}\xi_\ell\psi_\ell$ induces a probability measure on $B^{\tilde{s}}_{p,p}(\TT^d)$ and is referred to as a \emph{Besov prior} in the inverse problem community.

The Besov prior has been extended by introducing an additional random coefficient into the wavelet expansion of $f$.
For example, the authors of \cite{KekkonenLaSaSi2023} choose the indices of the non-zero coefficients in the wavelet expansion according to a proper random tree generated by a Galton--Watson process, and in \cite{lan2024}, time-dependent random coefficients from a so called $q$-exponential process are investigated.
In \cite[Theorem~3.3]{KekkonenLaSaSi2023} and \cite[Theorem~4.2]{lan2024}, these authors prove that their priors almost surely belong to some Besov space depending on newly introduced parameters.

Apart from the applications in Bayesian inversion, wavelet series with randomly chosen coefficients have been studied in the literature for more than two decades now.
Aubry and Jaffard \cite{AubryJa2002} study the \emph{local} Hölder regularity of random wavelet series.
Because of the local nature of their analysis, the authors focus on the one-dimensional $1$-periodic case, i.e., functions $f:\TT\to\RR$ defined on the torus.
Besov regularity enters the discussion in \cite{AubryJa2002} through the \emph{scaling function} $\eta_f(p)\defeq\sup\setcond{s}{f\in B^{s/p}_{p,\infty}(\TT)}$, a function which is shown in \cite{AubryJa2002} to be closely related to histograms of wavelet coefficients at all scales.
Note that different from the setting of \cite{LassasSaSi2009}, the parameter $p$ in \cite{AubryJa2002} is an element of $(0,\infty)$.

Abramovich, Sapatinas, and Silverman \cite{Abramovich1998} investigate sparse random wavelet expansions on the interval $[0,1]$.
Specifically, the wavelet coefficients $d_{j,k}$ are chosen at random according to $d_{j,k}\sim (1-\pi_j)\delta_0+ \pi_j \mathcal{N}(0,\tau_j^2)$.
Here, $\pi_j\in [0,1]$ and $\tau_j>0$ are numbers, $\mathcal{N}(\mu,\sigma^2)$  is the normal distribution with mean $\mu$ and variance $\sigma^2$, and $\delta_0$ is the Dirac measure at $0$.
Special attention is paid to the case where $\pi_j=\min\setn{1,2^{-j\beta}C_2}$  and $\tau_j^2=2^{-j\alpha}C_1$ for some absolute constants $C_1$ and $C_2$, and parameters $\alpha$ and $\beta$, which are then subject to the characterization of the almost sure membership of the corresponding wavelet series in the Besov space $B^s_{p,q}([0,1])$, see \cite[Theorem~1]{Abramovich1998}.
Bochkina \cite{Bochkina2013} extends this setting to more general distributions of the random variables.

For Besov spaces $B^s_{p,q}(Q)$ with parameters $s\in\RR$ and $p,q\in (0,\infty)$ on domains $Q\subset\RR^d$ which admit an orthonormal basis of compactly supported wavelets, the authors of \cite{Cioica2014} choose products of Bernoulli variables (with success probabilities decaying exponentially in $j$) and Gaussians for the coefficients of a wavelet series.
In \cite[Theorem~5.4.7]{Cioica2014}, a characterization of the almost sure membership of such a wavelet series $g$ in $B^s_{p,q}(Q)$ is given in terms of inequalities involving $d$, $s$, $p$, $q$, and the parameters governing the random variables.
Furthermore, the same inequalities are shown to be sufficient for $\EE[\mnorm{g}_{B^s_{p,q}(Q)}^q]<\infty$.

Another account on random wavelet series on $\RR^d$ is given by Aziznejad and Fageot \cite[Proposition~3.4]{AziznejadFa2020}.
There it is shown that if the wavelet coefficients of $f$ are i.i.d.\ standard Gaussian random variables (which corresponds to $\alpha=\beta=\theta=0$ and $\xi_{j,t,m}\stackrel{i.i.d.}{\sim}\mathcal{N}(0,1)$ in \eqref{eq:expansion-0}) and $p\in (0,\infty]$, then $f$ almost surely belongs to a weighted Besov space $B^s_{p,p}(\RR^d,w_\sigma)$ if $s<-\frac{d}{2}$ and $\sigma<-\frac{d}{p}$, and almost surely does not belong to $B^s_{p,p}(\RR^d,w_\sigma)$ if $s\geq -\frac{d}{2}$ or $\sigma\geq -\frac{d}{p}$; see \cref{remark:weights} for a definition of $w_\sigma$.

Esser, Jaffard, and Vedel \cite{EsserJaVe2023} study regularity properties of wavelet series through randomization.
More precisely, for wavelet bases on the $d$-dimensional torus $\TT^d$ and on $\RR^d$, they compare the Hölder continuity and boundedness properties of a wavelet series before and after multiplying every coefficient with a member of a family of i.i.d.\ random variables. 
Their analysis shows that regarding the conditions on the random variables, there is an essential distinction between bounded and unbounded random variables.

In the paper at hand, we also provide necessary and sufficient conditions for random wavelet series to almost surely belong to a given Besov space (or for its Besov quasi-norm to have finite moments and finite moment-generating function.)
But, different from the existing literature, we simultaneously allow the the following conditions: (i)  we allow much more general distributions of the random coefficients, only imposing mild conditions such as finiteness of certain moments or of the moment generating function of certain powers of $\abs{\xi_{j,t,m}}$, (ii) we allow Besov spaces $B^s_{p,q}$ with $p\neq q$, (iii) we work on $\RR^d$ instead of $\TT^d$, and (iv) we consider the full range of exponents $p,q\in (0,\infty]$.

\subsection{Structure of the paper}
\cref{chap:preliminaries} is a review of the necessary preliminaries on Besov spaces and their characterization by wavelets.

In \cref{chap:sufficiency}, we show that under suitable conditions on the random variables $\xi_{j,t,m}$, the conditions \eqref{eq:property-a-0-1} and \eqref{eq:property-a-0-2} are \emph{sufficient} for having $f\in B^s_{p,q}(\RR^d)$ almost surely and $\EE[\mnorm{f}_{B^s_{p,q}(\RR^d)}^r]<\infty$ when $f$ is of the form \eqref{eq:expansion-0}.
Additionally, we show the sufficiency of an analog of \eqref{eq:property-a-0-1} and \eqref{eq:property-a-0-2} for having $\EE[\mnorm{f}_{B^s_{p,q}(\RR^d)}^r]<\infty$ for the series \eqref{eq:expansion-0-bernoulli}.
\cref{chap:sufficiency} is concluded with the proof of the sufficiency of \eqref{eq:property-a-0-1} and \eqref{eq:property-a-0-2} for having $\EE[\exp(c\mnorm{f}_{B^s_{p,q}(\RR^d)}^r)]<\infty$ for $c,r>0$ when $f$ is of the form \eqref{eq:expansion-0}.

In \cref{chap:necessity}, we show that \eqref{eq:property-a-0-1} and \eqref{eq:property-a-0-2} are \emph{necessary} for the series in \eqref{eq:expansion-0} to almost surely belong to the Besov space $B^s_{p,q}(\RR^d)$, given a  condition on the moments of the random variables $\xi_{j,t,m}$ in \eqref{eq:expansion-0}.
Additionally, we show the necessity of an analog of \eqref{eq:property-a-0-1} and \eqref{eq:property-a-0-2} for the series in \eqref{eq:expansion-0-bernoulli} to almost surely belong to the Besov space $B^s_{p,q}(\RR^d)$.

In \cref{chap:discussion}, we discuss the necessity of the conditions on the random variables in the theorems about the series \eqref{eq:expansion-0}.

Finally, \cref{chap:appendix} contains auxiliary statements and technical proofs.

\section{Preliminaries and notation}\label{chap:preliminaries}
\subsection{General notation}

For two sequences $(a_n)_{n\in\NN}$ and $(b_n)_{n\in\NN}$ of non-negative real numbers, we write $a_n\lesssim b_n$ if there exists a constant $C>0$ such that $a_n\leq Cb_n$ for all $n\in\NN$.
We write $a_n\gtrsim b_n$ if $b_n\lesssim a_n$, and $a_n\asymp b_n$ if both $a_n\lesssim b_n$ and $b_n\lesssim a_n$.
Unless mentioned explicitly, the symbol $\log$ refers to the natural logarithm.

For an at most countable set $I\neq\emptyset$, $x\in\RR^I$, and $p\in (0,\infty]$, we write
\begin{equation*}
\mnorm{x}_{\ell_p(I)}\defeq\begin{cases}\lr{\sum_{i\in I}\abs{x_i}^p}^{\frac{1}{p}}&\text{if }0<p<\infty,\\\sup_{i\in I}\abs{x_i}&\text{if }p=\infty.\end{cases}
\end{equation*}
We shall usually omit $I$ from the notation, as it will be clear from the context.
When $I=\setn{1,\ldots,d}$, we write $\mnorm{x}_p\defeq\mnorm{x}_{\ell_p(I)}$.

For a (measurable) function $f:\RR^d\to\RR$, we write
\begin{equation*}
\mnorm{f}_{L_p(\RR^d)}\defeq\begin{cases}\lr{\int_{\RR^d}\abs{f(x)}^p\dd x }^{\frac{1}{p}}&\text{if }0<p<\infty,\\\esssup_{x\in\RR^d}\abs{f(x)}&\text{if }p=\infty.\end{cases}
\end{equation*}
For $\gamma=(\gamma_1,\ldots,\gamma_d)\in\NN_0^d$ and sufficiently smooth $f:\RR^d\to\RR$, we denote the partial derivatives of $f$ by $D^\gamma f=\frac{\p^{\gamma_1+\ldots+\gamma_d}}{\p x_1^{\gamma_1}\ldots \p x_d^{\gamma_d}}f$.

\subsection{Wavelets and Besov spaces}
Our discussion of wavelet series and their Besov regularity relies on the wavelet characterization of Besov spaces.
He present a brief outline of this characterization here, where we follow \cite[Chapter~1]{Triebel2008} for the definitions of the Fourier transform and (weighted) Besov (sequence) spaces.
We will not go into the minute details in what sense the occurring integrals are well-defined, how the Schwartz space $\CalS(\RR^d)$ and its dual space $\CalS^\prime(\RR^d)$ are defined, and in which way the Fourier transform is extended from $\CalS(\RR^d)$ to $\CalS^\prime(\RR^d)$, see \cite[Section~2.3]{Triebel1997}.
Instead, our aim is to fix the notation and to avoid ambiguities regarding different conventions in the literature.

The \emph{Fourier transform} $\hat{f}$ of $f\in \CalS(\RR^d)$ is defined as
\begin{equation*}
\hat{f}(x)\defeq (2\piup)^{-\frac{d}{2}} \int_{\RR^d}\ee^{-\ii x^\top y}f(y)\dd y
\end{equation*}
for $x\in\RR^d$.
Here $x^\top y$ denotes the standard inner product of $x,y\in\RR^d$.
The inverse of the Fourier transform shall be denoted by $\check{f}$ or $f^\vee$; it is given by $\check{f}(x)=\hat{f}(-x)$.
\begin{Def}\label{def:fourier}
Let $d\in \NN$, $p,q\in (0,\infty]$, $s \in\RR$.

An \emph{admissible weight} \cite[Definition~1.22]{Triebel2008} is a function $w\in C^{\infty}(\RR^d,\RR)$ such that for every multiindex $\gamma\in\NN_0^d$, there exists a constant $c_\gamma>0$ such that the partial derivative $D^\gamma w$ of $w$ satisfies
\begin{equation*}
\abs{D^\gamma w(x)}\leq c_\gamma w(x)\qquad \text{for all}\qquad x\in\RR^d,
\end{equation*}
and there exist $c>0$ and $\alpha\geq 0$ such that
\begin{equation*}
0<w(x)\leq cw(y)(1+\mnorm{x-y}_2^2)^{\frac{\alpha}{2}} \qquad \text{for all}\qquad x,y\in\RR^d.
\end{equation*}

If $(\varphi_j)_{j\in\NN_0}$ is a suitable resolution of unity, cf.~\cite[(1.5), (1.6), (1.7)]{Triebel2008}, then the \emph{(weighted) Besov space} $B^s_{p,q}(\RR^d,w)$ is defined as the collection of all $f\in \CalS^\prime(\RR^d)$ such that
\begin{equation*}
\mnorm{f}_{B^s_{p,q}(\RR^d,w)}\defeq\lr{\sum_{j\in\NN_0} 2^{jsq}\mnorm{w\cdot(\varphi_j\hat{f})^\vee}_{L_p(\RR^d)}^q}^{\frac{1}{q}}<\infty,
\end{equation*}
with the usual modifications for $q=\infty$, see \cite[Definition~1.22]{Triebel2008}.
(This definition makes sense since $\varphi_j\hat{f}\in\CalS^\prime(\RR^d)$ is a compactly supported distribution so that $(\varphi_j\hat{f})^\vee\in C^\infty(\RR^d)$ is a (smooth) function by the Paley--Wiener theorem.) 

Define $T_0\defeq \setn{F}^d$, $T_j\defeq\setn{F,M}^d\setminus T_0$ for $j\in\NN$, 
\begin{equation*}
J\defeq\setcond{(j,t,m)}{j\in\NN_0,t\in T_j,m\in\ZZ^d},
\end{equation*}
 and for a sequence $a=(a_{j,t,m})_{(j,t,m)\in J}\in \RR^J$, following \cite[Definition~1.24]{Triebel2008} with a slight modification, we let
\begin{equation*}
\mnorm{a}_{b^s_{p,q}(\RR^d,w)}\defeq \mnorm{\lr{2^{j(s+\frac{d}{2}-\frac{d}{p})}\mnorm{\lr{w(2^{-(j-\delta_{j\neq 0})}m)\abs{a_{j,t,m}}}_{m\in\ZZ^d}}_{\ell_p}}_{j\in\NN_0,t\in T_j}}_{\ell_q},
\end{equation*}
where $\delta_{j\neq 0}=1$ if $j\neq 0$ and $\delta_{j\neq 0}=0$ if $j=0$.
The \emph{(weighted) Besov sequence space} $b^s_{p,q}(\RR^d,w)$ is defined as the collection of all sequences $(a_{j,t,m})_{(j,t,m)\in J}\in\RR^J$ such that $\mnorm{a}_{b^s_{p,q}(\RR^d,w)}<\infty$.
We write $B^s_{p,q}(\RR^d)\defeq B^s_{p,q}(\RR^d,w)$ and  $b^s_{p,q}(\RR^d)\defeq b^s_{p,q}(\RR^d,w)$ when $w(x)=1$ for all $x\in\RR^d$.
\end{Def}
It turns out that $(B^s_{p,q}(\RR^d,w),\mnorm{\bullet}_{B^s_{p,q}(\RR^d,w)})$ and $(b^s_{p,q}(\RR^d,w),\mnorm{\bullet}_{b^s_{p,q}(\RR^d),w})$ are quasi-Banach spaces, and different choices of the resolution of unity lead to equivalent quasi-norms, see \cite[Section~2.2]{HaroskeTr2004}.

The characterization of Besov spaces in terms of wavelet series is usually done using compactly supported wavelets.
Daubechies describes in \cite[Chapter~6]{Daubechies1992} the construction of a multiresolution analysis of $L_2(\RR)$ whose scaling function and wavelet function are sufficiently smooth and have sufficiently many vanishing moments, cf.\ \cite[Definition~1.49 and Remark~1.50]{Triebel2006} for the wavelet parlance.
We cite the existence result for such \enquote{nice} compactly supported wavelets from \cite[Theorem~10]{GrohsKlVo2023}.
\begin{Satz}\label{thm:compactly-supported-wavelet}
For each $k\in\NN$, there exists a multiresolution analysis $(V_j)_{j\in\ZZ}$ of $L_2(\RR)$ with scaling function $\psi_F\in L_2(\RR)$ and wavelet function $\psi_M\in L_2(\RR)$ such that the following hold:
\begin{enumerate}[label={(\roman*)},leftmargin=*,align=left,noitemsep]
\item{$\psi_F$ and $\psi_M$ are real-valued and compactly supported,}
\item{$\psi_F$, $\psi_M\in C^k(\RR)$,}
\item{$\widehat{\psi_F}(0)=(2\piup)^{-\frac{1}{2}}$,}
\item{$\int_\RR x^\ell \psi_M(x)\dd x =0$ for all $\ell\in\setn{0,\ldots,k}$.\label{moments}}
\end{enumerate}
\end{Satz}

The multiresolution analysis from \cref{thm:compactly-supported-wavelet} is then used to define an orthonormal basis $(\Psi_{j,t,m})_{(j,t,m)\in J}$ of $L_2(\RR^d)$ via 
\begin{equation}\label{eq:wavelet-tensorization}
\Psi_{j,t,m}:\RR^d\to\RR,\quad x\mapsto \begin{cases}\prod_{k=1}^d\psi_F(x_k-m_k)&\text{if }j=0,\\[6pt]
2^{(j-1)\frac{d}{2}}\prod_{k=1}^d \psi_{t_k}(2^{j-1}x_k-m_k)&\text{if }j\in\NN\end{cases}
\end{equation}
for $(j,t,m)\in J$, see \cite[Theorem~1.53]{Triebel2006}.

It turns out that mapping the coefficient sequence to the corresponding wavelet expansion constitutes an isomorphism between the Besov sequence space and the corresponding Besov function space.
This is the above-mentioned wavelet characterization of Besov spaces.
The following theorem makes this precise; it is taken from \cite[Theorem~1.26]{Triebel2008}.
We note that unlike Triebel in \cite{Triebel2008}, we use the $L_2$-normalized wavelet system $(\Psi_{j,t,m})_{(j,t,m)\in J}$ instead of the $L_\infty$-bounded wavelet system $(2^{-j\frac{d}{2}}\Psi_{j,t,m})_{(j,t,m)\in J}$ used in \cite{Triebel2008}.
This is compensated by the factor $2^{j\frac{d}{2}}$ in the definition of $b^s_{p,q}(\RR^d)$ which does not occur in \cite{Triebel2008}; see also \cref{bem:notation}.
\begin{Satz}\label{thm:wavelet-characterization}
Let $d\in \NN$, $p,q\in (0,\infty]$, $s \in\RR$.
Denote by $(\Psi_{j,t,m})_{(j,t,m)\in J}$ the orthonormal basis of $L_2(\RR^d)$ given in \eqref{eq:wavelet-tensorization}.
Assume that $k>\max\setn{s,d\lr{\frac{1}{p}-1}-s}$ if $p\in (0,1)$ and $k>\abs{s}$ else.
Let $f\in\CalS^\prime(\RR^d)$.
Then $f\in B^s_{p,q}(\RR^d,w)$ if and only if it can be represented as
\begin{equation}\label{eq:wavelet-characterization}
f=\sum_{(j,t,m)\in J} a_{j,t,m}  \Psi_{j,t,m}
\end{equation}
with $a=(a_{j,t,m})_{(j,t,m)\in J}\in b^s_{p,q}(\RR^d,w)$.
The convergence is unconditional in $\CalS^\prime(\RR^d)$.
Furthermore, the representation is unique, i.e., for given  $f\in B^s_{p,q}(\RR^d,w)$, the unique element $(a_{j,t,m})_{(j,t,m)\in J}\in b^s_{p,q}(\RR^d,w)$ such that \eqref{eq:wavelet-characterization} holds is given by $a_{j,t,m}=\skpr{f}{\Psi_{j,t,m}}$, where the duality pairing is suitably defined.
The map
\begin{equation*}
I: f\mapsto (\skpr{f}{\Psi_{j,t,m}})_{(j,t,m)\in J}
\end{equation*}
defines an isomorphism of quasi-Banach spaces $B^s_{p,q}(\RR^d,w)\to b^s_{p,q}(\RR^d,w)$, and the map
\begin{equation*}
S:  (a_{j,t,m})_{(j,t,m)\in J}\mapsto \sum_{(j,t,m)\in J}a_{j,t,m}\Psi_{j,t,m}
\end{equation*}
defines an isomorphism of quasi-Banach spaces $b^s_{p,q}(\RR^d,w)\to B^s_{p,q}(\RR^d,w)$.
\end{Satz}

\begin{Bem}\label{bem:notation}
Let us comment on the notational differences to \cite{Triebel2008}.
Triebel \cite[(1.89) and (1.90)]{Triebel2008} uses different index sets for the wavelet system, namely $G_0\defeq \setn{F,M}^d$, $G_j\defeq \setn{F,M}^d\setminus \setn{F}^d$.
The wavelet system \cite[(1.91)]{Triebel2008} then takes the form  
\begin{equation*}
\widetilde{\Psi}_{j,t,m}(x)=2^{j\frac{d}{2}}\prod_{k=1}^d \psi_{t_k}(2^j x_k-m_k).
\end{equation*}
If we let $\widetilde{J}\defeq \setcond{(j,t,m)}{j\in\NN_0,t\in G_j,m\in\ZZ^d}$, we note that Triebel's definition
\begin{equation*}
\mnorm{a}_{\widetilde{b}^s_{p,q}(\RR^d,w)}=\lr{\sum_{j\in\NN_0} 2^{jq(s-\frac{d}{p})}\sum_{t\in G_j} \lr{w(2^{-j}m)^p \abs{a_{j,t,m}}^p}^{\frac{q}{p}}}^{\frac{1}{q}}
\end{equation*}
 of the Besov sequence norm of $a=(a_{j,t,m})_{(j,t,m)\in\widetilde{J}}$ does not involve a factor of $2^{j\frac{d}{2}}$, cf.\  \cite[(1.119)]{Triebel2008}, but the wavelet characterization in \cite[Theorem~1.26]{Triebel2008}
does:
\begin{equation*}
\mnorm{f}_{B^s_{p,q}(\RR^d,w)}\asymp \mnorm{\lr{2^{j\frac{d}{2}}\skpr{f}{\widetilde{\Psi}_{j,t,m}}}_{(j,t,m)\in \widetilde{J}}}_{\widetilde{b}^s_{p,q}(\RR^d,w)}.
\end{equation*}
But note that we also have
\begin{equation*}
 \mnorm{\lr{2^{j\frac{d}{2}}\skpr{f}{\widetilde{\Psi}_{j,t,m}}}_{(j,t,m)\in \widetilde{J}}}_{\widetilde{b}^s_{p,q}(\RR^d,w)}\asymp  \mnorm{\lr{\skpr{f}{\Psi_{j,t,m}}}_{(j,t,m)\in J}}_{b^s_{p,q}(\RR^d,w)}.
\end{equation*}
\end{Bem}

Once we fix the orthonormal basis $(\Psi_{j,t,m})_{(j,t,m)\in J}$ of $L_2(\RR^d)$, we may consider random wavelet series $f$ as in \eqref{eq:expansion-0}.
In view of the application in Bayesian inverse problems as prior distribution, we say that $f$ follows a Besov prior.
Note that the following definition is slightly more general than \eqref{eq:expansion-0}, since we allow the additional parameter $\gamma$.
\begin{Def}\label{def:besov-prior}
Let $d,k\in\NN$, $\alpha,\hf,\lf,\theta\in\RR$, $\mu,\nu\in (-\infty,0]$, and let $X$ be a (real-valued) random variable.
\begin{enumerate}[label={(\alph*)},leftmargin=*,align=left,noitemsep]
\item{We say that $f$ \emph{follows a Besov prior with parameters $(\alpha,\hf,\lf,\theta,k)$ and template random variable $X$} if the orthonormal basis $(\Psi_{j,t,m})_{(j,t,m)\in J}$ of $L_2(\RR^d)$ is the one from \eqref{eq:wavelet-tensorization} (with $\psi_F,\psi_M\in C^k(\RR)$ and $\psi_M$ having $k$ vanishing moments as in \cref{thm:compactly-supported-wavelet}\label{moments}) and there are i.i.d.\ random variables $\xi_{j,t,m}\sim X$ for $(j,t,m)\in \NN_0\times \setn{F,M}^d\times \ZZ^d$ such that 
\begin{equation}\label{eq:expansion}
f=\sum_{(j,t,m)\in J}2^{j\alpha}(j+1)^\theta \lr{1+\frac{\mnorm{m}_\infty}{2^j}}^{\hf + (\lf-\hf)\one_{j=0}}\xi_{j,t,m}\Psi_{j,t,m},
\end{equation}
with unconditional convergence of the series in the sense of tempered distributions.}
\item{We say that $f$ \emph{follows a \priortwo{} prior with parameters $(\alpha,\hf,\lf,\mu,\nu,k)$ and template random variable $X$} if the orthonormal basis $(\Psi_{j,t,m})_{(j,t,m)\in J}$ of $L_2(\RR^d)$ is the one from \eqref{eq:wavelet-tensorization} (with $\psi_F,\psi_M\in C^k(\RR)$ and $\psi_M$ having $k$ vanishing moments), there are i.i.d.\ random variables $\xi_{j,t,m}\sim X$ for $(j,t,m)\in \NN_0\times \setn{F,M}^d\times \ZZ^d$, and
Bernoulli$(\varrho_{j,m})$-distributed random variables $\lambda_{j,t,m}$ for $(j,t,m)\in \NN_0\times \setn{F,M}^d\times \ZZ^d$, independent among each other and from all the $\xi_{j,t,m}$, with
\begin{equation*}
\varrho_{j,m}=2^{j\mu}\lr{1+\frac{\mnorm{m}_\infty}{2^j }}^\nu
\end{equation*}
such that
\begin{equation}\label{eq:expansion-bernoulli}
f=\sum_{(j,t,m)\in J}2^{j\alpha} \lr{1+\frac{\mnorm{m}_\infty}{2^j}}^{\hf + (\lf-\hf)\one_{j=0}}\lambda_{j,t,m}\xi_{j,t,m}\Psi_{j,t,m},
\end{equation}
with unconditional convergence of the series \eqref{eq:expansion-bernoulli} in the sense of tempered distributions.}
\end{enumerate}
Here $\one_{j=0}\defeq 1$ if $j=0$, and $\one_{j=0}\defeq 0$ else.
Furthermore, if $j=0$ for a summand in the expansion \eqref{eq:expansion}, then also $t=(F,\ldots,F)$ is fixed, and in that case we shorten notation by omitting both of them and only keep the relevant index $m\in\ZZ^d$, i.e., we write $\xi_m\defeq \xi_{0,(F,\ldots,F),m}$ and $\lambda_m\defeq\lambda_{0,(F,\ldots,F),m}$.
\end{Def}
The analysis for the functions and their coefficient sequences are not entirely the same; see the discussion before \cref{thm:target-statement-necessity-functions}.
Therefore it is appropriate to introduce analogous notions for the coefficient sequences.
The notational conventions from \cref{def:besov-prior} apply.
\begin{Def}\label{def:sequence-prior}
Let $d\in\NN$, $\alpha,\hf,\lf,\theta\in\RR$, $\mu,\nu\in (-\infty,0]$, and let $X$ be a (real-valued) random variable.
\begin{enumerate}[label={(\alph*)},leftmargin=*,align=left,noitemsep]
\item{We say that $a=(a_{j,t,m})_{(j,t,m)\in J}\in\RR^J$ \emph{follows a Besov sequence prior with parameters $(\alpha,\hf,\lf,\theta)$ and template random variable $X$} if there are i.i.d.\ random variables $\xi_{j,t,m}\sim X$ for $(j,t,m)\in \NN_0\times \setn{F,M}^d\times \ZZ^d$ such that 
\begin{equation}\label{eq:expansion-seq}
a_{j,t,m}=2^{j\alpha}(j+1)^\theta \lr{1+\frac{\mnorm{m}_\infty}{2^j}}^{\hf + (\lf-\hf)\one_{j=0}}\xi_{j,t,m}.
\end{equation}}
\item{We say that $a=(a_{j,t,m})_{(j,t,m)\in J}\in\RR^J$ \emph{follows a \priortwo{} sequence prior with parameters $(\alpha,\hf,\lf,\mu,\nu)$ and template random variable $X$} if there are i.i.d.\ random variables $\xi_{j,t,m}\sim X$ for $(j,t,m)\in \NN_0\times \setn{F,M}^d\times \ZZ^d$ and
Bernoulli$(\varrho_{j,m})$-distributed random variables $\lambda_{j,t,m}$ for $(j,t,m)\in \NN_0\times \setn{F,M}^d\times \ZZ^d$, independent among each other and from all the $\xi_{j,t,m}$, with
\begin{equation*}
\varrho_{j,m}=2^{j\mu}\lr{1+\frac{\mnorm{m}_\infty}{2^j }}^\nu
\end{equation*}
such that
\begin{equation}\label{eq:expansion-bernoulli-seq}
a_{j,t,m}=2^{j\alpha} \lr{1+\frac{\mnorm{m}_\infty}{2^j}}^{\hf + (\lf-\hf)\one_{j=0}}\lambda_{j,t,m}\xi_{j,t,m}.
\end{equation}}
\end{enumerate}
\end{Def}
Note that in \cref{def:besov-prior,def:sequence-prior}, we assume to have random variables $\xi_{j,t,m}$ and $\lambda_{j,t,m}$ for $(j,t,m)\in \NN_0\times \setn{F,M}^d\times \ZZ^d$, although only those for $(j,t,m)\in J$ are \enquote{used} in \eqref{eq:expansion}, \eqref{eq:expansion-bernoulli}, \eqref{eq:expansion-seq}, and \eqref{eq:expansion-bernoulli-seq}.
This is done for convenience in the proofs.
In our analysis, expectations $\EE$ and variances $\VV$ are taken with respect to these random variables $\xi_{j,t,m}$ and $\lambda_{j,t,m}$.

\cref{fig:besov,fig:bernoullibesov} show samples of the two priors defined in \cref{def:besov-prior} for different sets of parameters. (A Python implementation for generating these samples can be found at \url{https://github.com/AndreasHorst/RWS}.)
In both figures, we used standard normal variables for $\xi_{j,tm}$ and Daubechies-$10$ wavelets for $\Psi_{j,t,m}$.
We truncated the basis expansions at $j=10$.
Note that Daubechies shows in \cite[p.~226]{Daubechies1992} that this wavelet system has its scaling and wavelet function in $C^k$ for all $k<2.406$ and, more generally, that the Daubechies-$N$ wavelet system has their scaling and wavelet function in $C^{\mu N}$ with $\mu\approx 0.2075$ for sufficiently large $N$.

\begin{figure}[h!]
    \centering
\begin{tikzpicture}
\begin{scope}
\begin{axis}[height=5cm,width=4.5cm,xmin=-25,xmax=25,xtick={-20,0,20},ytick={-1,-0.5,0,0.5,1},ymin=-1.2,ymax=1.4,xlabel={$x$},title={$\gamma=\beta=-1$, $\alpha=-1$},label style={font=\tiny},tick label style={font=\tiny},title style={yshift=-1.5ex,font=\footnotesize}]
\addplot[line join=round,pyblue,mark=none] table [col sep=comma] {besov00.csv};
\end{axis}
\end{scope}

\begin{scope}[xshift=4.05cm]
\begin{axis}[height=5cm,width=4.5cm,xmin=-25,xmax=25,xtick={-20,0,20},ytick={-1,-0.5,0,0.5,1},ymin=-1.2,ymax=1.4,xlabel={$x$},title={$\gamma=\beta=-2$, $\alpha=-1$},label style={font=\tiny},tick label style={font=\tiny},title style={yshift=-1.5ex,font=\footnotesize}]
\addplot[line join=round,pyblue,mark=none] table [col sep=comma] {besov01.csv};
\end{axis}
\end{scope}

\begin{scope}[xshift=8.1cm]
\begin{axis}[height=5cm,width=4.5cm,xmin=-25,xmax=25,xtick={-20,0,20},ytick={-1,-0.5,0,0.5,1},ymin=-1.2,ymax=1.4,xlabel={$x$},title={$\gamma=\beta=-4$, $\alpha=-1$},label style={font=\tiny},tick label style={font=\tiny},title style={yshift=-1.5ex,font=\footnotesize}]
\addplot[line join=round,pyblue,mark=none] table [col sep=comma] {besov02.csv};
\end{axis}
\end{scope}

\begin{scope}[yshift=-4.8cm]
\begin{axis}[height=5cm,width=4.5cm,xmin=-25,xmax=25,xtick={-20,0,20},ytick={-1,-0.5,0,0.5,1},ymin=-1.5,ymax=1.5,xlabel={$x$},title={$\gamma=\beta=-1$, $\alpha=-1$},label style={font=\tiny},tick label style={font=\tiny},title style={yshift=-1.5ex,font=\footnotesize}]
\addplot[line join=round,pyblue,mark=none] table [col sep=comma] {besov10.csv};
\end{axis}
\end{scope}

\begin{scope}[xshift=4.05cm,yshift=-4.8cm]
\begin{axis}[height=5cm,width=4.5cm,xmin=-25,xmax=25,xtick={-20,0,20},ytick={-1,-0.5,0,0.5,1},ymin=-1.5,ymax=1.5,xlabel={$x$},title={$\gamma=\beta=-1$, $\alpha=-2$},label style={font=\tiny},tick label style={font=\tiny},title style={yshift=-1.5ex,font=\footnotesize}]
\addplot[line join=round,pyblue,mark=none] table [col sep=comma] {besov11.csv};
\end{axis}
\end{scope}

\begin{scope}[xshift=8.1cm,yshift=-4.8cm]
\begin{axis}[height=5cm,width=4.5cm,xmin=-25,xmax=25,xtick={-20,0,20},ytick={-1,-0.5,0,0.5,1},ymin=-1.5,ymax=1.5,xlabel={$x$},title={$\gamma=\beta=-1$, $\alpha=-4$},label style={font=\tiny},tick label style={font=\tiny},title style={yshift=-1.5ex,font=\footnotesize}]
\addplot[line join=round,pyblue,mark=none] table [col sep=comma] {besov12.csv};
\end{axis}
\end{scope}
\end{tikzpicture}
\caption{Samples of our Besov prior for different sets of parameters. The parameters $\beta$ and $\gamma$ impact the spatial decay of the function, whereas the smoothness changes with $\alpha$. In all panels, we set $\theta=0$.\label{fig:besov}}
\end{figure}
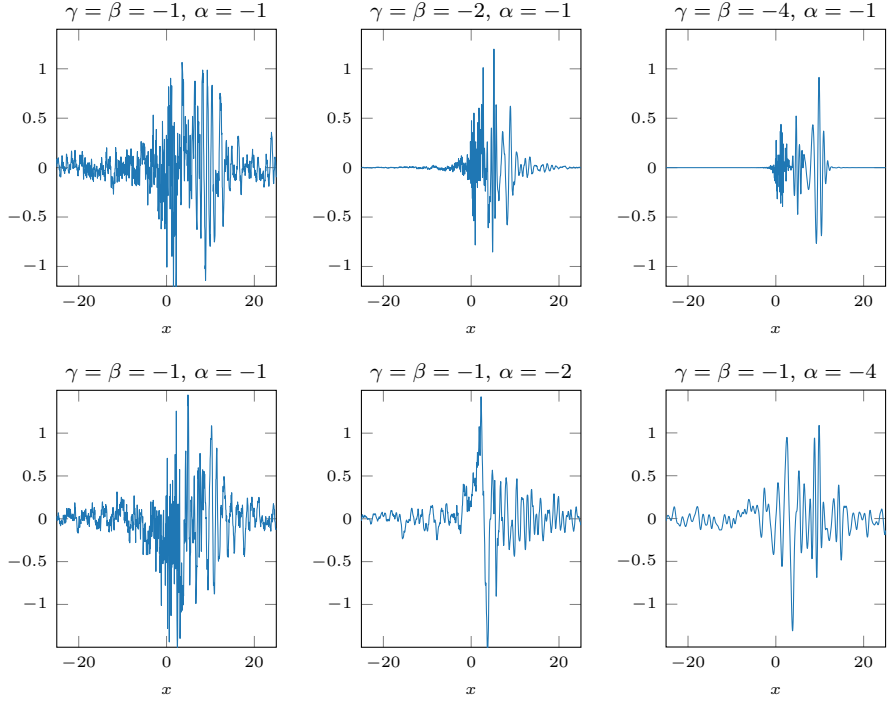

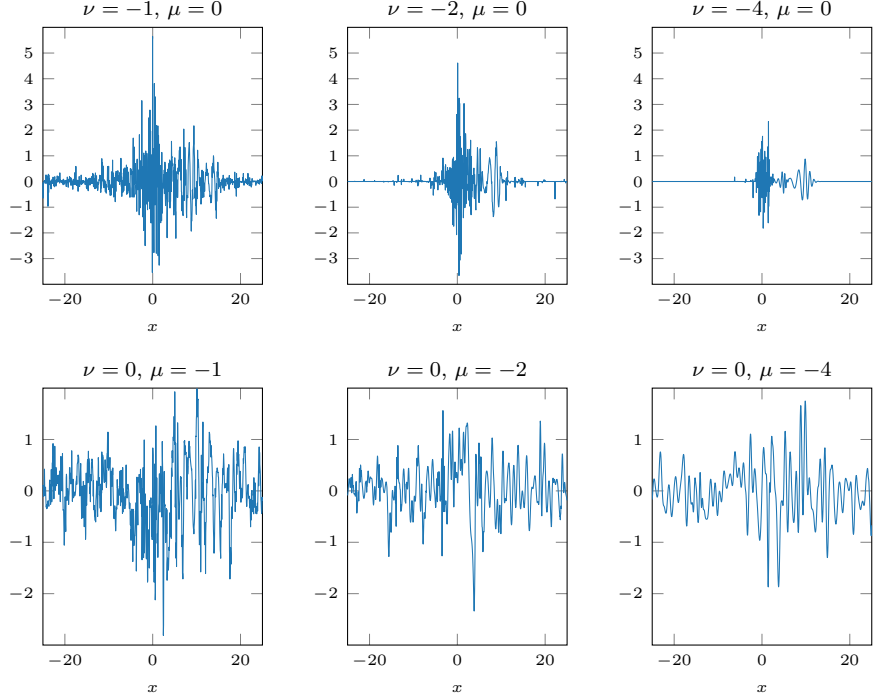
\begin{figure}[h!]
    \centering
\begin{tikzpicture}
\begin{scope}
\begin{axis}[height=5cm,width=4.5cm,xmin=-25,xmax=25,xtick={-20,0,20},ytick={-3,-2,-1,0,1,2,3,4,5},ymin=-4,ymax=6,xlabel={$x$},title={$\nu=-1$, $\mu=0$},label style={font=\tiny},tick label style={font=\tiny},title style={yshift=-1.5ex,font=\footnotesize}]
\addplot[line join=round,pyblue,mark=none] table [col sep=comma] {bernoulli00.csv};
\end{axis}
\end{scope}

\begin{scope}[xshift=4.05cm]
\begin{axis}[height=5cm,width=4.5cm,xmin=-25,xmax=25,xtick={-20,0,20},ytick={-3,-2,-1,0,1,2,3,4,5},ymin=-4,ymax=6,xlabel={$x$},title={$\nu=-2$, $\mu=0$},label style={font=\tiny},tick label style={font=\tiny},title style={yshift=-1.5ex,font=\footnotesize}]
\addplot[line join=round,pyblue,mark=none] table [col sep=comma] {bernoulli01.csv};
\end{axis}
\end{scope}

\begin{scope}[xshift=8.1cm]
\begin{axis}[height=5cm,width=4.5cm,xmin=-25,xmax=25,xtick={-20,0,20},ytick={-3,-2,-1,0,1,2,3,4,5},ymin=-4,ymax=6,xlabel={$x$},title={$\nu=-4$, $\mu=0$},label style={font=\tiny},tick label style={font=\tiny},title style={yshift=-1.5ex,font=\footnotesize}]
\addplot[line join=round,pyblue,mark=none] table [col sep=comma] {bernoulli02.csv};
\end{axis}
\end{scope}

\begin{scope}[yshift=-4.8cm]
\begin{axis}[height=5cm,width=4.5cm,xmin=-25,xmax=25,xtick={-20,0,20},ytick={-2,-1,0,1},ymin=-3,ymax=2,xlabel={$x$},title={$\nu=0$, $\mu=-1$},label style={font=\tiny},tick label style={font=\tiny},title style={yshift=-1.5ex,font=\footnotesize}]
\addplot[line join=round,pyblue,mark=none] table [col sep=comma] {bernoulli10.csv};
\end{axis}
\end{scope}

\begin{scope}[xshift=4.05cm,yshift=-4.8cm]
\begin{axis}[height=5cm,width=4.5cm,xmin=-25,xmax=25,xtick={-20,0,20},ytick={-2,-1,0,1},ymin=-3,ymax=2,xlabel={$x$},title={$\nu=0$, $\mu=-2$},label style={font=\tiny},tick label style={font=\tiny},title style={yshift=-1.5ex,font=\footnotesize}]
\addplot[line join=round,pyblue,mark=none] table [col sep=comma] {bernoulli11.csv};
\end{axis}
\end{scope}

\begin{scope}[xshift=8.1cm,yshift=-4.8cm]
\begin{axis}[height=5cm,width=4.5cm,xmin=-25,xmax=25,xtick={-20,0,20},ytick={-2,-1,0,1},ymin=-3,ymax=2,xlabel={$x$},title={$\nu=0$, $\mu=-4$},label style={font=\tiny},tick label style={font=\tiny},title style={yshift=-1.5ex,font=\footnotesize}]
\addplot[line join=round,pyblue,mark=none] table [col sep=comma] {bernoulli12.csv};
\end{axis}
\end{scope}
\end{tikzpicture}
\caption{Samples of our \priortwo{} prior for different sets of parameters. The parameter $\nu$ impacts the spatial decay, whereas the smoothness changes with $\mu$. In all panels, we set $\alpha=\gamma=\beta=-0.5$.\label{fig:bernoullibesov}}
\end{figure}

\begin{Bem}\label{remark:weights}
In the sequel, we will work with unweighted Besov spaces $B^s_{p,q}(\RR^d)$ and $b^s_{p,q}(\RR^d)$.
However, the entire analysis can be readily transferred to weighted Besov spaces with weight function $w_\sigma:\RR^d\to\RR$, $w_\sigma(x)\defeq (1+\mnorm{x}_2^2)^{\frac{\sigma}{2}}$. (For its admissibility, see \cref{thm:admissibility}.)
To this end, observe that for $j\in\NN_0$ and $m\in\ZZ^d$, we have
\begin{equation*}
w_\sigma(2^{-(j-\delta_{j\neq 0})} m)\asymp \lr{1+\frac{\mnorm{m}_2}{2^{j-\delta_{j\neq 0}}}}^\sigma\asymp  \lr{1+\frac{\mnorm{m}_\infty}{2^j}}^\sigma.
\end{equation*}
Now, if $f$ follows a Besov prior with parameters $(\alpha,\hf,\lf,\theta,k)$ and template random variable $X$, \cref{thm:wavelet-characterization} yields
\begin{align}
&\mnorm{f}_{B^s_{p,q}(\RR^d,w_\sigma)}\nonumber\\
&\asymp \mnorm{\lr{2^{j\alpha}(j+1)^\theta \lr{1+\frac{\mnorm{m}_\infty}{2^j}}^{\hf + (\lf-\hf)\one_{j=0}}\xi_{j,t,m}}_{(j,t,m)\in J}}_{b^s_{p,q}(\RR^d,w_\sigma)}\label{eq:weight}\\
&\asymp \mnorm{\lr{2^{j\alpha}(j+1)^\theta \lr{1+\frac{\mnorm{m}_\infty}{2^j}}^{\hf +\sigma+ (\lf+\sigma-(\hf+\sigma))\one_{j=0}}\xi_{j,t,m}}_{(j,t,m)\in J}}_{b^s_{p,q}(\RR^d)}\nonumber\\
&\asymp\mnorm{g}_{B^s_{p,q}(\RR^d)},\nonumber
\end{align}
where $g$ follows a Besov prior with parameters $(\alpha,\hf+\sigma,\lf+\sigma,\theta,k)$ and template random variable $X$.
This holds true under the assumption that \eqref{eq:weight} is finite. 
\end{Bem}

The terminology introduced in the following definitions will greatly simplify the statements and proofs of our main results.
\begin{Def}\label{def:property-a}
Let $d\in\NN$, $s\in\RR$, $p,q\in (0,\infty]$, $\alpha,\hf,\lf,\theta\in\RR$, $\mu,\nu\in (-\infty,0]$, and $r\in (0,\infty)$.
We say that \emph{Property A is satisfied} if
 \begin{equation*}
\lf <-\frac{d}{p},\qquad \hf<-\frac{d}{p},\qquad s+\frac{d}{2}+\alpha<0
\end{equation*}
or 
 \begin{equation*}
\lf <-\frac{d}{p},\qquad \hf<-\frac{d}{p},\qquad s+\frac{d}{2}+\alpha=0,\qquad \theta  <-\frac{1}{q}.
\end{equation*}
When $p=\infty$, the conditions  $\lf <-\frac{d}{p}$, $\hf<-\frac{d}{p}$ have to be replaced by  $\lf\leq 0$, $\hf\leq 0$.
When $q=\infty$, the condition $\theta  <-\frac{1}{q}$ has to be replaced by $\theta\leq 0$.

We say that \emph{Property A' is satisfied} if one of the following sets of conditions is satisfied
\begin{enumerate}[label={(\alph*)},leftmargin=*,align=left,noitemsep]
\item{$p\neq\infty$, $q\neq \infty$, $\lf +\frac{d+\nu}{p}<0$, $\hf +\frac{d+\nu}{p}< 0$, $s+\frac{d}{2}+\alpha+\frac{\mu}{p}< 0$;}
\item{$p\neq\infty$, $q= \infty$, $\lf +\frac{d+\nu}{p}<0$, $\hf +\frac{d+\nu}{p}< 0$, $s+\frac{d}{2}+\alpha+\frac{\mu}{p}\leq 0$;}
\item{$p=\infty$, $q\neq \infty$, $\lf\leq 0$, $\hf\leq 0$, $s+\frac{d}{2}+\alpha< 0$;}
\item{$p=\infty$, $q= \infty$, $\lf\leq 0$, $\hf\leq 0$, $s+\frac{d}{2}+\alpha\leq 0$.}
\end{enumerate}
We say that \emph{Property A'' is satisfied} if one of the following sets of conditions is satisfied
\begin{enumerate}[label={(\alph*)},leftmargin=*,align=left,noitemsep]
\item{$p\neq\infty$, $\lf +\frac{d}{p} +\frac{\nu }{\max\setn{r,p}}< 0$, $\hf  +\frac{d}{p}+\frac{\nu}{\max\setn{r,p}}< 0$, $s+\frac{d}{2}+\alpha+\frac{\mu}{\max\setn{r,p}}< 0$;}
\item{$p=\infty$, $\lf\leq 0$, $\hf\leq 0$, $s+\frac{d}{2}+\alpha\leq 0$.}
\end{enumerate}
\end{Def}

\begin{Def}\label{def:property-b}
Let $d,k\in\NN$, $s,\alpha\in\RR$, and $p\in (0,\infty]$.
We say that \emph{Property B is satisfied} if one of the following sets of conditions is satisfied
\begin{enumerate}[label={(\alph*)},leftmargin=*,align=left,noitemsep]
\item{$p\in (0,1)$, $k>\max\setn{s,d\lr{\frac{1}{p}-1}-s}$,}
\item{$p\in [1,\infty]$, $k>\abs{s}$.}
\end{enumerate}
We say that \emph{Property B' is satisfied} if one of the following sets of conditions is satisfied
\begin{enumerate}[label={(\alph*)},leftmargin=*,align=left,noitemsep]
\item{$p\in (0,1)$, $k>d\lr{\frac{1}{p}-1}+\frac{d}{2}+\alpha$,}
\item{$p\in [1,\infty]$, $k>\frac{d}{2}+\alpha$.}
\end{enumerate}
\end{Def}
Note that Properties B and B' are always satisfied if $k$ is large enough, depending on $d$, $p$, $s$, and $\alpha$.

\section{Sufficiency of Properties A and A''}\label{chap:sufficiency}
In this section, we show that Properties A or A'' are sufficient conditions for having $\EE[\mnorm{f}_{B^s_{p,q}(\RR^d)}^r]<\infty$ for some $r\in(0,\infty)$ if $f$ follows a Besov prior or a \priortwo{} prior, respectively.
(The number $r$ only depends on the concentration properties of the template random variable $X$.)
In case of the Besov prior, we also show the sufficiency of Property A for having $\EE[\exp(c\mnorm{f}_{B^s_{p,q}(\RR^d)}^r)]<\infty$ for some $c\in (0,\infty)$, where $r\in (0,\infty)$ is again given by the concentration properties of the template random variable $X$.
The assumptions imposed on the template random variable $X$ in \cref{def:besov-prior} when $p\neq \infty$ concern the finiteness of its moments.
On the other hand, our standing assumption on $X$ is essential boundedness when $p=\infty$, i.e., there exists a real number $R>0$ such that $\abs{X}\leq R$ almost surely.
This assumption makes the analysis relatively easy but also it cannot be neglected if we want to keep Property A as it is,  cf.\ \cref{thm:essential-boundedness,bem:essential-boundedness}.

The main observation for the proof are the following lemmas which show for $p\neq\infty$ that under Property A, the norm $\mnorm{f}_{B^s_{p,q}(\RR^d)}$ of a function $f$ following a Besov prior is bounded above by a random quantity $\Xi$ that does not depend on the parameters $(\alpha,\hf,\lf,\theta,k)$ but only on the template variable $X$ and on the parameters $p$ and $d$ governing the Besov space.
\begin{Lem}\label{thm:master-lemma}
Let $d\in\NN$, $p\in (0,\infty)$, and $\hf \in\RR$ such that $\hf <-\frac{d}{p}$.
Let further $(\xi_{j,t,m})_{(j,t,m)\in\NN_0\times\setn{F,M}^d\times\ZZ^d}$ be a family of random variables.
Then for all $j\in\NN_0$ and $t\in \setn{F,M}^d$, we have
\begin{equation}\label{eq:master-estimate}
Z_{j,t}\defeq \mnorm{\lr{\lr{1 + \frac{\mnorm{m}_\infty}{2^j}}^\hf \abs{\xi_{j,t,m}}}_{m\in\ZZ^d}}_{\ell_p}\lesssim 2^{j\frac{d}{p}} \cdot \Xi^{\frac{1}{p}},
\end{equation}
where the implied constant only depends on $d$, $\hf$, and $p$.
Here
\begin{align}
\Xi_t&\defeq \sup_{j \in \NN_0} \lr{\frac{1}{M_j} \sum_{\tau=1}^{M_j} \abs{\xi_{j,t,\phi(\tau)}}^p+\sup_{\ell\geq j} \frac{1}{N_\ell} \sum_{\tau=M_\ell + 1}^{M_{\ell+1}} \abs{\xi_{j,t,\phi(\tau)}}^p}\nonumber\\
\intertext{for $t\in\setn{F,M}^d$, and}
\Xi&\defeq \max_{t\in\setn{F,M}^d}\Xi_t\label{eq:master-quantity}
\end{align}
where $\phi:\NN\to \ZZ^d$ is a bijection such that $\tau\mapsto \mnorm{\phi(\tau)}_\infty$ is nondecreasing, and for $j\in\NN_0$, we set
\begin{align*}
N_j&\defeq\#\setcond{m\in\ZZ^d}{2^j<\mnorm{m}_\infty\leq 2^{j+1}},\\
\intertext{and}
M_j&\defeq \#\setcond{m\in\ZZ^d}{\mnorm{m}_{\infty}\leq 2^j}=(2^{j+1}+1)^d.
\end{align*}
\end{Lem}
\begin{proof}
We have $M_j\asymp 2^{jd}$, and also 
\begin{align*}
N_j&=M_{j+1}-M_j=(2^{j+2}+1)^d-(2^{j+1}+1)^d=d\cdot y^{d-1} ((2^{j+2}+1)-(2^{j+1}+1))\\
&\asymp 2^{(d-1)j}2^j=2^{jd}
\end{align*}
for some $y\in (2^{j+1}+1,2^{j+2}+1)$, due to the mean value theorem applied to $x\mapsto x^d$. 
Next note that
\begin{align*}
Z_{j,t}^p&= \sum_{m\in\ZZ^d} \lr{1 + \frac{\mnorm{m}_\infty}{2^j}}^{\hf p}  \abs{\xi_{j,t,m}}^p\nonumber \\
&= \sum_{\substack{m\in\ZZ^d\\\mnorm{m}_\infty\leq 2^j}}\lr{1 + \frac{\mnorm{m}_\infty}{2^j}}^{\hf p}  \abs{\xi_{j,t,m}}^p+\sum_{\substack{m\in\ZZ^d\\\mnorm{m}_\infty> 2^j}}\lr{1 + \frac{\mnorm{m}_\infty}{2^j}}^{\hf p}  \abs{\xi_{j,t,m}}^p \\
&\asymp \sum_{\substack{m\in\ZZ^d\\\mnorm{m}_\infty\leq 2^j}} \abs{\xi_{j,t,m}}^p+\sum_{\substack{m\in\ZZ^d\\\mnorm{m}_\infty> 2^j}}\lr{ \frac{\mnorm{m}_\infty}{2^j}}^{\hf p}  \abs{\xi_{j,t,m}}^p \\
&= \sum_{\tau=1}^{M_j} \abs{\xi_{j,t,\phi(\tau)}}^p+\sum_{\ell\geq j} \sum_{\tau=M_\ell+1}^{M_{\ell+1}}\lr{ \frac{\mnorm{\phi(\tau)}_\infty}{2^j}}^{\hf p}  \abs{\xi_{j,t,\phi(\tau)}}^p \\
&\asymp \sum_{\tau=1}^{M_j} \abs{\xi_{j,t,\phi(\tau)}}^p+2^{-j\hf p}\sum_{\ell\geq j} 2^{\ell d+\ell\hf p} \frac{1}{N_\ell}\sum_{\tau=M_\ell+1}^{M_{\ell+1}} \abs{\xi_{j,t,\phi(\tau)}}^p \\
&= M_j S_{j,t}+2^{-j\hf p}\sum_{\ell\geq j} 2^{\ell d+\ell\hf p} S_{j,t,\ell}\\
&\leq \lr{M_j + 2^{- j \hf p} \sum_{\ell\geq j} 2^{\ell (d + \hf p)} }\cdot\Xi_t\\
&\lesssim 2^{jd} \Xi_t\\
&\leq 2^{jd}\Xi
\end{align*}
since $d+\hf p<0$, where
\begin{equation*}
S_{j,t}\defeq \frac{1}{M_j}\sum_{\tau=1}^{M_j} \abs{\xi_{j,t,\phi(\tau)}}^p\leq \Xi_t \qquad \text{and} \qquad S_{j,t,\ell} \defeq \frac{1}{N_\ell} \sum_{\tau=M_\ell + 1}^{M_{\ell+1}} \abs{\xi_{j,t,\phi(\tau)}}^p\leq \Xi_t
\end{equation*}
for $\ell\geq j$.
 \end{proof}

\begin{Lem}\label{thm:norm-lemma}
Let $d\in\NN$, $s,\alpha,\hf,\lf,\theta \in\RR$, $p\in (0,\infty)$, and $q\in (0,\infty]$.
Let further $(\xi_{j,t,m})_{(j,t,m)\in\NN_0\times\setn{F,M}^d\times\ZZ^d}$ by a family of of random variables.
Set $a=(a_{j,t,m})_{(j,t,m)\in J}$ where
\begin{equation*}
a_{j,t,m}\defeq 2^{j\alpha} \lr{1+\frac{\mnorm{m}_\infty}{2^j}}^{\hf + (\lf-\hf)\one_{j=0}}\lambda_{j,t,m}\xi_{j,t,m}
\end{equation*}
for $(j,t,m)\in J$.
If Property A is satisfied, then for $\Xi$ as in \eqref{eq:master-quantity}, we have $\mnorm{a}_{b^s_{p,q}(\RR^d)}\lesssim\Xi^{\frac{1}{p}}$, where the implied constant only depends on $d$, $p$, $q$, and $\hf$.
\end{Lem}

\begin{proof}
Recall that $\xi_m\defeq \xi_{0,(F,\ldots,F),m}$ for $m\in\ZZ^d$, and let
\begin{align*}
\eta_1&\defeq\mnorm{\lr{(1+\mnorm{m}_\infty)^\lf \abs{\xi_m}}_{m\in\ZZ^d}}_{\ell_p},\\
\eta_2&\defeq\mnorm{\lr{2^{j(s+\frac{d}{2}-\frac{d}{p}+\alpha)}(j+1)^\theta Z_{j,t}}_{j\in\NN,t\in T_j}}_{\ell_q}
\end{align*}
with $Z_{j,t}$ as in \eqref{eq:master-estimate}.
Utilize \eqref{eq:master-estimate} from \cref{thm:master-lemma} with $\lf$ replacing $\hf$, and $j=0$, to conclude that $\eta_1\lesssim\Xi^{\frac{1}{p}}$.
Furthermore, we have
\begin{align*}
\eta_2&=\mnorm{\lr{2^{j(s+\frac{d}{2}-\frac{d}{p}+\alpha)}(j+1)^\theta Z_{j,t}}_{j\in\NN,t\in T_j}}_{\ell_q}\\
&\lesssim\mnorm{\lr{2^{j(s+\frac{d}{2}+\alpha)}(j+1)^\theta \Xi^{\frac{1}{p}}}_{j\in\NN,t\in T_j}}_{\ell_q}\\
&\lesssim\mnorm{\lr{2^{j(s+\frac{d}{2}+\alpha)}(j+1)^\theta }_{j\in\NN}}_{\ell_q}\Xi^{\frac{1}{p}}\\
&\lesssim\Xi^{\frac{1}{p}},
\end{align*}
thanks to Property A.
If $q\neq \infty$, then $\mnorm{a}_{b^s_{p,q}(\RR^d)}=(\eta_1^q+\eta_2^q)^\frac{1}{q}\lesssim\Xi^{\frac{1}{p}}$.
On the other hand, if $q=\infty$, then $\mnorm{a}_{b^s_{p,q}(\RR^d)}=\max \setn{\eta_1,\eta_2}\lesssim\Xi^{\frac{1}{p}}$.
\end{proof}

Now we are ready to prove the sufficiency of Property A for our three target conditions when $f$ follows a Besov prior.
We start with the almost sure membership of $f$ in $B^s_{p,q}(\RR^d)$.
\begin{Satz}\label{thm:target-statement-sufficiency}
Let $d,k\in\NN$, $s,\alpha,\hf,\lf,\theta\in\RR$, and $p,q\in (0,\infty]$.
Assume that $a$ follows a Besov sequence prior and $f$ follows a Besov prior, both with parameters $(\alpha,\hf,\lf,\theta,k)$ and template random variable $X$, satisfying
\begin{equation}\label{eq:finite-moments}
\begin{cases} \ex \eps>0:\,\EE\bigl[\abs{X}^{p(1+\eps)}\bigr]< \infty&\text{if }p\in (0,\infty),\\
\ex  R>0:\, \abs{X}\stackrel{a.s.}{\leq}R&\text{if }p=\infty.\end{cases}
\end{equation}
If Property A holds, then $a$ almost surely belongs to $b^s_{p,q}(\RR^d)$.
Moreover, if Properties A and B are satisfied, then $f$ belongs almost surely to $B^s_{p,q}(\RR^d)$.
\end{Satz}
\begin{proof}
\emph{Step 1: $p\neq\infty$.}
From \cref{thm:norm-lemma}, we obtain $\mnorm{a}_{b^s_{p,q}(\RR^d)}\lesssim \Xi^{\frac{1}{p}}$ with $\Xi$ as in \eqref{eq:master-quantity}.
Next we show that $\Xi\stackrel{a.s.}{<}\infty$.
To this end, for fixed $t\in\setn{F,M}^d$, apply \cref{thm:1+eps} to $X_{j,\tau}\defeq \abs{\xi_{j,t,\phi(\tau)}}^p$ and $\sigma\defeq 1$ , where again $\phi:\NN\to\ZZ^d$ is a bijection such that $\tau\mapsto \mnorm{\phi(\tau)}_\infty$ is nondecreasing.
We obtain $\EE[\Xi_t]<\infty$, which implies that $\Xi_t\stackrel{a.s.}{<}\infty$, whence also $\Xi=\max_{t\in\setn{F,M}^d}\Xi_t\stackrel{a.s.}{<}\infty$.\\[\baselineskip]
\emph{Step 2: $p=\infty$.} Note that $\abs{\xi_{j,t,m}}\stackrel{a.s.}{\leq} R$ and thus
\begin{equation*}
Z_{j,t}= \mnorm{\lr{\lr{1 + \frac{\mnorm{m}_\infty}{2^j}}^\hf \abs{\xi_{j,t,m}}}_{m\in\ZZ^d}}_{\ell_p}\stackrel{a.s.}{\leq}R
\end{equation*}
by $\hf\leq 0$ (which follows from Property A), and moreover
\begin{equation*}
\eta_1=\mnorm{((1+\mnorm{m}_\infty)^\lf \abs{\xi_m})_{m\in\ZZ^d}}_{\ell_p}\stackrel{a.s.}{\leq}R
\end{equation*}
by $\lf\leq 0$  (which follows from Property A).
Finally,
\begin{align*}
\eta_2&=\mnorm{(2^{j(s+\frac{d}{2}+\alpha)}(j+1)^\theta Z_{j,t})_{j\in\NN, t\in T_j}}_{\ell_q}\\
&\stackrel{a.s.}{\leq}R\mnorm{(2^{j(s+\frac{d}{2}+\alpha)}(j+1)^\theta)_{j\in\NN, t\in T_j}}_{\ell_q}<\infty,
\end{align*}
thanks to the conditions on $s$, $d$, $\alpha$, $\theta$, and $q$ from Property A.
This easily implies $\mnorm{a}_{b^s_{p,q}(\RR^d)}\stackrel{a.s.}{\lesssim} R<\infty$.

\emph{Step 3.} If Properties A and B are satisfied, then \cref{thm:wavelet-characterization} gives $\mnorm{f}_{B^s_{p,q}(\RR^d)}= \mnorm{S(a)}_{B^s_{p,q}(\RR^d)}\asymp \mnorm{a}_{b^s_{p,q}(\RR^d)}$ and hence $\mnorm{f}_{B^s_{p,q}(\RR^d)}<\infty$ almost surely.
\end{proof}
Next we show the sufficiency of Property A for the finiteness of the moments of $\mnorm{f}_{B^s_{p,q}(\RR^d)}$ s when $f$ follows a Besov prior.
\begin{Satz}\label{thm:powers}
Let $d,k\in\NN$, $s,\alpha,\hf,\lf,\theta\in\RR$, $p,q\in (0,\infty]$, and $r\in (0,\infty)$.
Assume that $a$ follows a Besov sequence prior and $f$ follows a Besov prior, both with parameters $(\alpha,\hf,\lf,\theta,k)$ and template random variable $X$, satisfying
\begin{equation}\label{eq:rv-powers}
\begin{cases}\ex \eps>0:\,\EE\left[\abs{X}^{(1+\eps)\max\setn{r,p}}\right]<\infty&\text{if }p\in (0,\infty),\\
\ex R>0:\,\abs{X}\stackrel{a.s.}{\leq} R&\text{if }p=\infty.\end{cases}
\end{equation}
If Property A holds, then we have $\EE[\mnorm{a}_{b^s_{p,q}(\RR^d)}^r]<\infty$.
Moreover, if Properties A and B are satisfied, then we have $\EE[\mnorm{f}_{B^s_{p,q}(\RR^d)}^r]<\infty$.
\end{Satz}
\begin{proof}
\emph{Step 1: $p\neq \infty$.}
From \cref{thm:norm-lemma} and the monotonicity of the expectation, we conclude $\EE[\mnorm{a}_{b^s_{p,q}(\RR^d)}^r]\lesssim\EE[\Xi^{\frac{r}{p}}]$ with $\Xi$ as in \eqref{eq:master-quantity}.
Next we show that $\EE[\Xi^{\frac{r}{p}}]<\infty$.

To this end, fix $t\in \setn{F,M}^d$ and apply \cref{thm:1+eps} to $X_{j,\tau}\defeq\abs{\xi_{j,t,\phi(\tau)}}^p$ and $\sigma\defeq \max\setn{1,\frac{r}{p}}$, where again $\phi:\NN\to\ZZ^d$ is a bijection such that $\tau\mapsto \mnorm{\phi(\tau)}_\infty$ is nondecreasing.
This lemma is applicable because
\begin{equation*}
\EE\left[X_{j,\tau}^{\sigma(1+\eps)}\right]=\EE\left[\abs{\xi_{j,t,\phi(\tau)}}^{\max\setn{p(1+\eps),r(1+\eps)}}\right]=\EE\left[\abs{X}^{(1+\eps)\max\setn{p,r}}\right]<\infty.
\end{equation*}
Hence, \cref{thm:1+eps} shows that $\EE[\Xi_t^\sigma]<\infty$ and thus also $\EE[\Xi_t^{\frac{r}{p}}]<\infty$ by Jensen's inequality.
It follows that also
\begin{equation*}
\EE\left[\Xi^{\frac{r}{p}}\right]=\EE\left[\max_{t\in\setn{F,M}^d}\Xi_t^{\frac{r}{p}}\right]\leq\EE\left[\sum_{t\in\setn{F,M}^d}\Xi_t^{\frac{r}{p}}\right]=\sum_{t\in\setn{F,M}^d}\EE\left[\Xi_t^{\frac{r}{p}}\right]<\infty.
\end{equation*}
Step 2: $p=\infty$.
As in Step 2 of the proof of \cref{thm:target-statement-sufficiency}, we have $\mnorm{a}_{b^s_{p,q}(\RR^d)}^r\stackrel{a.s.}{\lesssim} R^r<\infty$.\\[\baselineskip]
\emph{Step 3.} If Properties A and B are satisfied, then \cref{thm:wavelet-characterization} gives $\mnorm{f}_{B^s_{p,q}(\RR^d)}= \mnorm{S(a)}_{B^s_{p,q}(\RR^d)}\asymp \mnorm{a}_{b^s_{p,q}(\RR^d)}$ and hence 
\begin{equation*}
\EE[\mnorm{f}_{B^s_{p,q}(\RR^d)}^r]= \EE[\mnorm{S(a)}_{B^s_{p,q}(\RR^d)}^r]\asymp \EE[\mnorm{a}_{b^s_{p,q}(\RR^d)}^r]<\infty.\qedhere
\end{equation*}
\end{proof}

Next we show the sufficiency of Property A for the finiteness of the moment generating function $\EE[\exp(c\mnorm{f}_{B^s_{p,q}(\RR^d)}^r)]$ for small $c$ when $f$ follows a Besov prior.
If $c\in (-\infty,0]$, then $\exp(c\mnorm{f}_{B^s_{p,q}(\RR^d)}^r)\in (0,1]$ is bounded, which implies that $\EE[\exp(c\mnorm{f}_{B^s_{p,q}(\RR^d)}^r)]<\infty$.
The following theorem is about the non-trivial case where $c$ is positive.
\begin{Satz}\label{thm:mgf}
Let $d,k\in\NN$, $s,\alpha,\hf,\lf,\theta\in\RR$, $p,q\in (0,\infty]$, and $r\in (0,\infty)$.
Assume that $f$ follows a Besov prior with parameters $(\alpha,\hf,\lf,\theta,k)$ and template random variable $X$ with
\begin{equation*}
\begin{cases}\ex C>0:\,\EE[\exp(C\abs{X}^{\max\setn{r,p}})]<\infty&\text{if }p\in (0,\infty),\\
\ex R>0:\,\abs{X}\stackrel{a.s.}{\leq}R&\text{if }p=\infty.\end{cases}
\end{equation*}
If Properties A and B are satisfied, then there exists $c>0$ with $\EE[\exp(c\mnorm{f}_{B^s_{p,q}}^r)]<\infty$.
\end{Satz}
\begin{proof}
\emph{Step 1: $p\neq\infty$.}
For $\Xi$ as in \eqref{eq:master-quantity}, we have $\Xi^{\frac{r}{p}}\leq 1+\Xi^{\frac{\max\setn{r,p}}{p}}$.
\cref{thm:norm-lemma} gives the existence of a constant $\kappa_1=\kappa_1(d,\hf,p,q)>0$ with $\mnorm{a}_{b^s_{p,q}(\RR^d)}\leq \kappa_1 \Xi^{\frac{1}{p}}$ where $a$ is as in \eqref{eq:expansion-seq}.
Moreover, \cref{thm:wavelet-characterization} yields the existence of a constant $\kappa_2=\kappa_2(d,s,p,q,k)>0$ such that $\mnorm{f}_{B^s_{p,q}(\RR^d)}=\mnorm{S(a)}_{B^s_{p,q}(\RR^d)}\leq \kappa_2\mnorm{a}_{b^s_{p,q}(\RR^d)}\leq \kappa_1\kappa_2\Xi^{\frac{1}{p}}$ whenever $\Xi<\infty$.
Thus, for $c>0$ and $\kappa\defeq \kappa_1\kappa_2$, we have
\begin{equation*}
\exp(c \mnorm{f}_{B^s_{p,q}(\RR^d)}^r)\leq \exp(c \kappa^r \Xi^{\frac{r}{p}}) \lesssim \exp\lr{c\kappa^r \Xi^{\frac{\max\setn{r,p}}{p}}}.
\end{equation*}
The monotonicity of the expectation yields
\begin{equation*}
\EE[\exp(c \mnorm{f}_{B^s_{p,q}(\RR^d)}^r)]\lesssim \EE\left[\exp\lr{c\kappa^r \Xi^{\frac{\max\setn{r,p}}{p}}}\right].
\end{equation*}
Next we show that $\EE\left[\exp\lr{c\kappa^r \Xi^{\frac{\max\setn{r,p}}{p}}}\right]<\infty$ for sufficiently small $c>0$.
To this end, fix $t\in \setn{F,M}^d$ and apply \cref{thm:1+eps-exp} to $X_{j,\tau} \defeq\abs{\xi_{j,t,\phi(\tau)}}^p$ and $\sigma\defeq \frac{\max\setn{r,p}}{p}$, where again $\phi:\NN\to\ZZ^d$ is a bijection such that $\tau\mapsto \mnorm{\phi(\tau)}_\infty$ is nondecreasing.
This lemma is applicable by choosing $c=\frac{C}{2^{\sigma+1} (1+\eps)\kappa^r}$ for some $\eps>0$ and noting
\begin{equation*}
\EE\left[\exp\lr{CX_{j,\tau}^\sigma}\right]=\EE\left[\exp\lr{C\abs{\xi_{j,t,\phi(\tau)}}^{\max\setn{r,p}}}\right]=\EE\left[\exp\lr{C\abs{X}^{\max\setn{r,p}}}\right]<\infty.
\end{equation*}
 \cref{thm:1+eps-exp} now yields $\EE[\exp(c\kappa^r \Xi_t^\sigma)]=\EE[\exp(\frac{C}{2^{\sigma+1} (1+\eps)}\Xi_t^\sigma)]<\infty$.
It follows that
\begin{align*}
\EE\left[\exp\lr{c\kappa^r \Xi^\sigma}\right]&= \EE\left[\exp\lr{c\kappa^r \max_{t\in\setn{F,M}^d}\Xi_t^\sigma}\right]\leq \EE\left[\exp\lr{c\kappa^r \sum_{t\in\setn{F,M}^d}\Xi_t^\sigma}\right]\\
&= \EE\left[ \prod_{t\in\setn{F,M}^d}\exp\lr{c\kappa^r\Xi_t^\sigma}\right]=\prod_{t\in\setn{F,M}^d} \EE\left[ \exp\lr{c\kappa^r\Xi_t^\sigma}\right]<\infty.
\end{align*}
~\\[\baselineskip]
\emph{Step 2: $p=\infty$.}
As in Step 2 of the proof of \cref{thm:target-statement-sufficiency}, we have $\mnorm{a}_{b^s_{p,q}(\RR^d)}\lesssim R$ where $a$ is as in \eqref{eq:expansion-seq}.
By \cref{thm:wavelet-characterization}, it follows again that $\mnorm{f}_{B^s_{p,q}(\RR^d)}\lesssim R$, i.e., there exists a constant $\kappa>0$ with $\mnorm{f}_{B^s_{p,q}(\RR^d)}^r\leq \kappa R^r$ almost surely.
Then the monotonicity of the expectation directly gives $\EE[\exp(c \mnorm{f}_{B^s_{p,q}(\RR^d)}^r)]\leq \EE[\exp(c\kappa R^r)]<\infty$.
\end{proof}

We conclude this section by establishing the sufficiency of Property A'' for the finiteness of the moments of $\mnorm{f}_{B^s_{p,q}(\RR^d)}$  (and thus for the almost sure membership of $f$ in $B^s_{p,q}(\RR^d)$) when $f$ follows a \priortwo{} prior.
Let us start with the counterpart of \cref{thm:master-lemma}.
\begin{Lem}\label{thm:master-lemma-bernoulli}
Let $d\in\NN$, $p,r\in (0,\infty)$, $\hf \in\RR$, $\mu,\nu\in (-\infty,0]$, and $\delta\in (0,1)$ such that $\hf p+d+\frac{\nu p}{\max\setn{r,p}} (1-\delta)<0$.
Let further $(\xi_{j,t,m})_{(j,t,m)\in\NN_0\times\setn{F,M}^d\times\ZZ^d}$ and $(\lambda_{j,t,m})_{(j,t,m)\in\NN_0\times\setn{F,M}^d\times\ZZ^d}$ be two families of random variables with $\xi_{j,t,m}\in\RR$ and $\lambda_{j,.t,m}\in\setn{0,1}$ for $(j,t,m)\in\NN_0\times\setn{F,M}^d\times\ZZ^d$.
Finally, set
\begin{equation*}
\varrho_{j,m}\defeq 2^{j\mu}\lr{1+\frac{\mnorm{m}_\infty}{2^j }}^\nu
\end{equation*}
for $j\in\NN_0$ and $m\in\ZZ^d$.
Then for all $j\in\NN_0$ and $t\in \setn{F,M}^d$, we have
\begin{equation}\label{eq:master-estimate-bernoulli}
\widetilde{Z}_{j,t}\defeq \mnorm{\lr{\lr{1 + \frac{\mnorm{m}_\infty}{2^j}}^\hf  \lambda_{j,t,m}\abs{\xi_{j,t,m}}}_{m\in\ZZ^d}}_{\ell_p}\lesssim 2^{j\frac{d}{p}+j\frac{\mu}{\max\setn{r,p}}(1-\delta)}\widetilde{\Xi}^{\frac{1}{p}},
\end{equation}
where the implied constant only depends on $d$, $p$, $\hf$, $\mu$, $\nu$, $r$, and $\delta$.
Here 
\begin{align}
\widetilde{\Xi}_t&\defeq \sup_{j \in \NN_0} \Biggl(\frac{1}{M_j} \sum_{\tau=1}^{M_j} \varrho_{j,\phi(\tau)}^{\frac{p}{\max\setn{r,p}}(\delta-1)}\lambda_{j,t,\phi(\tau)}\abs{\xi_{j,t,\phi(\tau)}}^p\nonumber\\
&\hspace{2.5cm}+\sup_{\ell\geq j} \frac{1}{N_\ell} \sum_{\tau=M_\ell + 1}^{M_{\ell+1}}\varrho_{j,\phi(\tau)}^{\frac{p}{\max\setn{r,p}}(\delta-1)} \lambda_{j,t,\phi(\tau)}\abs{\xi_{j,t,\phi(\tau)}}^p\Biggr),\nonumber\\
\intertext{for $t\in\setn{F,M}^d$ and}
\widetilde{\Xi}&\defeq \max_{t\in\setn{F,M}^d}\widetilde{\Xi}_t,\label{eq:master-quantity-bernoulli}
\end{align}
where $\phi:\NN\to \ZZ^d$ is a bijection such that $\tau\mapsto \mnorm{\phi(\tau)}_\infty$ is nondecreasing, and for $j\in\NN_0$,
\begin{align*}
N_j&\defeq\#\setcond{m\in\ZZ^d}{2^j<\mnorm{m}_\infty\leq 2^{j+1}},\\
\intertext{and}
M_j&\defeq \#\setcond{m\in\ZZ^d}{\mnorm{m}_{\infty}\leq 2^j}=(2^{j+1}+1)^d.
\end{align*}
\end{Lem}
\begin{proof}
Let $j\in \NN_0$.
If $1\leq \tau\leq M_j$, then $0\leq\mnorm{\phi(\tau)}_\infty \leq 2^j$, i.e., $1\leq 1+\frac{\mnorm{\phi(\tau)}_\infty}{2^j}\leq 2$.
It follows that $\varrho_{j,\phi(\tau)}\asymp 2^{j\mu}$.
If, on the other hand, $\tau>M_j$, then there exists a unique number $\ell\in\NN_0$ with $\ell\geq j$ such that $M_\ell+1\leq \tau\leq M_{\ell+1}$.
Then $2^\ell< \mnorm{\phi(\tau)}_\infty\leq 2^{\ell+1}$ and
\begin{equation*}
2^{\ell-j}\leq 1+2^{\ell-j}\leq 1+\frac{\mnorm{\phi(\tau)}_\infty}{2^j}\leq 1+2^{\ell+1-j}\leq 3\cdot 2^{\ell-j},
\end{equation*}
i.e., $\varrho_{j,\phi(\tau)}\asymp 2^{j\mu}\cdot 2^{(\ell-j)\nu}$.
For $j\in \NN_0$ and $t\in T_j$, we thus obtain
\begin{align*}
\widetilde{Z}_{j,t}^p &=\sum_{m\in\ZZ^d} \lr{1+\frac{\mnorm{m}_\infty}{2^j}}^{\hf p}\lambda_{j,t,m}\abs{\xi_{j,t,m}}^p\\
&=\sum_{\tau=1}^{M_j}\lr{1+\frac{\mnorm{\phi(\tau)}_\infty}{2^j}}^{\hf p}\lambda_{j,t,\phi(\tau)}\abs{\xi_{j,t,\phi(\tau)}}^p\\
&\qquad+\sum_{\ell\geq j}\sum_{\tau=M_\ell+1}^{M_{\ell+1}}\lr{1+\frac{\mnorm{\phi(\tau)}_\infty}{2^j}}^{\hf p}\lambda_{j,t,\phi(\tau)}\abs{\xi_{j,t,\phi(\tau)}}^p\\
&\lesssim M_j\frac{1}{M_j}\sum_{\tau=1}^{M_j}\lambda_{j,t,\phi(\tau)}\abs{\xi_{j,t,\phi(\tau)}}^p+\sum_{\ell\geq j}N_\ell\frac{1}{N_\ell}\sum_{\tau=M_\ell+1}^{M_{\ell+1}}2^{(\ell -j)\hf p}\lambda_{j,t,\phi(\tau)}\abs{\xi_{j,t,\phi(\tau)}}^p.
\end{align*}
Next, we note 
\begin{align*}
&M_j\frac{1}{M_j}\sum_{\tau=1}^{M_j}\lambda_{j,t,\phi(\tau)}\abs{\xi_{j,t,\phi(\tau)}}^p\\
&\asymp 2^{jd}2^{j\mu \frac{p}{\max\setn{r,p}}(1-\delta)}\frac{1}{M_j}\sum_{\tau=1}^{M_j} \varrho_{j,\phi(\tau)}^{\frac{p}{\max\setn{r,p}}(\delta-1)}\lambda_{j,t,\phi(\tau)}\abs{\xi_{j,t,\phi(\tau)}}^p\\
&\lesssim  2^{jd} 2^{j\mu \frac{p}{\max\setn{r,p}}(1-\delta)} \widetilde{\Xi}_t.
\end{align*}
Finally, recall from the proof of \cref{thm:master-lemma} that $N_\ell \asymp 2^{\ell d}$, and hence from the assumptions of the present lemma that $\hf p+d+\frac{\nu p}{\max\setn{r,p}}(1-\delta)<0$, and hence
\begin{align*}
&\sum_{\ell\geq j}N_\ell\frac{1}{N_\ell}\sum_{\tau=M_\ell+1}^{M_{\ell+1}}2^{(\ell -j)\hf p}\lambda_{j,t,\phi(\tau)}\abs{\xi_{j,t,\phi(\tau)}}^p\\
&\asymp \sum_{\ell\geq j} 2^{(\ell -j)\hf p+\ell d }(2^{j\mu}2^{(\ell-j)\nu})^{\frac{p(1-\delta)}{\max\setn{r,p}}}\frac{1}{N_\ell}\sum_{\tau=M_\ell+1}^{M_{\ell+1}}\varrho_{j,\phi(\tau)}^{\frac{p(\delta-1)}{\max\setn{r,p}}}\lambda_{j,t,\phi(\tau)}\abs{\xi_{j,t,\phi(\tau)}}^p\\
&\lesssim \sum_{\ell\geq j} 2^{(\ell -j)\hf p+\ell d}(2^{j\mu}2^{(\ell-j)\nu})^{\frac{p(1-\delta)}{\max\setn{r,p}}}\widetilde{\Xi}_t\\
&=2^{-j\hf p+j\mu \frac{p}{\max\setn{r,p}}(1-\delta)-j\nu \frac{p(1-\delta)}{\max\setn{r,p}}}\sum_{\ell\geq j} 2^{\ell \hf p+\ell d +\ell  \nu\frac{p(1-\delta)}{\max\setn{r,p}}}\widetilde{\Xi}_t\\
&\lesssim 2^{-j\hf p+j\mu \frac{p}{\max\setn{r,p}}(1-\delta)-j\nu \frac{p(1-\delta)}{\max\setn{r,p}}} 2^{j(\hf p+d+\nu \frac{p(1-\delta)}{\max\setn{r,p}})}\widetilde{\Xi}_t\\
&= 2^{jd} 2^{j\mu \frac{p(1-\delta)}{\max\setn{r,p}}} \widetilde{\Xi}_t,
\end{align*}
Overall, we conclude $\widetilde{Z}_{j,t}^p\lesssim  2^{jd+j\mu\frac{p}{\max\setn{r,p}}(1-\delta)}\widetilde{\Xi}_t \leq   2^{jd+j\mu\frac{p}{\max\setn{r,p}}(1-\delta)}\widetilde{\Xi}$.
\end{proof}

Next, we show an analog of \cref{thm:norm-lemma}.
\begin{Lem}\label{thm:norm-lemma-bernoulli}
Let $d\in\NN$, $s,\alpha,\hf,\lf\in\RR$, $p,r\in (0,\infty)$, $q\in (0,\infty]$, and $\mu,\nu \in (-\infty,0]$.
Let further $(\xi_{j,t,m})_{(j,t,m)\in\NN_0\times\setn{F,M}^d\times\ZZ^d}$ and $(\lambda_{j,t,m})_{(j,t,m)\in\NN_0\times\setn{F,M}^d\times\ZZ^d}$ be two families of random variables with $\xi_{j,t,m}\in\RR$ and $\lambda_{j,t,m}\in\setn{0,1}$ for $(j,t,m)\in\NN_0\times\setn{F,M}^d\times\ZZ^d$.
Set
\begin{equation*}
\varrho_{j,m}\defeq 2^{j\mu}\lr{1+\frac{\mnorm{m}_\infty}{2^j }}^\nu
\end{equation*}
and $a=(a_{j,t,m})_{(j,t,m)\in J}$ where
\begin{equation*}
a_{j,t,m}\defeq 2^{j\alpha} \lr{1+\frac{\mnorm{m}_\infty}{2^j}}^{\hf + (\lf-\hf)\one_{j=0}}\lambda_{j,t,m}\xi_{j,t,m}
\end{equation*}
for $(j,t,m)\in J$.
If Property A'' is satisfied, then for $\widetilde{\Xi}$ as in \eqref{eq:master-quantity-bernoulli}, we have $\mnorm{a}_{b^s_{p,q}(\RR^d)}\lesssim\widetilde{\Xi}^{\frac{1}{p}}$ where the implied constant only depends on $\mu$, $\nu$, $\hf$, $\lf$, $p$, $d$, and $r$.
\end{Lem}
\begin{proof}
Recall that $\lambda_m\defeq \lambda_{0,(F,\ldots,F),m}$ and $\xi_m\defeq \xi_{0,(F,\ldots,F),m}$ for $m\in\ZZ^d$, and let
\begin{align*}
\widetilde{\eta}_1&\defeq\mnorm{\lr{(1+\mnorm{m}_\infty)^\lf \lambda_m\abs{\xi_m}}_{m\in\ZZ^d}}_{\ell_p},\\
\widetilde{\eta}_2&\defeq\mnorm{\lr{2^{j(s+\frac{d}{2}-\frac{d}{p}+\alpha)} \widetilde{Z}_{j,t}}_{j\in\NN,t\in T_j}}_{\ell_q}
\end{align*}
with $\widetilde{Z}_{j,t}$ as in \eqref{eq:master-estimate-bernoulli}.

By Property A'', there exists $\delta\in (0,1)$ such that $\hf p+d+\nu \frac{p}{\max\setn{r,p}}(1-\delta)<0$,  $\lf p+d+\nu \frac{p}{\max\setn{r,p}}(1-\delta)<0$ and $s+\frac{d}{2}+\alpha+ \frac{\mu}{\max\setn{r,p}}(1-\delta)<0$.
\cref{thm:master-lemma-bernoulli} with $\lf$ replacing $\hf$ and with $j=0$ directly gives $\widetilde{\eta}_1 \lesssim \widetilde{\Xi}^{\frac{1}{p}}$.
Furthermore, again by \cref{thm:master-lemma-bernoulli}, we have
\begin{align*}
\widetilde{\eta}_2&=\mnorm{\lr{2^{j(s+\frac{d}{2}-\frac{d}{p}+\alpha)} \widetilde{Z}_{j,t}}_{j\in\NN,t\in T_j}}_{\ell_q}\\
&\lesssim\mnorm{\lr{2^{j(s+\frac{d}{2}+\alpha+\frac{\mu}{\max\setn{r,p}}(1-\delta))}\widetilde{\Xi}^{\frac{1}{p}}}_{j\in\NN,t\in T_j}}_{\ell_q}\\
&\lesssim\mnorm{\lr{2^{j(s+\frac{d}{2}+\alpha+\frac{\mu}{\max\setn{r,p}}(1-\delta))}}_{j\in\NN}}_{\ell_q}\widetilde{\Xi}^{\frac{1}{p}}\\
&\lesssim\widetilde{\Xi}^{\frac{1}{p}}
\end{align*}
since $s+\frac{d}{2}+\alpha+\frac{\mu}{\max\setn{r,p}}(1-\delta)<0$.
If $q\neq \infty$, then $\mnorm{a}_{b^s_{p,q}(\RR^d)}=(\widetilde{\eta}_1^q+\widetilde{\eta}_2^q)^\frac{1}{q}\lesssim\widetilde{\Xi}^{\frac{1}{p}}$.
On the other hand, if $q=\infty$, then $\mnorm{a}_{b^s_{p,q}(\RR^d)}=\max \setn{\widetilde{\eta}_1,\widetilde{\eta}_2}\lesssim\widetilde{\Xi}^{\frac{1}{p}}$.
\end{proof}

Now we are ready to show, analogously to \cref{thm:powers}, that Property A'' is sufficient to guarantee $\EE[\mnorm{f}_{B^s_{p,q}(\RR^d)}^r]<\infty$ if $f$ follows a \priortwo{} prior.
\begin{Satz}\label{thm:powers-bernoulli}
Let $d,k\in\NN$, $s,\alpha,\hf,\lf\in\RR$, $p,q\in (0,\infty]$, $r\in (0,\infty)$, and $\mu,\nu \in (-\infty,0]$.
Assume that $a$ follows a \priortwo{} prior and $f$ follows a \priortwo{} prior, both with parameters $(\alpha,\hf,\lf,\mu,\nu,k)$ and template random variable $X$, satisfying
\begin{equation}\label{eq:rv-powers-bernoulli}
\begin{cases} \ex \eps>0:\,\EE\bigl[\abs{X}^{(1+\eps)\max\setn{r,p}}\bigr]< \infty&\text{if }p\in (0,\infty),\\
\ex  R>0:\, \abs{X}\stackrel{a.s.}{\leq}R&\text{if }p=\infty.\end{cases}
\end{equation}
If Property A'' is satisfied, then $\EE[\mnorm{a}_{b^s_{p,q}(\RR^d)}^r]<\infty$.
If Properties A'' and B are satisfied, then $\EE[\mnorm{f}_{B^s_{p,q}(\RR^d)}^r]<\infty$.
In particular, $f$ belongs to $B^s_{p,q}(\RR^d)$ almost surely.
\end{Satz}
\begin{proof}
\emph{Step 1: $p\neq\infty$.}
Let $\eps_0>0$ such that $\EE[\abs{X}^{(1+\eps_0)\max\setn{r,p}}]<\infty$.
By Property A'', there exists $\delta\in (0,1)$ such that $\hf p+d+\nu \frac{p}{\max\setn{r,p}}(1-\delta)<0$,  $\lf p+d+\nu \frac{p}{\max\setn{r,p}}(1-\delta)<0$, and $s+\frac{d}{2}+\alpha+ \frac{\mu}{\max\setn{r,p}}(1-\delta)<0$.
Moreover, we choose $\eps\in (0,\eps_0)$ and $\delta+\delta\eps-\eps\geq 0$.
From \cref{thm:norm-lemma-bernoulli} and the monotonicity of the expectation, we obtain $\EE[\mnorm{a}_{b^s_{p,q}(\RR^d)}^r]\lesssim \EE[\widetilde{\Xi}^{\frac{r}{p}}]$ with $\widetilde{\Xi}$ as in \eqref{eq:master-quantity-bernoulli}.
Next we show that $\EE[\widetilde{\Xi}^{\frac{r}{p}}]\stackrel{a.s.}{<}\infty$.
To this end, for fixed $t\in\setn{F,M}^d$, apply \cref{thm:bernoulli-lemma} to $X_{j,\tau}\defeq\abs{\xi_{j,t,\phi(\tau)}}^p$ and $\sigma\defeq \max\setn{1,\frac{r}{p}}$, where again $\phi:\NN\to\ZZ^d$ is a bijection such that $\tau\mapsto \mnorm{\phi(\tau)}_\infty$ is nondecreasing.
This lemma is applicable because
\begin{equation*}
\EE\left[X_{j,\tau}^{\sigma(1+\eps)}\right]=\EE\left[\abs{X}^{(1+\eps)\max\setn{p,r}}\right]\lesssim \EE\left[\abs{X}^{(1+\eps_0)\max\setn{p,r}}\right]<\infty.
\end{equation*}
Hence, \cref{thm:bernoulli-lemma} shows that $\EE[\widetilde{\Xi}_t^\sigma]<\infty$ and thus also $\EE[\widetilde{\Xi}_t^{\frac{r}{p}}]<\infty$ by Jensen's inequality.
It follows that also
\begin{equation*}
\EE\left[\widetilde{\Xi}^{\frac{r}{p}}\right]=\EE\left[\max_{t\in\setn{F,M}^d}\widetilde{\Xi}_t^{\frac{r}{p}}\right]\leq\EE\left[\sum_{t\in\setn{F,M}^d}\widetilde{\Xi}_t^{\frac{r}{p}}\right]=\sum_{t\in\setn{F,M}^d}\EE\left[\widetilde{\Xi}_t^{\frac{r}{p}}\right]<\infty.
\end{equation*}
\emph{Step 2: $p=\infty$.} Note that $ \lambda_{j,t,m} \abs{\xi_{j,t,m}}\stackrel{a.s.}{\leq} R$ and thus
\begin{equation*}
\widetilde{Z}_{j,t}= \mnorm{\lr{\lr{1 + \frac{\mnorm{m}_\infty}{2^j}}^\hf \lambda_{j,t,m} \abs{\xi_{j,t,m}}}_{m\in\ZZ^d}}_{\ell_p}\stackrel{a.s.}{\leq}R
\end{equation*}
by $\hf\leq 0$ (which holds by Property A''),
\begin{equation*}
\widetilde{\eta}_1=\mnorm{((1+\mnorm{m}_\infty)^\lf \lambda_m\abs{\xi_m})_{m\in\ZZ^d}}_{\ell_p}\stackrel{a.s.}{\leq}R
\end{equation*}
by $\lf\leq 0$ (which holds by Property A''), and
\begin{equation*}
\widetilde{\eta}_2=\mnorm{(2^{j(s+\frac{d}{2}+\alpha)}\widetilde{Z}_{j,t})_{j\in\NN, t\in T_j}}_{\ell_q}\stackrel{a.s.}{\leq}R\mnorm{(2^{j(s+\frac{d}{2}+\alpha)})_{j\in\NN, t\in T_j}}_{\ell_q}\lesssim R,
\end{equation*}
since $s+\frac{d}{2}+\alpha\leq 0$ by Property A''.

\emph{Step 3.} If Properties A'' and B are satisfied, \cref{thm:wavelet-characterization} gives $\mnorm{f}_{B^s_{p,q}(\RR^d)}= \mnorm{S(a)}_{B^s_{p,q}(\RR^d)}\asymp  \mnorm{a}_{b^s_{p,q}(\RR^d)}$ and hence $\EE[\mnorm{f}_{B^s_{p,q}(\RR^d)}^r]\asymp \EE[\mnorm{a}_{b^s_{p,q}(\RR^d)}^r]<\infty$.
\end{proof}

\section{Necessity of Properties A and A'}\label{chap:necessity}
The weakest condition we consider is that $f\in B^s_{p,q}(\RR^d)$ almost surely.
In this section, we show that Property A is necessary for this when $f$ follows a Besov prior, and that Property A' is necessary if $f$ follows a \priortwo{} prior. 
For both priors, we start with the necessity for the finiteness of the Besov sequence norm of the wavelet coefficients of $f$, and then proceed with the necessity of for the finiteness of the Besov norm of $f$ itself.

\begin{Satz}\label{thm:target-statement-necessity}
Let $d\in\NN$, $s,\alpha,\hf,\lf,\theta \in\RR$, and $p,q\in (0,\infty]$.
Assume that $a$ follows a Besov sequence prior with parameters $(\alpha,\hf,\lf,\theta)$ and template random variable $X$ satisfying
\begin{equation}\label{eq:nontrivial}
\PP(X\neq 0)>0.
\end{equation}
If $a\in b^s_{p,q}(\RR^d)$ almost surely, then also Property A is satisfied.
\end{Satz}
\begin{proof}
For $j\in\NN_0$ and $t\in T_j$, we define the following (random) quantities
\begin{align}
Z_{j,t}&\defeq \mnorm{\lr{\lr{1 + \frac{\mnorm{m}_\infty}{2^j}}^\hf \abs{\xi_{j,t,m}}}_{m\in\ZZ^d}}_{\ell_p},\nonumber\\
\eta_1&\defeq\mnorm{\lr{(1+\mnorm{m}_\infty)^\lf \abs{\xi_m}}_{m\in\ZZ^d}}_{\ell_p},\nonumber\\
\eta_2&\defeq\mnorm{\lr{2^{j(s+\frac{d}{2}-\frac{d}{p}+\alpha)}(j+1)^\theta Z_{j,t}}_{j\in\NN, t\in T_j}}_{\ell_q}.\label{eq:high-frequency-part}
\end{align}
We have $\mnorm{a}_{b^s_{p,q}(\RR^d)}\stackrel{a.s.}{<}\infty$ if and only if $\eta_1\stackrel{a.s.}{<}\infty$ and $\eta_2\stackrel{a.s.}{<}\infty$.\\[\baselineskip]
Since $\PP(X\neq 0)>0$, there exists $R\geq 1$ such that $\PP(\frac{1}{R}\leq \abs{X}\leq R)>0$.
For $(j,t,m)\in J$, define $X_{j,t,m}\defeq \one_{\abs{\xi_{j,t,m}}\geq\frac{1}{R}}$.
We write $X_m\defeq X_{j,t,m}$ when $j=0$ (and consequently $t=(F,\ldots,F)$).\\[\baselineskip]
\emph{Step 1a: When $p\neq \infty$, we show that necessarily $\lf <-\frac{d}{p}$.}
Note that $\abs{\xi_m}^p\geq R^{-p} X_m$.
We thus obtain from
\begin{equation*}
\infty\stackrel{a.s.}{>}\eta_1^p\geq R^{-p} \sum_{m\in\ZZ^d}(1+\mnorm{m}_\infty)^{\lf p} X_m
\end{equation*}
and \cref{thm:workhorse-lemma} that
\begin{align*}
\infty &> \sum_{m\in\ZZ^d} \EE\left[\min\setn{1,(1+\mnorm{m}_\infty)^{\lf p}X_m}\right]\\
&=\sum_{m\in\ZZ^d} \EE\left[X_m \min\setn{1,(1+\mnorm{m}_\infty)^{\lf p}}\right]\\
&=\EE[X_0] \sum_{m\in\ZZ^d}  \min\setn{1,(1+\mnorm{m}_\infty)^{\lf p}}.
\end{align*}
Because of $\EE[X_0]>0$, this implies first of all that $  \min\setn{1,(1+\mnorm{m}_\infty)^{\lf p}}\to 0$ as $\mnorm{m}_\infty\to\infty$, and hence $\lf p<0$, so that $\min\setn{1,(1+\mnorm{m}_\infty)^{\lf p}}=(1+\mnorm{m}_\infty)^{\lf p}$.
Then \cref{thm:count-lattice-points} yields $\lf <-\frac{d}{p}$.\\[\baselineskip]
\emph{Step 1b: When $p=\infty$, we show that necessarily $\lf\leq 0$.}
If we assume $\lf >0$, then
\begin{equation*}
\infty\stackrel{a.s.}{>}\eta_1 \geq R^{-1}\sup_{m\in\ZZ^d}(1+\mnorm{m}_\infty)^\lf X_m
\end{equation*}
implies that almost surely, $X_m=1$ only holds for finitely many $m\in\ZZ^d$.
Since the $X_m$ are independent, the Borel--Cantelli lemma gives $\sum_{m\in\ZZ^d}\PP(X_m=1)<\infty$, which yields the desired contradiction since $\PP(X_m=1)=\PP(\abs{\xi_m}\geq\frac{1}{R})>0$ is independent of $m$.\\[\baselineskip]
\emph{Step 2a: When $p\neq\infty$, we show the remaining part of Property A.}
Let $\xi_{j,t,m}^\ast \defeq \xi_{j,t,m}\one_{\frac{1}{R}\leq\abs{\xi_{j,t,m}}\leq R}$ for $(j,t,m)\in J$ be the truncation of $\xi_{j,t,m}$.
Then the random variables $\xi_{j,t,m}^\ast$ for $(j,t,m)\in J$ are independent and identically distributed with $\EE[\abs{\xi_{j,t,m}^\ast}^r]<\infty$ for all $r\in (0,\infty)$, and
\begin{equation*}
\infty \stackrel{a.s.}{>}\eta_2\geq \mnorm{\lr{2^{j(s+\frac{d}{2}-\frac{d}{p}+\alpha)}(j+1)^\theta Y_{j,t}^\frac{1}{p}}_{j\in\NN,t\in T_j}}_{\ell_q}
\end{equation*}
where, for $j\in\NN$ and $t\in T_j$, we set
\begin{equation*}
Y_{j,t}\defeq \mnorm{\lr{\lr{1 + \frac{\mnorm{m}_\infty}{2^j}}^\hf \abs{\xi_{j,t,m}^\ast}}_{m\in\ZZ^d}}_{\ell_p}^p.
\end{equation*}
In particular, we see that almost surely, we have for  $t^\ast=(M,\ldots,M)$ that
\begin{equation*}
\infty >Y_{1,t^\ast}=\sum_{m\in\ZZ^d}\lr{1+\frac{\mnorm{m}_\infty}{2}}^{\hf p}\abs{\xi_{1,t^\ast,m}^\ast}^p,
\end{equation*}
so that \cref{thm:workhorse-lemma} shows for $X^\ast_{j,t,m}\defeq \one_{\frac{1}{R}\leq \abs{\xi_{j,t,m}}\leq R}$ that
\begin{align*}
\infty &>\sum_{m\in\ZZ^d}\EE\left[\min\setn{1,\lr{1+\frac{\mnorm{m}_\infty}{2}}^{\hf p}\abs{\xi_{1,t^\ast,m}^\ast}^p}\right]\\
&\gtrsim R^{-p}\sum_{m\in\ZZ^d}\EE\left[X^\ast_{1,t^\ast,m}\min\setn{1,(1+\mnorm{m}_{\infty})^{\hf p}}\right]\\
&\gtrsim \sum_{m\in\ZZ^d}\min\setn{1,(1+\mnorm{m}_{\infty})^{\hf p}}.
\end{align*}
Using a similar argument as in Step 1a, we obtain $\hf <-\frac{d}{p}$.

Now, the fact that the random variables $\xi_{j,t,m}$ are identically distributed and $\hf p<-d$ yields with \cref{thm:count-lattice-points} that
\begin{align}
\EE[Y_{j,t}]&=\sum_{m\in\ZZ^d}\lr{1+\frac{\mnorm{m}_\infty}{2^j}}^{\hf p}\EE[\abs{\xi_{j,t,m}^\ast}^p]\nonumber\\
&=\EE[\abs{\xi_0^\ast}^p]\sum_{m\in\ZZ^d}\lr{1+\frac{\mnorm{m}_\infty}{2^j}}^{\hf p}\asymp 2^{jd},\label{eq:e-cj}
\end{align}
uniformly with respect to $j\in\NN$ and $t\in T_j$.
Here, we used that $\PP(\abs{\xi_0^\ast}>0)=\PP(\frac{1}{R}\leq \abs{X}\leq R)>0$, so that $\EE[\abs{\xi_0^\ast}^p]>0$.

Now let $\tau\in (1,\infty)$ and let $\tau^\prime=\frac{\tau}{\tau-1}$ the conjugate exponent of $\tau$.
Since $\EE[\abs{\xi_0^\ast}^{p\tau}]<\infty$, \cref{thm:count-lattice-points} and Hölder's inequality give
\begin{align}
\EE[Y_{j,t}^\tau]&=\EE\left[\lr{\sum_{m\in\ZZ^d}\lr{1+\frac{\mnorm{m}_\infty}{2^j}}^{\hf p} \abs{\xi_{j,t,m}^\ast}^p}^\tau\right]\nonumber\\
&=\EE\left[\lr{\sum_{m\in\ZZ^d}\lr{1+\frac{\mnorm{m}_\infty}{2^j}}^{\frac{\hf p}{\tau}}\lr{1+\frac{\mnorm{m}_\infty}{2^j}}^{\frac{\hf p}{\tau^\prime}} \abs{\xi_{j,t,m}^\ast}^p}^\tau\right]\nonumber\\
&\leq\EE\left[ \lr{\sum_{m\in\ZZ^d}\lr{1+\frac{\mnorm{m}_\infty}{2^j}}^{\hf p}\abs{\xi_{j,t,m}^\ast}^{p\tau}}\lr{\sum_{m\in\ZZ^d}\lr{1+\frac{\mnorm{m}_\infty}{2^j}}^{\hf p}}^{\frac{\tau}{\tau^\prime}}\right]\nonumber\\
&=\EE\left[ \sum_{m\in\ZZ^d}\lr{1+\frac{\mnorm{m}_\infty}{2^j}}^{\hf p}\abs{\xi_{j,t,m}^\ast}^{p\tau}\right]\lr{\sum_{m\in\ZZ^d}\lr{1+\frac{\mnorm{m}_\infty}{2^j}}^{\hf p}}^{\frac{\tau}{\tau^\prime}}\nonumber\\
&\asymp\sum_{m\in\ZZ^d}\lr{1+\frac{\mnorm{m}_\infty}{2^j}}^{\hf p}\EE\left[\abs{\xi_0^\ast}^{p\tau}\right]2^{jd\frac{\tau}{\tau^\prime}}\nonumber\\
&\lesssim  2^{jd} 2^{jd\frac{\tau}{\tau^\prime}}\nonumber\\
&= 2^{jd\tau}.\label{eq:e-cj-tau}
\end{align}
Now let $\tau=2$ in \eqref{eq:e-cj} and \eqref{eq:e-cj-tau} and invoke \cref{thm:paley-zygmund} (Paley--Zygmund inequality) with $X=Y_{j,t}$ and $\sigma=\frac{1}{2}$.
We obtain
\begin{equation*}
\PP\lr{Y_{j,t}\geq \frac{1}{2}\EE[Y_{j,t}]}\geq \frac{1}{4}\frac{\EE[Y_{j,t}]^2}{\EE[Y_{j,t}^2]}\gtrsim 1,
\end{equation*}
which means that there exists a constant $\kappa \in (0,1]$ such that 
\begin{equation*}
\PP\lr{Y_{j,t}\geq \frac{1}{2}\EE[Y_{j,t}]}\geq\kappa
\end{equation*}
for all $j\in\NN$ and $t\in T_j$.

The random variables $\chi_{j,t}\defeq\one_{\setn{Y_{j,t}\geq \frac{1}{2}\EE[Y_{j,t}]}}$ with $j\in\NN$ and $t\in T_j$ are independent and Bernoulli$(p_{j,t})$-distributed, for some $p_{j,t}\geq \kappa$.
Using \eqref{eq:e-cj}, we obtain
\begin{align*}
\infty&\stackrel{a.s}{>}\mnorm{\lr{2^{j(s+\frac{d}{2}-\frac{d}{p}+\alpha)}(j+1)^\theta Y_{j,t}^{\frac{1}{p}}}_{j\in\NN,t\in T_j}}_{\ell_q}\\
&\geq \mnorm{\lr{2^{j(s+\frac{d}{2}-\frac{d}{p}+\alpha)}(j+1)^\theta \lr{\frac{1}{2}\EE[Y_{j,t}]}^{\frac{1}{p}} \chi_{j,t}}_{j\in\NN,t\in T_j}}_{\ell_q}\\
&\gtrsim \mnorm{\lr{2^{j(s+\frac{d}{2}-\frac{d}{p}+\alpha)}(j+1)^\theta  \lr{2^{jd}}^{\frac{1}{p}} \chi_{j,t}}_{j\in\NN,t\in T_j}}_{\ell_q}\\
&\gtrsim \mnorm{\lr{2^{j(s+\frac{d}{2}+\alpha)}(j+1)^\theta  \chi_{j,t}}_{j\in\NN,t\in T_j}}_{\ell_q}.
\end{align*}
Let $t^\ast=(M,\ldots,M)$.
If $q=\infty$, note that the events $E_j\defeq\setn{\chi_{j,t^\ast}=1}$ for $j\in \NN$ are independent, and it holds $\sum_{j\in\NN}\PP(E_j)\geq \sum_{j\in\NN}\kappa =\infty$. 
The Borel--Cantelli lemma thus yields that almost surely, the (random) set $I\defeq\setcond{j\in\NN}{\chi_{j,t^\ast}=1}$ is an infinite set and 
\begin{equation*}
\infty \stackrel{a.s.}{>} \mnorm{\lr{2^{j(s+\frac{d}{2}+\alpha)}(j+1)^\theta  \chi_{j,t^\ast}}_{j\in\NN}}_{\ell_q}\geq \sup_{j\in I}2^{j(s+\frac{d}{2}+\alpha)}(j+1)^\theta.
\end{equation*}
This shows that $s+\frac{d}{2}+\alpha \leq 0$ and in case of $s+\frac{d}{2}+\alpha=0$ also that $\theta\leq 0$.
On the other hand, if $q\neq \infty$,  it follows from \cref{thm:workhorse-lemma} that
\begin{align*}
\infty&>\sum_{j\in\NN}\sum_{t\in T_j}\EE\left[\min\setn{1,2^{jq(s+\frac{d}{2}+\alpha)}(j+1)^{\theta q}  \chi_{j,t}}\right]\\
&=\sum_{j\in\NN}\sum_{t\in T_j}\EE\left[\chi_{j,t}\min\setn{1,2^{jq(s+\frac{d}{2}+\alpha)}(j+1)^{\theta q} }\right]\\
&=\sum_{j\in\NN}\sum_{t\in T_j} \EE[\chi_{j,t}]\min\setn{1,2^{jq(s+\frac{d}{2}+\alpha)}(j+1)^{\theta q}  }\\
&\geq\kappa\sum_{j\in\NN} \min\setn{1,2^{jq(s+\frac{d}{2}+\alpha)}(j+1)^{\theta q} },
\end{align*}
i.e., $s+\frac{d}{2}+\alpha \leq 0$ and moreover $\theta q <-1$ in case that $s+\frac{d}{2}+\alpha=0$.\\[\baselineskip]
\emph{Step 2b: When $p=\infty$, we show the remaining part of Property A.}
Suppose $\hf>0$.
Note because of $\eta_2\stackrel{a.s.}{<}\infty$ that
\begin{equation*}
\infty\stackrel{a.s.}{>}Z_{1,t^\ast} \geq R^{-1}\sup_{m\in\ZZ^d}\lr{1+\frac{\mnorm{m}_\infty}{2}}^\hf X_{1,t^\ast,m}.
\end{equation*}
But then it holds almost surely that $X_{1,t^\ast,m}=1$ for only finitely many $m\in\ZZ^d$.
The Borel--Cantelli lemma shows that $\sum_{m\in\ZZ^d}\PP(X_{1,t^\ast,m}=1)<\infty$, which is absurd since $\PP(X_{1,t^\ast,m}=1)=\PP(\abss{\xi_{1,t^\ast,m}}\geq \frac{1}{R})>0$ is independent of $m\in \ZZ^d$.
Hence $\hf \leq 0$.

Now that we know $\hf\leq 0$, we conclude
\begin{align*}
\infty \stackrel{a.s.}{>}\eta_2\geq \mnorm{ \lr{2^{j(s+\frac{d}{2}+\alpha)}(j+1)^\theta\xi_{j,t^\ast,0}}_{j\in\NN}  }_{\ell_q}.
\end{align*}
If $q=\infty$, then the Borel--Cantelli lemma implies as in Step 1b that $s+\frac{d}{2}+\alpha\leq 0$ and that $\theta\leq 0$ in case of $s+\frac{d}{2}+\alpha=0$.

On the other hand, if $q\neq \infty$, we know that with $t^\ast=(M,\ldots,M)$, we have
\begin{align*}
\infty&\stackrel{a.s.}{>}\mnorm{\lr{2^{j(s+\frac{d}{2}+\alpha)}(j+1)^\theta \xi_{j,t^\ast,0}}_{j\in\NN}}_{\ell_q}^q\\
&= \sum_{j\in\NN}2^{jq(s+\frac{d}{2}+\alpha)}(j+1)^{\theta q}\abs{\xi_{j,t^\ast,0}}^q\\
&\geq R^{-q}\sum_{j\in\NN}2^{jq(s+\frac{d}{2}+\alpha)}(j+1)^{\theta q}X_{j,t^\ast,0}.
\end{align*}
The random variables $X_{j,t^\ast,0}$ are independent as well as identically Bernoulli$(\varrho)$-distributed with some parameter $\varrho\in (0,1]$.
\cref{thm:workhorse-lemma} now gives
\begin{align*}
\infty&>\sum_{j\in\NN}\EE\left[\min\setn{1,2^{jq(s+\frac{d}{2}+\alpha)}(j+1)^{\theta q}X_{j,t^\ast,0}}\right]\\
&=\sum_{j\in\NN}\EE\left[X_{j,t^\ast,0}\min\setn{1,2^{jq(s+\frac{d}{2}+\alpha)}(j+1)^{\theta q}}\right]\\
&=\varrho\sum_{j\in\NN}\min\setn{1,2^{jq(s+\frac{d}{2}+\alpha)}(j+1)^{\theta q}}
\end{align*}
This first of all implies that $\lim_{j\to\infty}\min\setn{1,2^{jq(s+\frac{d}{2}+\alpha)}(j+1)^{\theta q}}=0$, and hence $s+\frac{d}{2}+\alpha\leq 0$.
In case $s+\frac{d}{2}+\alpha=0$, it also follows that $\sum_{j\in\NN}(j+1)^{\theta q}<\infty$, i.e., $\theta q<-1$.
\end{proof}
The subtle distinction between the conditions of \cref{thm:target-statement-necessity} and the following \cref{thm:target-statement-necessity-functions} is ignored in most of the literature, where it is simply pretended that if $f=S(a)$, then the conditions $f\in B^s_{p,q}(\RR^d)$ and $a\in b^s_{p,q}(\RR^d)$ are equivalent.
However, one only knows that 
\begin{enumerate}[label={(\roman*)},leftmargin=*,align=left,noitemsep]
\item{if $a,\overline{a}\in b^s_{p,q}(\RR^d)$ with $S(a)=f=S(\overline{a})$, then $a=\overline{a}$,}
\item{if $f\in B^s_{p,q}(\RR^d)$, then $f=S(\overline{a})$ for some $\overline{a}\in b^s_{p,q}(\RR^d)$.}
\end{enumerate}
Without knowing anything about the decay of the coefficient sequence $a$, it is not clear that this implies $a=\overline{a} \in b^s_{p,q}(\RR^d)$.
\begin{Satz}\label{thm:target-statement-necessity-functions}
Let $d,k\in\NN$, $s,\alpha,\hf,\lf,\theta \in\RR$, and $p,q\in (0,\infty]$.
Assume that $f$ follows a Besov prior with parameters $(\alpha,\hf,\lf,\theta,k)$ and template random variable $X$ with 
\begin{equation*}
\begin{cases} \ex \eps>0:\,\EE\bigl[\abs{X}^{p(1+\eps)}\bigr]< \infty&\text{if }p\in (0,\infty),\\
\ex  R>0:\, \abs{X}\stackrel{a.s.}{\leq}R&\text{if }p=\infty.\end{cases}
\end{equation*}
and
\begin{equation*}
\PP(X\neq 0)>0.
\end{equation*}
If $f \in B^s_{p,q}(\RR^d)$ almost surely and if Properties B and B' are satisfied, then also Property A is satisfied.
\end{Satz}
\begin{proof}
From Property B, it follows that $d\lr{\frac{1}{p}-1}-k<s<k$ if $p\in (0,1)$ and $-k<s<k$ otherwise.
Taking Property B' into account, we may choose $\widetilde{s}\in\RR$ such that $d\lr{\frac{1}{p}-1}-k<\widetilde{s}<s$ and $\widetilde{s}+\frac{d}{2}+\alpha<0$ when $p\in (0,1)$.
On the other hand, if $p\in [1,\infty]$, we may choose  $\widetilde{s}\in\RR$ such that $-k<\widetilde{s}<s$ and $\widetilde{s}+\frac{d}{2}+\alpha<0$.
Note that in each of the two cases Property B is satisfied with $\widetilde{s}$ in place of $s$.
Choose $\sigma<0$ small enough such that $\lf+\sigma <-\frac{d}{p}$ and $\hf+\sigma<-\frac{d}{p}$ when $p\neq \infty$, and $\lf+\sigma \leq 0$ and $\hf+\sigma\leq 0$ when $p=\infty$.
Since \eqref{eq:finite-moments} holds, we conclude from \cref{thm:target-statement-sufficiency} that $\bar{a}=(\bar{a}_{j,t,m})_{(j,t,m)\in J}\in b^{\widetilde{s}}_{p,q}(\RR^d)$ almost surely where $\xi_{j,t,m}\sim X$ are the i.i.d.\ random variables from \eqref{eq:expansion} and
\begin{equation*}
\bar{a}_{j,t,m}\defeq 2^{j\alpha}(j+1)^\theta \lr{1+\frac{\mnorm{m}_\infty}{2^j}}^{\hf+\sigma +(\lf-\hf)\one_{j=0}}\xi_{j,t,m}
\end{equation*}
With $w:\RR^d\to\RR$, $w_\sigma(x)=(1+\mnorm{x}_2^2)^{\frac{\sigma}{2}}$, we then see in view of \cref{remark:weights} that $a=(a_{j,t,m})_{(j,t,m)\in J}\in b^{\widetilde{s}}_{p,q}(\RR^d,w_\sigma)$ where
\begin{equation*}
a_{j,t,m}=2^{j\alpha}(j+1)^\theta \lr{1+\frac{\mnorm{m}_\infty}{2^j}}^{\hf+(\lf-\hf)\one_{j=0}}\xi_{j,t,m}.
\end{equation*}
If $f\in B^s_{p,q}(\RR^d)$ almost surely, then almost surely $f=S(\widetilde{a})$ for some $\widetilde{a}\in b^s_{p,q}(\RR^d)$ by  \cref{thm:wavelet-characterization}.
From $\sigma<0$ and $s>\widetilde{s}$, we get the embeddings
\begin{equation*}
b^s_{p,q}(\RR^d)\hookrightarrow b^s_{p,q}(\RR^d,w_\sigma)\hookrightarrow b^{\widetilde{s}}_{p,q}(\RR^d,w_\sigma).
\end{equation*} 
But this implies $a,\widetilde{a}\in  b^{\widetilde{s}}_{p,q}(\RR^d,w_\sigma)$ almost surely with $S(a)=f=S(\widetilde{a})$, from which it follows that $a=\widetilde{a}\in b^s_{p,q}(\RR^d)$ almost surely by \cref{thm:wavelet-characterization}.
From \cref{thm:target-statement-necessity}, we conclude that Property A holds.
\end{proof}
We now show that Property A' is necessary for the coefficient sequence in \eqref{eq:expansion-bernoulli} to almost surely belong to the Besov space $b^s_{p,q}(\RR^d)$. 
Our a priori assumptions  $\mu\geq -d$ and $\nu\geq -d$ in the following theorem are parallel to the assumption $\beta\leq 1$ in \cite[Theorem~1]{Abramovich1998}.
\begin{Satz}\label{thm:necessity-bernoulli}
Let $d\in\NN$, $s,\alpha,\hf,\lf \in\RR$, $p,q\in (0,\infty]$, and $\mu,\nu \in [-d,0]$.
Assume that $a$ follows a \priortwo{} sequence prior parameters $(\alpha,\hf,\lf,\mu,\nu)$ with template variable $X$ satisfying
\begin{equation*}
\PP(X\neq 0)>0.
\end{equation*}
If $a \in b^s_{p,q}(\RR^d)$ almost surely, then also Property A' is satisfied.
\end{Satz}
\begin{proof}
For $j\in\NN_0$ and $t\in T_j$, we define the following (random) quantities
\begin{align*}
\widetilde{Z}_{j,t}&\defeq \mnorm{\lr{\lr{1 + \frac{\mnorm{m}_\infty}{2^j}}^\hf \lambda_{j,t,m}\abs{\xi_{j,t,m}}}_{m\in\ZZ^d}}_{\ell_p},\\
\widetilde{\eta}_1&\defeq\mnorm{\lr{(1+\mnorm{m}_\infty)^\lf \lambda_m\abs{\xi_m}}_{m\in\ZZ^d}}_{\ell_p},\\
\widetilde{\eta}_2&\defeq\mnorm{\lr{2^{j(s+\frac{d}{2}-\frac{d}{p}+\alpha)} \widetilde{Z}_{j,t}}_{j\in\NN, t\in T_j}}_{\ell_q}.
\end{align*}
We have $\mnorm{a}_{b^s_{p,q}(\RR^d)}\stackrel{a.s.}{<}\infty$ if and only if $\widetilde{\eta_1}\stackrel{a.s.}{<}\infty$ and $\widetilde{\eta_2}\stackrel{a.s.}{<}\infty$.\\[\baselineskip]
Since $\PP(X\neq 0)>0$, there exists $R\geq 1$ such that $\PP(\frac{1}{R}\leq \abs{X}\leq R)>0$.
For $(j,t,m)\in\NN_0\times\setn{F,M}^d\times \ZZ^d$, define $X_{j,t,m}\defeq \one_{\abs{\xi_{j,t,m}}\geq\frac{1}{R}}$.
We write $X_m\defeq X_{j,t,m}$, $\lambda_m\defeq \lambda_{j,t,m}$, and $\xi_m\defeq \xi_{j,t,m}$ when $j=0$ (and consequently $t=(F,\ldots,F)$).\\[\baselineskip]
\emph{Step 1a: When $p\neq \infty$, we show that necessarily $\nu+\lf p<-d$.}
Note that $\abs{\xi_m}^p\geq R^{-p} X_m$.
We thus obtain from
\begin{equation*}
\infty\stackrel{a.s.}{>}\widetilde{\eta}_1^p\geq R^{-p} \sum_{m\in\ZZ^d}(1+\mnorm{m}_\infty)^{\lf p}\lambda_m X_m
\end{equation*}
and \cref{thm:workhorse-lemma} that
\begin{align*}
\infty &> \sum_{m\in\ZZ^d} \EE\left[\min\setn{1,(1+\mnorm{m}_\infty)^{\lf p}\lambda_m X_m}\right]\\
&=\sum_{m\in\ZZ^d} \EE\left[\lambda_m X_m \min\setn{1,(1+\mnorm{m}_\infty)^{\lf p}}\right]\\
&=\EE[X_0] \sum_{m\in\ZZ^d} \EE[\lambda_m]  \min\setn{1,(1+\mnorm{m}_\infty)^{\lf p}}\\
&=\EE[X_0] \sum_{m\in\ZZ^d} (1+\mnorm{m}_\infty)^\nu  \min\setn{1,(1+\mnorm{m}_\infty)^{\lf p}}\\
&\asymp \sum_{m\in\ZZ^d} (1+\mnorm{m}_\infty)^{\min\setn{\nu,\nu+\lf p}}.
\end{align*}
Then \cref{thm:count-lattice-points} yields $\min\setn{\nu,\nu+\lf p}<-d$ or, equivalently, $\min\setn{0,\gamma}+\frac{d+\nu}{p}<0$.
Since $\nu\geq -d$, we obtain $\gamma<0$, i.e., $0>\min\setn{0,\gamma}+\frac{d+\nu}{p}=\gamma+\frac{d+\nu}{p}$.\\[\baselineskip]
\emph{Step 1b: When $p= \infty$, we show that necessarily $\lf\leq 0$.}
Let us assume towards a contradiction that $\lf >0$.
Then
\begin{equation*}
\infty\stackrel{a.s.}{>}\widetilde{\eta}_1 \geq R^{-1}\sup_{m\in\ZZ^d}(1+\mnorm{m}_\infty)^\lf \lambda_mX_m
\end{equation*}
implies that almost surely, $\lambda_mX_m=1$ only holds for finitely many $m\in\ZZ^d$.
The Borel--Cantelli lemma gives 
\begin{align*}
\sum_{m\in\ZZ^d}\PP(\lambda_m X_m=1)&=\sum_{m\in\ZZ^d}\PP(\lambda_m=1)\PP(X_m=1)\\
&=\sum_{m\in\ZZ^d}(1+\mnorm{m}_\infty)^\nu \PP(X_m=1)<\infty,
\end{align*}
which yields the desired contradiction since $\PP(X_m=1)=\PP(\abs{\xi_m}\geq\frac{1}{R})>0$ is independent of $m$, and $\nu\geq -d$, see \cref{thm:count-lattice-points}.\\[\baselineskip]
\emph{Step 2a: When $p\neq\infty$, we show the remaining part of Property A'.}
Let $\xi_{j,t,m}^\ast \defeq \xi_{j,t,m}\one_{\frac{1}{R}\leq\abs{\xi_{j,t,m}}\leq R}$ for $(j,t,m)\in J$ be the truncation of $\xi_{j,t,m}$.
Then the random variables $\xi_{j,t,m}^\ast$ are independent and identically distributed with $\EE[\abs{\xi_{j,t,m}^\ast}^r]<\infty$ for all $r\in (0,\infty)$, and
\begin{equation*}
\infty \stackrel{a.s.}{>}\widetilde{\eta}_2\geq \mnorm{\lr{2^{j(s+\frac{d}{2}-\frac{d}{p}+\alpha)} \widetilde{Y}_{j,t}^\frac{1}{p}}_{j\in\NN,t\in T_j}}_{\ell_q}
\end{equation*}
where, for $j\in\NN$ and $t\in T_j$, we set
\begin{equation*}
\widetilde{Y}_{j,t}\defeq \mnorm{\lr{\lr{1 + \frac{\mnorm{m}_\infty}{2^j}}^\hf \lambda_{j,t,m}\abs{\xi_{j,t,m}^\ast}}_{m\in\ZZ^d}}_{\ell_p}^p.
\end{equation*}
In particular, we see that almost surely, we have
\begin{equation*}
\infty >\widetilde{Y}_{1,t^\ast}=\sum_{m\in\ZZ^d}\lr{1+\frac{\mnorm{m}_\infty}{2}}^{\hf p}\lambda_{1,t^\ast,m}\abs{\xi_{1,t^\ast,m}^\ast}^p,
\end{equation*}
so that \cref{thm:workhorse-lemma} shows for $X^\ast_{j,t,m}\defeq \one_{\frac{1}{R}\leq \abs{\xi_{j,t,m}}\leq R}$ that
\begin{align*}
\infty &>\sum_{m\in\ZZ^d}\EE\left[\min\setn{1,\lr{1+\frac{\mnorm{m}_\infty}{2}}^{\hf p}\lambda_{1,t^\ast,m}\abs{\xi_{1,t^\ast,m}^\ast}^p}\right]\\
&\gtrsim R^{-p}\sum_{m\in\ZZ^d}\EE\left[\lambda_{1,t^\ast,m} X^\ast_{1,t^\ast,m}\min\setn{1,(1+\mnorm{m}_{\infty})^{\hf p}}\right]\\
&\gtrsim \sum_{m\in\ZZ^d}\EE\left[\lambda_{1,t^\ast,m}\right]\min\setn{1,(1+\mnorm{m}_{\infty})^{\hf p}}\\
&\gtrsim \sum_{m\in\ZZ^d}\min\setn{(1+\mnorm{m}_{\infty})^\nu,(1+\mnorm{m}_{\infty})^{\nu+\hf p}}.
\end{align*}
Using a similar argument as in Step 1a, we obtain $\min\setn{\nu,\nu+\hf p}<-d$, which in turn because of $\nu\geq -d$ implies $\hf+\frac{d+\nu}{p}<0$, and, in particular, $\hf<0$.
Further we have
\begin{equation*}
\infty \stackrel{a.s.}{>}\widetilde{\eta}_2\geq \mnorm{\lr{2^{j(s+\frac{d}{2}-\frac{d}{p}+\alpha)} U_{j,t}^\frac{1}{p}}_{j\in\NN,t\in T_j}}_{\ell_q}
\end{equation*}
where, for $j\in\NN$ and $t\in T_j$, we set
\begin{equation*}
U_{j,t}\defeq \sum_{\substack{m\in\ZZ^d:\\\mnorm{m}_\infty\leq 2^j}}\lr{1 + \frac{\mnorm{m}_\infty}{2^j}}^{\hf p} \lambda_{j,t,m}\abs{\xi_{j,t,m}^\ast}^p.
\end{equation*}
In particular, we see that
\begin{equation*}
\infty \stackrel{a.s.}{>} U_{j,t}\asymp \sum_{\substack{m\in\ZZ^d:\\\mnorm{m}_\infty\leq 2^j}}\lambda_{j,t,m}\abs{\xi_{j,t,m}^\ast}^p\asymp \sum_{\substack{m\in\ZZ^d:\\\mnorm{m}_\infty\leq 2^j}}\lambda_{j,t,m}X_{j,t,m}^\ast.
\end{equation*}
Since the random variables $X_{j,t,m}^\ast$ are i.i.d.\ Bernoulli($\rho$)-distributed for some $\rho\in (0,1]$, we conclude that $\lambda_{j,t,m}X_{j,t,m}^\ast$ is Bernoulli($2^{j\mu}(1+\frac{\mnorm{m}_\infty}{2^j})^\nu \rho$)-distributed.

If $\mnorm{m}_\infty\leq 2^j$, we have 
\begin{equation*}
2^{j\mu}\lr{1+\frac{\mnorm{m}_\infty}{2^j}}^\nu \rho\geq 2^{j\mu+\nu} \rho.
\end{equation*}
Thus there exists a Bernoulli($2^{j\mu+\nu} \rho$)-distributed random variable $B_{j,t,m}$ such that $\lambda_{j,t,m}X_{j,t,m}^\ast\geq B_{j,t,m}$ and such that the family $(B_{j,t,m})_{(j,t,m)}$ is independent.
We conclude
\begin{equation}\label{eq:def-ajt}
\infty \stackrel{a.s.}{>}U_{j,t}\asymp\sum_{\substack{m\in\ZZ^d:\\\mnorm{m}_\infty\leq 2^j}}\lambda_{j,t,m}X_{j,t,m}^\ast\geq \sum_{\substack{m\in\ZZ^d:\\\mnorm{m}_\infty\leq 2^j}} B_{j,t,m}\eqdef A_{j,t}
\end{equation}
with $A_{j,t}\sim\mathrm{Bin}((1+2^{j+1})^d,2^{j\mu+\nu}\rho)$ following a binomial distribution.
\cref{thm:binomial} yields for $\tau\geq 1$ because of $\mu\geq -d$ that
\begin{align*}
\EE[A_{j,t}^\tau]&\lesssim \max\setn{(1+2^{j+1})^d 2^{j\mu+\nu} \rho,((1+2^{j+1})^d 2^{j\mu+\nu} \rho)^\tau}\\
&\asymp \max\setn{2^{j(\mu+d)},2^{j\tau (\mu+d)}}= 2^{j\tau (\mu+d)}.
\end{align*}

Set $\tau=2$ and invoke \cref{thm:paley-zygmund} (Paley--Zygmund inequality) with $X=A_{j,t}$ and $\sigma=\frac{1}{2}$.
We obtain
\begin{equation*}
\PP\lr{A_{j,t}\geq \frac{1}{2}\EE[A_{j,t}]}\geq \frac{1}{4}\frac{\EE[A_{j,t}]^2}{\EE[A_{j,t}^2]}\gtrsim 1,
\end{equation*}
which means that there exists a constant $\kappa >0$ such that 
\begin{equation}\label{eq:paley-zygmund-bernoulli}
\PP\lr{A_{j,t}\geq \frac{1}{2}\EE[A_{j,t}]}\geq\kappa
\end{equation}
for all $j\in\NN$ and $t\in T_j$.

The random variables $\chi_{j,t}\defeq\one_{\setn{A_{j,t}\geq \frac{1}{2}\EE[A_{j,t}]}}$ with $j\in\NN$ and $t\in T_j$ are independent and Bernoulli$(p_{j,t})$-distributed, for some $p_{j,t}\geq \kappa$.
Recall that $U_{j,t}\gtrsim A_{j,t}$. 
Hence,
\begin{align*}
\infty&\stackrel{a.s}{>}\mnorm{\lr{2^{j(s+\frac{d}{2}-\frac{d}{p}+\alpha)}A_{j,t}^{\frac{1}{p}}}_{j\in\NN,t\in T_j}}_{\ell_q}\\
&\geq \mnorm{\lr{2^{j(s+\frac{d}{2}-\frac{d}{p}+\alpha)}\lr{\frac{1}{2}\EE[A_{j,t}]}^{\frac{1}{p}} \chi_{j,t}}_{j\in\NN,t\in T_j}}_{\ell_q}\\
&\gtrsim \mnorm{\lr{2^{j(s+\frac{d}{2}-\frac{d}{p}+\alpha)} \lr{2^{j(\mu+d)}}^{\frac{1}{p}} \chi_{j,t}}_{j\in\NN,t\in T_j}}_{\ell_q}\\
&\gtrsim \mnorm{\lr{2^{j(s+\frac{d}{2}+\alpha+\frac{\mu}{p})} \chi_{j,t}}_{j\in\NN,t\in T_j}}_{\ell_q}.
\end{align*}
If $q=\infty$, let $t^\ast=(M,\ldots,M)$ and note that the events $E_j\defeq\setn{\chi_{j,t^\ast}=1}$ for $j\in\NN$ are independent.
Furthermore, it holds $\sum_{j\in\NN}\PP(E_j)\geq \sum_{j\in\NN}\kappa =\infty$. 
The Borel--Cantelli lemma thus yields that almost surely, the set $I\defeq\setcond{j\in\NN}{\chi_{j,t^\ast}=1}$ is an infinite set and 
\begin{equation*}
\infty \stackrel{a.s.}{>} \mnorm{\lr{2^{j(s+\frac{d}{2}+\alpha+\frac{\mu}{p})} \chi_{j,t}}_{j\in\NN,t\in T_j}}_{\ell_q}\geq \sup_{j\in I}2^{j(s+\frac{d}{2}+\alpha+\frac{\mu}{p})}.
\end{equation*}
This shows that $s+\frac{d}{2}+\alpha +\frac{\mu}{p}\leq 0$.
On the other hand, if $q\neq \infty$, it follows from \cref{thm:workhorse-lemma} that
\begin{align*}
\infty&>\sum_{j\in\NN}\sum_{t\in T_j}\EE\left[\min\setn{1,2^{jq(s+\frac{d}{2}+\alpha+\frac{\mu}{p})} \chi_{j,t}}\right]\\
&\geq\sum_{j\in\NN}\sum_{t\in T_j}\EE\left[\chi_{j,t}\min\setn{1,2^{jq(s+\frac{d}{2}+\alpha+\frac{\mu}{p})} }\right]\\
&=\sum_{j\in\NN}\sum_{t\in T_j} \EE[\chi_{j,t}]\min\setn{1,2^{jq(s+\frac{d}{2}+\alpha+\frac{\mu}{p})} }\\
&\geq\kappa\sum_{j\in\NN} \min\setn{1,2^{jq(s+\frac{d}{2}+\alpha+\frac{\mu}{p})} },
\end{align*}
and thus $s+\frac{d}{2}+\alpha+\frac{\mu}{p}< 0$.\\[\baselineskip]
\emph{Step 2b: When $p=\infty$, we show the remaining part of Property A'.}
Suppose towards a contradiction that $\hf>0$.
Note because of $\widetilde{\eta}_2\stackrel{a.s.}{<}\infty$ that with $t^\ast=(M,\ldots,M)$, we have
\begin{equation*}
\infty\stackrel{a.s.}{>}\widetilde{Z}_{1,t^\ast} \geq R^{-1}\sup_{m\in\ZZ^d}\lr{1+\frac{\mnorm{m}_\infty}{2}}^\hf \lambda_{1,t^\ast,m} X_{1,t^\ast,m}.
\end{equation*}
But then it holds almost surely that $\lambda_{1,t^\ast,m} X_{1,t^\ast,m}=1$ for only finitely many $m\in\ZZ^d$.
Since the random variables $\lambda_{1,t^\ast,m} X_{1,t^\ast,m}$ for $m\in\ZZ^d$ are independent, the Borel--Cantelli lemma thus shows that 
\begin{align*}
\infty&>\sum_{m\in\ZZ^d}\PP(\lambda_{1,t^\ast,m} X_{1,t^\ast,m}=1)\\
&=\sum_{m\in\ZZ^d}\PP(\lambda_{1,t^\ast,m}=1)\PP(X_{1,t^\ast,m}=1)\\
&=\sum_{m\in\ZZ^d}2^{\mu}\lr{1+\frac{\mnorm{m}_\infty}{2}}^\nu \PP(X_{1t^\ast,m}=1),
\end{align*}
which is absurd since $\PP(X_{1,t^\ast,m}=1)=\PP(|\xi_{1,t^\ast,m}|\geq \frac{1}{R})>0$ is independent of $m\in\ZZ^d$, and $\nu\geq -d$, see \cref{thm:count-lattice-points}.
Hence $\hf \leq 0$.

Further with $X^\ast_{j,t,m}\defeq\one_{\frac{1}{R}\leq\abs{\xi_{j,t,m}}\leq R}$ as before, we have for all $j\in\NN$ and $t\in T_j$ that
\begin{align*}
\widetilde{Z}_{j,t}&\geq \sup_{\substack{m\in\ZZ^d:\\\mnorm{m}_\infty\leq 2^j}}\lr{1+\frac{\mnorm{m}_\infty}{2^j}}^\hf \lambda_{j,t,m}\abs{\xi_{j,t,m}} \\
&\gtrsim \sup_{\substack{m\in\ZZ^d:\\\mnorm{m}_\infty\leq 2^j}}\lambda_{j,t,m}X_{j,t,m}^\ast \eqdef \widetilde{U}_{j,t}\in \setn{0,1},
\end{align*}
which in turn implies
\begin{equation}\label{eq:p-infty-bernoulli-necessary}
\infty \stackrel{a.s.}{>}\widetilde{\eta}_2\geq \mnorm{\lr{2^{j(s+\frac{d}{2}+\alpha)}\widetilde{U}_{j,t^\ast}}_{j\in\NN}}_{\ell_q}
\end{equation}
where again $t^\ast=(M,\ldots,M)$.
Let us note that \eqref{eq:paley-zygmund-bernoulli} and \eqref{eq:def-ajt} yield the existence of $c>0$ such that
\begin{equation*}
\EE[\widetilde{U}_{j,t^\ast}]=\PP(\widetilde{U}_{j,t^\ast}=1)\geq \PP(A_{jt^\ast}\neq 0)\geq \PP(A_{j,t^\ast}\geq c 2^{j(\mu+ d)}) \geq \kappa>0.
\end{equation*}
If $q=\infty$, then the Borel--Cantelli lemma and \eqref{eq:p-infty-bernoulli-necessary} imply as in Step 1b that $s+\frac{d}{2}+\alpha\leq 0$.
On the other hand, if $q\neq \infty$, then \eqref{eq:p-infty-bernoulli-necessary} also implies $\infty \stackrel{a.s.}{>}\mnorm{\lr{2^{j(s+\frac{d}{2}+\alpha)}\widetilde{U}_{j,t^\ast}}_{j\in\NN}}_{\ell_\infty}$ which in turn yields $s+\frac{d}{2}+\alpha\leq 0$.
Moreover, \cref{thm:workhorse-lemma} gives
\begin{align*}
\infty >\sum_{j\in\NN}\EE\left[\min\setn{1,2^{jq(s+\frac{d}{2}+\alpha)}\widetilde{U}_{j,t^\ast}}\right]=\sum_{j\in\NN} 2^{j(s+\frac{d}{2}+\alpha)}\EE[\widetilde{U}_{j,t^\ast}].
\end{align*}
As $\EE[\widetilde{U}_{j,t^\ast}]\geq\kappa$ is independent of $j\in\NN_0$, it follows that $s+\frac{d}{2}+\alpha<0$.
\end{proof}

\begin{Satz}\label{thm:necessity-bernoulli-functions}
Let $d,k\in\NN$, $s,\alpha,\hf,\lf \in\RR$, $p,q\in (0,\infty]$, and $\mu,\nu \in [-d,0]$.
Assume that $f$ follows the \priortwo{} prior with parameters $(\alpha,\hf,\lf,\mu,\nu,k)$ and template random variable $X$ with
\begin{equation*}
\begin{cases} \ex \eps>0:\,\EE\bigl[\abs{X}^{p(1+\eps)}\bigr]< \infty&\text{if }p\in (0,\infty),\\
\ex  R>0:\, \abs{X}\stackrel{a.s.}{\leq}R&\text{if }p=\infty\end{cases}
\end{equation*}
and
\begin{equation*}
\PP(X\neq 0)>0.
\end{equation*}
If $f \in B^s_{p,q}(\RR^d)$ almost surely and if Properties B and B' are satisfied, then also Property A' is satisfied.
\end{Satz}
\begin{proof}
From Property B, it follows that $d\lr{\frac{1}{p}-1}-k<s<k$ if $p\in (0,1)$ and $-k<s<k$ else.
Taking Property B' into account, we may choose $\widetilde{s}\in\RR$ such that $d\lr{\frac{1}{p}-1}-k<\widetilde{s}<s$ and $\widetilde{s}+\frac{d}{2}+\alpha<0$ when $p\in (0,1)$.
On the other hand, if $p\in [1,\infty]$, we may choose  $\widetilde{s}\in\RR$ such that $-k<\widetilde{s}<s$ and $\widetilde{s}+\frac{d}{2}+\alpha<0$.
Note that then Property B is satisfied with $\widetilde{s}$ in place of $s$.
Choose $\sigma<0$ small enough such that $\lf+\sigma +\frac{d+\nu}{p} <0$ and $\hf+\sigma+\frac{d+\nu}{p}<0$.
Since \eqref{eq:finite-moments} holds, we conclude from \cref{thm:powers-bernoulli} with $r=p$ that $\bar{a}=(\bar{a}_{j,t,m})_{(j,t,m)\in J}\in b^{\widetilde{s}}_{p,q}(\RR^d)$ almost surely where $\lambda_{j,t,m}$ and $\xi_{j,t,m}$ are the random variables from \eqref{eq:expansion-bernoulli} and
\begin{equation*}
\bar{a}_{j,t,m}\defeq 2^{j\alpha} \lr{1+\frac{\mnorm{m}_\infty}{2^j}}^{\hf+\sigma +(\lf-\hf)\one_{j=0}}\lambda_{j,t,m}\xi_{j,t,m}.
\end{equation*}
With $w:\RR^d\to\RR$, $w_\sigma(x)=(1+\mnorm{x}_2^2)^{\frac{\sigma}{2}}$, we then see in view of \cref{remark:weights} that $a=(a_{j,t,m})_{(j,t,m)\in J}\in b^{\widetilde{s}}_{p,q}(\RR^d,w_\sigma)$ where
\begin{equation*}
a_{j,t,m}=2^{j\alpha} \lr{1+\frac{\mnorm{m}_\infty}{2^j}}^{\hf+(\lf-\hf)\one_{j=0}}\lambda_{j,t,m}\xi_{j,t,m}.
\end{equation*}
If $f\in B^s_{p,q}(\RR^d)$ almost surely, then almost surely $f=S(\widetilde{a})$ for some $\widetilde{a}\in b^s_{p,q}(\RR^d)$ by \cref{thm:wavelet-characterization}.
From $\sigma<0$ and $s>\widetilde{s}$, we get the embeddings
\begin{equation*}
b^s_{p,q}(\RR^d)\hookrightarrow b^s_{p,q}(\RR^d,w_\sigma)\hookrightarrow b^{\widetilde{s}}_{p,q}(\RR^d,w_\sigma).
\end{equation*} 
But this implies $a,\widetilde{a}\in  b^{\widetilde{s}}_{p,q}(\RR^d,w_\sigma)$ almost surely with $S(a)=f=S(\widetilde{a})$, from which it follows that almost surely $a=\widetilde{a}\in b^s_{p,q}(\RR^d)$ by \cref{thm:wavelet-characterization}.
From \cref{thm:necessity-bernoulli}, we conclude that Property A' holds.
\end{proof}

\section{Discussion of the conditions on the random variables}\label{chap:discussion}
When $p\neq\infty$, the condition \eqref{eq:finite-moments} in \cref{thm:target-statement-sufficiency} can be replaced by a weaker condition.
In the following, let $\log_2^+(x)\defeq \max\setn{0,\log(x)}$.
\begin{Satz}\label{thm:sufficiency-log}
Let $d,k\in\NN$, $s,\alpha,\hf,\lf,\theta\in\RR$, and $p,q\in (0,\infty]$.
Assume that $a$ follows a Besov sequence prior and $f$ follows a Besov prior, both with parameters $(\alpha,\hf,\lf,\theta,k)$ and template random variable $X$, satisfying
\begin{equation}\label{eq:log-condition}
\begin{cases} \ex \eps>0:\,\EE\bigl[\abs{X}^p \cdot \log_2^+(\abs{X})\bigr]< \infty&\text{if }p\in (0,\infty),\\
\ex  R>0:\, \abs{X}\stackrel{a.s.}{\leq}R&\text{if }p=\infty.\end{cases}
\end{equation}
If Property A holds, then $a$ almost surely belongs to $b^s_{p,q}(\RR^d)$.
Moreover, if Properties A and B are satisfied, then $f$ belongs almost surely to $B^s_{p,q}(\RR^d)$.
\end{Satz}
\begin{proof}
The proof follows the same lines as the one of \cref{thm:target-statement-sufficiency} but we need a new argument for $\Xi\stackrel{a.s.}{<}\infty$, with $\Xi$ from \cref{thm:master-lemma} when $p\neq \infty$.
This is done by invoking \cref{thm:UniformLLN-new} for $X_{j,\tau}=\abs{\xi_{j,t,\phi(\tau)}}^p$, where again $\phi:\NN\to\ZZ^d$ is a bijection such that $\tau\mapsto \mnorm{\phi(\tau)}_\infty$ is nondecreasing.
Regarding the applicability of \cref{thm:UniformLLN-new}, note that
\begin{equation*}
\EE\left[X_{j,\tau}\cdot\log_2^+(X_{j,\tau})\right]=\EE\left[\abs{\xi_{j,t,\phi(\tau)}}^p\cdot\log_2^+(\abs{\xi_{j,t,\phi(\tau)}}^p)\right]=p\EE\left[\abs{X}^p\cdot\log_2^+(\abs{X})\right].
\end{equation*}
\end{proof}

The preceding theorem raises the question of the sharpness of conditions \eqref{eq:finite-moments} and \eqref{eq:log-condition}.
We address the case $p\neq\infty$ first.
He we only consider the condition $a=(a_{j,t,m})_{(j,t,m)\in J}\in b^s_{p,q}(\RR^d)$ almost surely, with $a_{j,t,m}$ as defined in \eqref{eq:expansion-seq}, instead of the condition $f\in B^s_{p,q}(\RR^d)$ almost surely,
See the remark preceding \cref{thm:target-statement-necessity-functions} for a discussion of the relation between these two conditions.
\begin{Prop}\label{thm:sharpness}
Let $d\in\NN$, $s,\alpha,\hf,\lf,\theta\in\RR$, $p\in (0,\infty)$, and $q\in (0,\infty]$.
Assume that $=(a_{j,t,m})_{(j,t,m)\in J}$ follows a Besov sequence prior with parameters $(\alpha,\hf,\lf,\theta)$ and template random variable $X$ satisfying $\PP(X\neq 0)>0$.
If $a\in b^s_{p,q}(\RR^d)$ almost surely, then $\gamma<0$ and $\EE[\abs{X}^{-\frac{d}{\lf}}]<\infty$.
\end{Prop}
\begin{proof}
Let $\widetilde{\lf}\defeq \min\setn{0,\lf}$.
Since $a\in b^s_{p,q}(\RR^d)$ almost surely, we know that also
\begin{equation*}
\mnorm{\lr{(1+\mnorm{m}_\infty)^{\widetilde{\lf}} \abs{\xi_m}}_{m\in\ZZ^d}}_{\ell_p}\leq \mnorm{\lr{(1+\mnorm{m}_\infty)^\lf \abs{\xi_m}}_{m\in\ZZ^d}}_{\ell_p}\stackrel{a.s.}{<}\infty.
\end{equation*}
By \cref{thm:workhorse-lemma}, this implies
\begin{equation*}
\infty >\sum_{m\in\ZZ^d}\EE\left[\min\setn{1,\lr{1+\mnorm{m}_\infty}^{\widetilde{\lf} p}\abs{\xi_m}^p}\right].
\end{equation*}
If we had $\widetilde{\lf}=0$, then this would imply
\begin{equation*}
\infty>\sum_{m\in\ZZ^d}\EE\left[\min{1,\abs{\xi_m}^p}\right]=\sum_{m\in\ZZ^d}\EE\left[\min{1,\abs{X}^p}\right]=\infty
\end{equation*}
since $\PP(\min\setn{1,\abs{X}^p}>0)=\PP(X\neq 0)>0$.
Hence $0\neq\widetilde{\lf}$ thus $\widetilde{\lf}=\lf<0$.
Next, we get
\begin{align}
\infty&>\sum_{m\in\ZZ^d}\EE\left[\min\setn{1,\lr{1+\mnorm{m}_\infty}^{\lf p}\abs{\xi_m}^p}\right]\nonumber\\
&\geq \sum_{m\in\ZZ^d}\EE\left[\lr{1+\mnorm{m}_\infty}^{\lf p} \abs{\xi_m}^p\one_{\abs{\xi_m}^p\leq \lr{1+\mnorm{m}_\infty}^{-\lf p}}\right]\nonumber\\
&= \sum_{m\in\ZZ^d}\EE\left[\lr{1+\mnorm{m}_\infty}^{\lf p} \abs{X}^p\one_{\abs{X}^p\leq \lr{1+\mnorm{m}_\infty}^{-\lf p}}\right]\nonumber\\
&= \EE\left[\abs{X}^p\sum_{m\in\ZZ^d}\lr{1+\mnorm{m}_\infty}^{\lf p} \one_{\abs{X}^p\leq \lr{1+\mnorm{m}_\infty}^{-\lf p}}\right]\nonumber\\
&\geq \EE\left[\abs{X}^p\one_{\abs{X}^p\geq 1}\sum_{\substack{m\in\ZZ^d\\ \mnorm{m}_\infty \geq \abs{X}^{-\frac{1}{\lf}}-1}}\lr{1+\mnorm{m}_\infty}^{\lf p}\right]\nonumber\\
&\asymp\EE\left[\abs{X}^p\one_{\abs{X}^p\geq 1} \lr{\abs{X}^{-\frac{1}{\lf}}}^{\lf p+d}\right]\label{eq:rv-necessary}\\
&=\EE\left[\abs{X}^{-\frac{d}{\lf}}\one_{\abs{X}^p\geq 1}\right].\nonumber
\end{align}
Since $\lf<0$, this implies $\EE[\abs{X}^{-\frac{d}{\lf}}]<\infty$.
The step \eqref{eq:rv-necessary} is justified by the following computation, valid for all $N\in\NN_0$, cf.\ also the proof of \cref{thm:count-lattice-points}:
\begin{align}
\sum_{\substack{m\in\ZZ^d\\\mnorm{m}_\infty\geq N}}(1+\mnorm{m}_\infty)^{\lf p}&=\sum_{\ell\geq N}(1+\ell)^{\lf p}\#\setcond{m\in\ZZ^d}{\mnorm{m}_\infty=\ell}\nonumber\\
&=\sum_{\ell\geq N}(1+\ell)^{\lf p} ((2\ell+1)^d-(2\ell-1)^d)\nonumber\\
&\asymp \sum_{\ell\geq N} (1+\ell)^{\lf p+d-1}\nonumber\\
&=\sum_{\ell\geq N}\int_{\ell}^{\ell+1}(1+\ell)^{\lf p+d-1}\dd x\nonumber\\
&\asymp \sum_{\ell\geq N}\int_{\ell}^{\ell+1}(1+x)^{\lf p+d-1}\dd x\nonumber\\
&= \int_N^\infty (1+x)^{\lf p+d-1}\dd x\nonumber\\
&= \begin{cases} \frac{(N+1)^{\lf p+d}}{\lf p+d}&\text{if }\lf p+d<0\\\infty&\text{else.}\end{cases}\label{eq:justify}
\end{align}
Now, regarding \eqref{eq:rv-necessary}, this computation shows that if $\gamma p+d\geq 0$, then
\begin{equation*}
\infty >\EE\left[\abs{X}^p\one_{\abs{X}^p\geq 1}\sum_{\substack{m\in\ZZ^d\\ \mnorm{m}_\infty \geq \abs{X}^{-\frac{1}{\lf}}-1}}\lr{1+\mnorm{m}_\infty}^{\lf p}\right] =\infty \cdot \EE\left[\abs{X}^p\one_{\abs{X}^p\geq 1}\right], 
\end{equation*}
and hence $\EE\left[\abs{X}^p\one_{\abs{X}^p\geq 1}\right]=0$ which shows that $\abs{X}<1$ almost surely, and hence \eqref{eq:rv-necessary} holds.
Of otherwise $\gamma p+d<0$, then the computation in \eqref{eq:justify} justifies \eqref{eq:rv-necessary} directly.
\end{proof}
\begin{Bem}\label{bem:sharpness-almost-sure-convergence}
Let $d\in\NN$, $s,\alpha,\hf,\lf,\theta\in\RR$, and $p,q\in (0,\infty]$.
Assume that $a$ follows a Besov sequence prior with parameters $(\alpha,\hf,\lf,\theta)$ and template random variable $X$ satisfying $\PP(X\neq 0)>0$.
If $a\in b^s_{p,q}(\RR^d)$ almost surely whenever Property A is satisfied, then \cref{thm:sharpness} implies that $\EE[\abs{X}^{-\frac{d}{\lf}}]<\infty$ whenever $\lf p<-d$, which in turn yields $\EE[\abs{X}^\rho]<\infty$ for all $\rho\in (0,p)$.
This shows that the conditions \eqref{eq:finite-moments} and \eqref{eq:log-condition} are quite sharp if $p\neq \infty$.
\end{Bem}

Next, we address the sharpness of the conditions \eqref{eq:finite-moments} and \eqref{eq:log-condition} when $p=\infty$.
\begin{Prop}\label{thm:essential-boundedness}
Let $d\in\NN$, $s,\alpha,\hf,\lf,\theta\in\RR$, $p=\infty$, and $q\in (0,\infty]$.
Assume that $a$ follows a Besov sequence prior with parameters $(\alpha,\hf,\lf,\theta)$ and template random variable $X$.
Suppose that $X$ is such that $a\in b^s_{\infty,q}(\RR^d)$ almost surely.
If $\hf=0$ or $\lf=0$, then there exists $R>0$ such that $\abs{X}\leq R$ almost surely.
\end{Prop}
\begin{proof}
Assume that $\gamma=0$ and that for all $R>0$, we have $\PP(\abs{X}\geq R)>0$.
Then $\sum_{m\in\ZZ^d} \PP(\abs{\xi_m}\geq R)=\infty$.
Thus, the Borel--Cantelli lemma shows for every $R\in\NN$ that the event $A_R$ on which $\sup_{m\in\ZZ^d}\abs{\xi_m}\geq R$ has probability $\PP(A_R)=1$.
Hence also $\PP(A)=1$ for $A\defeq\bigcap_{R\in\RR}A_R$.
But on the event $A$ we have $\sup_{m\in\ZZ^d}\abs{\xi_m}=\infty$, from which $\mnorm{a}_{b^s_{\infty,q}(\RR^d)}=\infty$ follows.
Similar arguments apply when $\hf=0$.
\end{proof}
The following remark shows that when $p$ is infinite but $q$ is not, then the essential boundedness of the template variable also cannot be simply droppend in case of $s+\frac{d}{2}+\alpha=0$.
\begin{Bem}\label{bem:essential-boundedness}
Let $d\in\NN$, $s,\alpha,\hf,\lf,\theta\in\RR$, $p=\infty$, and $q\in (0,\infty]$.
Let further $X\sim \mathcal{N}(0,1)$ is standard normal random variable, and assume that $a$ follows a Besov sequence prior with parameters $(\alpha,\hf,\lf,\theta)$ and template variable $X$.
Suppose that $s+\frac{d}{2}+\alpha=0$ and $\theta\geq -\frac{1}{2}-\frac{1}{q}$.
Then almost surely, we have $a\notin b^s_{\infty,q}(\RR^d)$.

This shows that even for $\hf,\lf\ll 0$, Property A (which in the present setting reduces to $\hf\leq 0$, $\lf\leq 0$,
 and $\theta <-\frac{1}{q}$) is in general \emph{not} sufficient to ensure that $a\in b^s_{p,q}(\RR^d)$ almost surely in the case $p=\infty$ if one drops condition \eqref{eq:finite-moments}.
\end{Bem}
\begin{proof}
Fix $t^\ast=(M,\ldots,M)$.
It follows from \cref{thm:SupSTDNorm} that there exists an $\NN$-valued random variable $N$ satisfying 
\begin{equation*}
\sup_{\substack{m\in\ZZ^d:\\\mnorm{m}_\infty\leq 2^j}}\abs{\xi_{j,t^\ast,m}}\geq \sqrt{\log((2^{j+1}+1)^d)}\geq c\sqrt{j}
\end{equation*}
for all $j\geq N$ and some absolute constant $c>0$.
Hence, with 
\begin{equation*}
\eta_2\defeq \mnorm{(2^{j(s+\frac{d}{2}+\alpha)}(j+1)^\theta Z_{j,t})_{j\in\NN, t\in T_j}}_{\ell_q},
\end{equation*}
we have
\begin{align*}
\mnorm{a}_{b^s_{\infty,q}(\RR^d)}^q\gtrsim \eta_2^q&\geq \sum_{j\in\NN} (j+1)^{\theta q}\sup_{\substack{m\in\ZZ^d:\\\mnorm{m}_\infty\leq 2^j}}\lr{1+\frac{\mnorm{m}_\infty}{2^j}}^{\hf q}\abs{\xi_{j,t^\ast,m}}^q\\
&\gtrsim \sum_{j\in\NN} (j+1)^{\theta q}\lr{\sup_{\substack{m\in\ZZ^d:\\\mnorm{m}_\infty\leq 2^j}}\abs{\xi_{j,t^\ast,m}}}^q\\
&\gtrsim \sum_{j\geq N}j^{(\theta+\frac{1}{2})q}\\
&=\infty,
\end{align*}
because $\theta\geq -\frac{1}{2}-\frac{1}{q}$.
\end{proof}

Similarly to \cref{bem:sharpness-almost-sure-convergence}, we can discuss the sharpness of the condition \eqref{eq:rv-powers} when $p\neq\infty$.
\begin{Bem}
Let $d\in\NN$, $s,\alpha,\hf,\lf,\theta\in\RR$, $p,r\in (0,\infty)$, and $q\in (0,\infty]$.
Let further $X$ be a random variable with $\PP(X\neq 0)>0$, and assume that $a$ follows a Besov sequence prior with parameters $(\alpha,\hf,\lf,\theta)$ and template random variable $X$.
Suppose that $X$ is such that $\EE[\mnorm{a}_{b^s_{p,q}(\RR^d)}^r]<\infty$ whenever Property A is satisfied.
On the one hand, this implies that $\mnorm{a}_{b^s_{p,q}(\RR^d)}< \infty$ almost surely whenever Property A is satisfied.
As seen in \cref{bem:sharpness-almost-sure-convergence}, this gives $\EE[\abs{X}^\rho]<\infty$  for all $\rho\in (0,p)$.
On the other hand, it holds that $\mnorm{a}_{b^s_{p,q}(\RR^d)}\gtrsim\abs{\xi_0}$.
Thus, if $\EE[\mnorm{a}_{b^s_{p,q}(\RR^d)}^r]<\infty$, then necessarily $\EE[\abs{\xi_0}^r]<\infty$.
Overall, this shows that $\EE[\abs{X}^{\max\setn{r,\rho}}]<\infty$ for all $\rho\in (0,p)$.
Hence \eqref{eq:rv-powers} is quite sharp if $p\neq\infty$.
\end{Bem}

\appendix
\section{Auxiliary statements}\label{chap:appendix}

The following lemma shows the admissibility of the function $w_\sigma:\RR^d\to\RR$ given by $w_\sigma(x)\defeq (1+\mnorm{x}_2^2)^{\frac{\sigma}{2}}$ as defined in \cref{def:fourier}.
\begin{Lem}\label{thm:admissibility}
Let $\sigma\in\RR$ and $w_\sigma:\RR^d\to\RR$, $w_\sigma(x)\defeq (1+\mnorm{x}_2^2)^{\frac{\sigma}{2}}$.
\begin{enumerate}[label={(\alph*)},leftmargin=*,align=left,noitemsep]
\item{For every multiindex $\beta\in\NN_0^d$, we have
\begin{equation*}
D^\beta (1+\mnorm{x}_2^2)^{\frac{\sigma}{2}}=p_{\sigma,\beta}(x)(1+\mnorm{x}_2^2)^{\frac{\sigma}{2}-\mnorm{\beta}_1}
\end{equation*}
where $p_{\sigma,\beta}(x)$ is a polynomial of degree at most $\mnorm{\beta}_1$.
In particular, we have $\abs{D^\beta (1+\mnorm{x}_2^2)^{\frac{\sigma}{2}}}\lesssim (1+\mnorm{x}_2^2)^{\frac{\sigma}{2}}$.\label{derivatives}}
\item{We have $w_\sigma(x)\lesssim w_{\abs{\sigma}}(x-y)w_\sigma(y)$ for all $x,y\in\RR^d$.\label{upper-bound}}
\end{enumerate}
\end{Lem}
\begin{proof}
We show \ref{derivatives} by induction.
If $\beta=0$, then we can choose $p_{\sigma,\beta}(x)=1$ which is indeed a polynomial of degree at most $\mnorm{\beta}_1=0$.
Now assume the claim holds true for $\beta$.
We show that it holds true for $\beta+e_k$ in place of $\beta$, where $e_k\in\RR^d$ denotes the vector whose entries are $0$ except for the $k$th, which is $1$.
Note that
\begin{align*}
&D^{\beta+e_k}(1+\mnorm{x}_2^2)^{\frac{\sigma}{2}}\\
&=D^{e_k}\lr{p_{\sigma,\beta}(x)(1+\mnorm{x}_2^2)^{\frac{\sigma}{2}-\mnorm{\beta}_1}}\\
&=\lr{D^{e_k}p_{\sigma,\beta}(x)}(1+\mnorm{x}_2^2)^{\frac{\sigma}{2}-\mnorm{\beta}_1}+p_{\sigma,\beta}(x)\lr{D^{e_k}(1+\mnorm{x}_2^2)^{\frac{\sigma}{2}-\mnorm{\beta}_1}}\\
&=\lr{D^{e_k}p_{\sigma,\beta}(x)}(1+\mnorm{x}_2^2)^{\frac{\sigma}{2}-\mnorm{\beta}_1}+p_{\sigma,\beta}(x)\lr{\frac{\sigma}{2}-\mnorm{\beta}_1}2x_k(1+\mnorm{x}_2^2)^{\frac{\sigma}{2}-\mnorm{\beta}_1-1}\\
&=(1+\mnorm{x}_2^2)^{\frac{\sigma}{2}-\mnorm{\beta}_1-1}\lr{\lr{D^{e_k}p_{\sigma,\beta}(x)}(1+\mnorm{x}_2^2)+p_{\sigma,\beta}(x)\lr{\frac{\sigma}{2}-\mnorm{\beta}_1}2x_k},
\end{align*}
and $\lr{D^{e_k}p_{\sigma,\beta}(x)}(1+\mnorm{x}_2^2)+p_{\sigma,\beta}(x)\lr{\frac{\sigma}{2}-\mnorm{\beta}_1}2x_k$ is a polynomial of degree at most $\mnorm{\beta+e_k}_1=\mnorm{\beta}_1+1$.

For the proof of \ref{upper-bound}, note that the case $\sigma=0$ is trivial. 
Further, we observe that
\begin{equation*}
1+\mnorm{x}_2\leq 1+ \mnorm{x-y}_2+\mnorm{y}_2\leq \lr{1+\mnorm{x-y}_2}\lr{1+\mnorm{y}_2}.
\end{equation*}
If $\sigma>0$, it follows that
\begin{equation*}
w_\sigma(x)\lesssim (1+\mnorm{x}_2)^\sigma\leq (1+\mnorm{x-y}_2)^\sigma (1+\mnorm{y}_2)^\sigma\lesssim w_\sigma(x-y)w_\sigma(y).
\end{equation*}
On the other hand, if $\sigma<0$, then
\begin{align*}
w_\sigma(x)&\asymp (1+\mnorm{x}_2)^\sigma =((1+\mnorm{x}_2)^{-1})^{-\sigma}\\
&\leq ((1+\mnorm{x-y}_2)(1+\mnorm{y}_2)^{-1})^{-\sigma}\\
&=(1+\mnorm{x-y}_2)^{-\sigma}(1+\mnorm{y}_2)^\sigma\\
&\asymp w_{-\sigma}(x-y)(1+\mnorm{y}_2)^\sigma\\
&\asymp w_{\abs{\sigma}}(x-y)w_\sigma(y).\qedhere
\end{align*}
\end{proof}

The following lemma is taken from \cite[Proposition~5.14]{Kallenberg2021}.
We use it in in the proofs of \cref{thm:target-statement-necessity,thm:necessity-bernoulli}.
\begin{Lem}\label{thm:workhorse-lemma}
Let $(\xi_k)_{k\in\NN}$ be a sequence of independent nonnegative (finite) random variables.
Then $\sum_{k\in\NN} \xi_k<\infty$ almost surely if and only if
\begin{equation*}
\sum_{k\in\NN} \EE[\min\setn{\xi_k,1}]<\infty.
\end{equation*}
\end{Lem}

The following calculation is used in the proofs of \cref{thm:target-statement-necessity,thm:necessity-bernoulli}.
\begin{Lem}\label{thm:count-lattice-points}
Let $d,a\in\NN$, $\beta\in\RR$, and $p\in (0,\infty)$.
We have
\begin{equation*}
\sum_{m\in\ZZ^d}\lr{1+\frac{\mnorm{m}_\infty}{a}}^{\hf p}<\infty
\end{equation*}
if and only if $\hf< -\frac{d}{p}$, in which case $\sum_{m\in\ZZ^d}\lr{1+\frac{\mnorm{m}_\infty}{a}}^{\hf p}\asymp a^d$, with the implied constants only depending on $\beta$, $p$, and $d$. 
\end{Lem}
\begin{proof}
Note for $\ell\in\NN$ that
\begin{align*}
&\#\setcond{m\in\ZZ^d}{\mnorm{m}_\infty=\ell}\\
&=\#\setcond{m\in\ZZ^d}{\mnorm{m}_\infty\leq\ell}-\#\setcond{m\in\ZZ^d}{\mnorm{m}_\infty\leq\ell-1}\\
&=(2\ell+1)^d-(2(\ell-1)+1)^d=(2\ell+1)^d-(2\ell-1)^d.
\end{align*}
Therefore, we see
\begin{align*}
\sum_{m\in\ZZ^d}\lr{1+\frac{\mnorm{m}_\infty}{a}}^{\hf p}&=1+\sum_{\ell\in\NN}\sum_{\substack{m\in\ZZ^d\\\mnorm{m}_\infty=\ell}}\lr{1+\frac{\mnorm{m}_\infty}{a}}^{\hf p}\\
&=1+\sum_{\ell\in\NN} \lr{(2\ell+1)^d-(2\ell-1)^d}\lr{1+\frac{\ell}{a}}^{\hf p}.
\end{align*}
Next, note that thanks to the mean value theorem, there exists a number $y=y_\ell\in (2\ell-1,2\ell+1)$ such that $(2\ell+1)^d-(2\ell-1)^d=d\cdot y_\ell^{d-1}\cdot 2\asymp \ell^{d-1}$.
This entails the estimate
\begin{align*}
&\sum_{\ell\in\NN} ((2\ell+1)^d-(2\ell-1)^d)\lr{1+\frac{\ell}{a}}^{\hf p}\\
&\asymp \sum_{\ell\in\NN} \ell^{d-1}\lr{1+\frac{\ell}{a}}^{\hf p}\\
&=\sum_{\ell=1}^{a-1}\ell^{d-1}\lr{1+\frac{\ell}{a}}^{\hf p}+\sum_{\ell\geq a} \ell^{d-1}\lr{1+\frac{\ell}{a}}^{\hf p}.
\end{align*}
In the first sum, we have $1\leq 1+\frac{\ell}{a}\leq 2$ and thus $\sum_{\ell=1}^{a-1}\ell^{d-1}\lr{1+\frac{\ell}{a}}^{\hf p}\asymp\sum_{\ell=1}^{a-1}\ell^{d-1}$.
In the second sum, we have $\frac{\ell}{a}\leq 1+\frac{\ell}{a}\leq 2\frac{\ell}{a}$ and thus $\sum_{\ell\geq a} \ell^{d-1}\lr{1+\frac{\ell}{a}}^{\hf p}\asymp \sum_{\ell\geq a} a^{-\hf p}\ell^{d-1+\hf p}$.
Next, note for any $\nu\in\RR$ and $\ell\in\NN$, that if $\ell\leq x\leq \ell +1$, then $\ell\leq x\leq 2\ell$.
Hence $x\asymp \ell$ and also $x^\nu\asymp \ell^\nu$.
It follows that
\begin{equation*}
\sum_{\ell=1}^{a-1}\ell^\nu=\sum_{\ell=1}^{a-1}\int_{\ell}^{\ell+1}\ell^\nu \dd x\asymp\sum_{\ell=1}^{a-1}\int_{\ell}^{\ell+1}x^\nu \dd x=\int_1^{a}x^\nu\dd x
\end{equation*}
and, similarly, 
\begin{equation*}
\sum_{\ell\geq a}\ell^\nu=\sum_{\ell\geq a}\int_{\ell}^{\ell+1}\ell^\nu \dd x\asymp\sum_{\ell\geq a}\int_{\ell}^{\ell+1}x^\nu \dd x=\int_a^\infty x^\nu\dd x.
\end{equation*}
Summarizing, we see
\begin{align*}
&\sum_{\ell=1}^{a-1}\ell^{d-1}\lr{1+\frac{\ell}{a}}^{\hf p}+\sum_{\ell\geq a} \ell^{d-1}\lr{1+\frac{\ell}{a}}^{\hf p}\\
&\asymp \sum_{\ell=1}^{a-1}\ell^{d-1} +\sum_{\ell=a}^\infty a^{-\hf p}\ell^{d-1+\hf p}\\
&\asymp \int_1^ax^{d-1}\dd x +a^{-\hf p}\int_a^\infty x^{d-1+\hf p}\dd x,
\end{align*}
where the right-hand side is finite if and only if $\hf p <-d$.
In this case, we get
\begin{equation*}
 \int_1^ax^{d-1}\dd x +a^{-\hf p}\int_a^\infty x^{d-1+\hf p}\dd x=\frac{a^d-1}{d}-a^{-\hf p}\frac{a^{d+\hf p}}{d+\hf p}\asymp a^d
\end{equation*}
so that $\sum_{m\in\ZZ^d}\lr{1+\frac{\mnorm{m}_\infty}{a}}^{\hf p}\asymp 1+a^d\asymp a^d$.
\end{proof}

A proof of the Paley--Zygmund inequality stated below can be found in \cite[Lemma~5.1]{Kallenberg2021}.
We use it in the proofs of \cref{thm:target-statement-necessity,thm:necessity-bernoulli}.
\begin{Lem}[Paley--Zygmund inequality]\label{thm:paley-zygmund}
Let $X\geq 0$ a random variable with $0<\EE[X^2]<\infty$, and let $\sigma\in [0,1]$.
Then
\begin{equation*}
\PP(X>\sigma \EE[X])\geq (1-\sigma)^2\frac{\EE[X]^2}{\EE[X^2]}.
\end{equation*}
\end{Lem}

The following lemma provides bounds for the raw moments for the binomial distribution.
We use it in the proof of \cref{thm:necessity-bernoulli}.
A related bound can be found in \cite[(4.5)]{Knoblauch2008}.
\begin{Lem}\label{thm:binomial}
Let $n\in\NN$, $\rho\in (0,1)$, $\sigma\in (0,\infty)$, and $X\sim \mathrm{Bin}(n,\rho)$ a binomially distributed random variable.
Then $\EE[X^\sigma]\lesssim \max\setn{n\rho,(n\rho)^\sigma}$, where the implied constant only depends on $\sigma$.
\end{Lem}
\begin{proof}
If $\sigma\leq 1$, we trivially have $\EE[X^\sigma]\leq \EE[X]=n\rho$.
Hence we may assume that $\sigma>1$ in what follows. 
If $n\rho\leq 1$, let $c\defeq \ceil{\sigma}\in\NN$.
With $S(c,k)$ and $n^{\underline{k}}$ denoting the Stirling numbers of the second kind and falling factorials, respectively, we obtain from \cite[(2.9) and (3.1)]{Knoblauch2008} and because of $S(c,0)=0$ that
\begin{equation*}
\EE[X^\sigma]\leq \EE[X^c]=\sum_{k=1}^c S(c,k) n^{\underline{k}}\rho^k\leq\sum_{k=1}^c S(c,k) (n \rho)^k\leq\sum_{k=1}^c S(c,k) n \rho\lesssim n \rho.
\end{equation*}
On the other hand, if $n\rho>1$, let $c\in\NN$ such that $c\leq \sigma < c+1$.
Thus there exists $\theta \in [0,1]$ such that $\frac{1}{\sigma}=\frac{\theta}{c+1}+\frac{1-\theta}{c}$.
Using \cite[(2.9) and (3.1)]{Knoblauch2008}, we obtain
\begin{equation*}
\EE[X^c]=\sum_{k=1}^c S(c,k)n^{\underline{k}}\rho^k\leq \sum_{k=1}^c S(c,k) (n\rho)^k\leq \sum_{k=1}^c S(c,k) (n\rho)^c\lesssim (n\rho)^c.
\end{equation*}
Thus  $\mnorm{X}_{L_c}=\EE[X^c]^{\frac{1}{c}} \lesssim n\rho$ and, similarly, $\mnorm{X}_{L_{c+1}}=\EE[X^{c+1}]^{\frac{1}{c+1}} \lesssim n\rho$.

Now \cite[Proposition~6.10]{Folland1999} gives
\begin{equation*}
\EE[X^\sigma]=\mnorm{X}_{L_\sigma}^\sigma \leq \lr{\mnorm{X}_{L_c}^{1-\theta}\mnorm{X}_{L_{c+1}}^{\theta}}^{\sigma}\lesssim (n\rho)^{\sigma}
\end{equation*}
and the proof is complete.
\end{proof}

The following lemma is the technical part of the proofs of \cref{thm:target-statement-sufficiency,thm:powers}.
Its proof is loosely based on Etemadi's proof of the strong law of large numbers in \cite[\textsection~12]{BauerProbabilityTheory}.
\begin{Lem}\label{thm:1+eps}
Let $(X_{j,\tau})_{j\in\NN_0,\tau\in\NN}$ be a family of non-negative i.i.d.\ random variables, $d\in\NN$, $\eps\in (0,\infty)$, and $\sigma\in [1,\infty)$.
Suppose that
\begin{equation*}
\EE[X_{j,\tau}^{\sigma(1+\eps)}]<\infty.
\end{equation*}
Let
\begin{align*}
M_j&\defeq \#\setcond{m\in\ZZ^d}{\mnorm{m}_{\infty}\leq 2^j}=(2^{j+1}+1)^d,\\
N_j&\defeq\#\setcond{m\in\ZZ^d}{2^j<\mnorm{m}_\infty\leq 2^{j+1}}=M_{j+1}-M_j
\end{align*}
for $j\in\NN_0$.
Then, for
\begin{equation*}
\Xi\defeq \sup_{j \in \NN_0} \lr{\frac{1}{M_j} \sum_{\tau=1}^{M_j} X_{j,\tau}+\sup_{\ell\geq j} \frac{1}{N_\ell} \sum_{\tau=M_\ell + 1}^{M_{\ell+1}} X_{j,\tau}},
\end{equation*} 
we have $\EE[\Xi^\sigma]<\infty$.
\end{Lem}
\begin{proof}
Define
\begin{align*}
E&\defeq \EE[X_{0,1}^\sigma],\\
\kappa_{j,\tau}&\defeq \begin{cases}2^{\frac{jd}{3}}&\text{if }1\leq \tau \leq M_j,\\2^{\frac{\ell d}{3}}&\text{if }M_\ell+1\leq \tau \leq M_{\ell+1}\text{ for some (unique) }\ell\geq j,\end{cases}\\
E_{j,\tau}&\defeq \EE[X_{j,\tau}^\sigma\one_{X_{j,\tau}^\sigma\leq\kappa_{j,\tau}}].
\end{align*}
From
\begin{align*}
\Xi^\sigma &\leq\lr{\sup_{j\in\NN_0}\frac{1}{M_j}\sum_{\tau=1}^{M_j}X_{j,\tau}+\sup_{j\in\NN_0}\sup_{\ell\geq j}\frac{1}{N_\ell}\sum_{\tau=M_\ell+1}^{M_{\ell+1}}X_{j,\tau}}^\sigma\\
&\lesssim \lr{\sup_{j\in\NN_0}\frac{1}{M_j}\sum_{\tau=1}^{M_j}X_{j,\tau}}^\sigma+\lr{\sup_{j\in\NN_0}\sup_{\ell\geq j}\frac{1}{N_\ell}\sum_{\tau=M_\ell+1}^{M_{\ell+1}}X_{j,\tau}}^\sigma\\
&=\sup_{j\in\NN_0}\lr{\frac{1}{M_j}\sum_{\tau=1}^{M_j}X_{j,\tau}}^\sigma+\sup_{j\in\NN_0}\sup_{\ell\geq j}\lr{\frac{1}{N_\ell}\sum_{\tau=M_\ell+1}^{M_{\ell+1}}X_{j,\tau}}^\sigma\\
&\leq \sup_{j\in\NN_0} \frac{1}{M_j}\sum_{\tau=1}^{M_j}X_{j,\tau}^\sigma +\sup_{j\in\NN_0}\sup_{\ell\geq j}\frac{1}{N_\ell}\sum_{\tau=M_\ell+1}^{M_{\ell+1}}X_{j,\tau}^\sigma\\
&\leq\sup_{j\in\NN_0}\frac{1}{M_j}\sum_{\tau=1}^{M_j}(X_{j,\tau}^\sigma-E)+\sup_{j\in\NN_0}\sup_{\ell\geq j}\frac{1}{N_\ell}\sum_{\tau=M_\ell+1}^{M_{\ell+1}}(X_{j,\tau}^\sigma-E)+2E\\
&\leq\sup_{j\in\NN_0}\frac{1}{M_j}\sum_{\tau=1}^{M_j}(X_{j,\tau}^\sigma-E_{j,\tau})+\sup_{j\in\NN_0}\sup_{\ell\geq j}\frac{1}{N_\ell}\sum_{\tau=M_\ell+1}^{M_{\ell+1}}(X_{j,\tau}^\sigma-E_{j,\tau})+2E\\
&\leq\sum_{j\in\NN_0}\abs{\frac{1}{M_j}\sum_{\tau=1}^{M_j}(X_{j,\tau}^\sigma-E_{j,\tau})}+\sum_{j\in\NN_0}\sum_{\ell\geq j}\abs{\frac{1}{N_\ell}\sum_{\tau=M_\ell+1}^{M_{\ell+1}}(X_{j,\tau}^\sigma-E_{j,\tau})}+2E,
\end{align*}
we obtain
\begin{align*}
\EE[\Xi^\sigma]&\lesssim 2E+\sum_{j\in\NN_0}\EE\left[\abs{\frac{1}{M_j}\sum_{\tau=1}^{M_j}(X_{j,\tau}^\sigma-E_{j,\tau})}\right]\\
&\qquad\quad+\sum_{j\in\NN_0}\sum_{\ell\geq j}\EE\left[\abs{\frac{1}{N_\ell}\sum_{\tau=M_\ell+1}^{M_{\ell+1}}(X_{j,\tau}^\sigma-E_{j,\tau})}\right].
\end{align*}
In order to show that $\EE[\Xi^\sigma]<\infty$, we show that the two infinite sums in the last expression are finite.
Let $X_\ast\defeq X_{0,1}$.
First, we see that
\begin{align*}
&\EE\left[\abs{\frac{1}{M_j}\sum_{\tau=1}^{M_j}(X_{j,\tau}^\sigma-E_{j,\tau})}\right]\\
&=\EE\left[\abs{\frac{1}{M_j}\sum_{\tau=1}^{M_j}(X_{j,\tau}^\sigma\one_{X_{j,\tau}^\sigma\leq \kappa_{j,\tau}}-E_{j,\tau})+\frac{1}{M_j}\sum_{\tau=1}^{M_j}X_{j,\tau}^\sigma\one_{X_{j,\tau}^\sigma>\kappa_{j,\tau}}}\right]\\
&\leq\EE\left[\abs{ \frac{1}{M_j}\sum_{\tau=1}^{M_j}(X_{j,\tau}^\sigma\one_{X_{j,\tau}^\sigma\leq \kappa_{j,\tau}}-E_{j,\tau})}\right]+\frac{1}{M_j}\sum_{\tau=1}^{M_j}\EE\left[\abs{X_{j,\tau}^\sigma\one_{X_{j,\tau}^\sigma>\kappa_{j,\tau}}}\right]\\
&\leq \sqrt{\EE\left[\lr{\frac{1}{M_j}\sum_{\tau=1}^{M_j}(X_{j,\tau}^\sigma\one_{X_{j,\tau}^\sigma\leq \kappa_{j,\tau}}-E_{j,\tau})}^2\right]}+\frac{1}{M_j}\sum_{\tau=1}^{M_j}\kappa_{j,\tau}^{-\eps}\EE\left[\abs{X_{j,\tau}^\sigma}^{1+\eps}\right]\\
&= \sqrt{\frac{1}{M_j^2}\sum_{\tau=1}^{M_j}\VV\left[X_{j,\tau}^\sigma\one_{X_{j,\tau}^\sigma\leq \kappa_{j,\tau}}\right]}+2^{-\frac{jd\eps}{3}}\EE\left[\abs{X_\ast^\sigma}^{1+\eps}\right]\\
&\leq \sqrt{\frac{1}{M_j^2}\sum_{\tau=1}^{M_j}\kappa_{j,\tau}^2}+2^{-\frac{jd\eps}{3}}\EE\left[\abs{X_\ast^\sigma}^{1+\eps}\right]\\
&\leq \sqrt{\frac{2^{\frac{2jd}{3}}}{M_j}}+2^{-\frac{jd\eps}{3}}\EE\left[\abs{X_\ast^\sigma}^{1+\eps}\right]\\
&\lesssim 2^{-\frac{jd}{6}}+2^{-\frac{jd\eps}{3}}.
\end{align*}
Here we used that if $Z_1,\ldots,Z_n$ are independent random variables, then
\begin{equation*}
\EE\left[\lr{\sum_{j=1}^n (Z_j-\EE[Z_j])}^2\right]=\VV\left[\sum_{j=1}^n Z_j\right]=\sum_{j=1}^n V[Z_j].
\end{equation*}
Similarly, we conclude
\begin{align*}
&\EE\left[\abs{\frac{1}{N_\ell}\sum_{\tau=M_\ell+1}^{M_{\ell+1}}(X_{j,\tau}^\sigma-E_{j,\tau})}\right]\\
&=\EE\left[\abs{\frac{1}{N_\ell}\sum_{\tau=M_\ell+1}^{M_{\ell+1}}(X_{j,\tau}^\sigma\one_{X_{j,\tau}^\sigma\leq \kappa_{j,\tau}}-E_{j,\tau})+\frac{1}{N_\ell}\sum_{\tau=M_\ell+1}^{M_{\ell+1}}X_{j,\tau}^\sigma\one_{X_{j,\tau}^\sigma>\kappa_{j,\tau}}}\right]\\
&\leq\EE\left[\abs{\frac{1}{N_\ell}\sum_{\tau=M_\ell+1}^{M_{\ell+1}}(X_{j,\tau}^\sigma\one_{X_{j,\tau}^\sigma\leq \kappa_{j,\tau}}-E_{j,\tau})}\right]+\frac{1}{N_\ell}\sum_{\tau=M_\ell+1}^{M_{\ell+1}}\EE\left[\abs{X_{j,\tau}^\sigma\one_{X_{j,\tau}^\sigma>\kappa_{j,\tau}}}\right]\\
&\leq \sqrt{\EE\left[\lr{\frac{1}{N_\ell}\sum_{\tau=M_\ell+1}^{M_{\ell+1}}(X_{j,\tau}^\sigma\one_{X_{j,\tau}^\sigma\leq \kappa_{j,\tau}}-E_{j,\tau})}^2\right]}+\frac{1}{N_\ell}\sum_{\tau=M_\ell+1}^{M_{\ell+1}}\kappa_{j,\tau}^{-\eps}\EE\left[\abs{X_{j,\tau}^\sigma}^{1+\eps}\right]\\
&= \sqrt{\frac{1}{N_\ell^2}\sum_{\tau=M_\ell+1}^{M_{\ell+1}}\VV\left[X_{j,\tau}^\sigma\one_{X_{j,\tau}^\sigma\leq \kappa_{j,\tau}}\right]}+2^{-\frac{\ell d\eps}{3}}\EE\left[\abs{X_\ast^\sigma}^{1+\eps}\right]\\
&\leq \sqrt{\frac{1}{N_\ell^2}\sum_{\tau=M_\ell+1}^{M_{\ell+1}}\kappa_{j,\tau}^2}+2^{-\frac{\ell d\eps}{3}}\EE\left[\abs{X_\ast^\sigma}^{1+\eps}\right]\\
&\leq \sqrt{\frac{2^{\frac{2\ell d}{3}}}{N_\ell}}+2^{-\frac{\ell d\eps}{3}}\EE\left[\abs{X_\ast^\sigma}^{1+\eps}\right]\\
&\lesssim  2^{-\frac{\ell d}{6}}+2^{-\frac{\ell d\eps}{3}}.
\end{align*}
Here, we used $M_j\asymp N_j\asymp 2^{jd}$, see the proof of \cref{thm:master-lemma}.
It follows that 
\begin{equation*}
\sum_{j\in\NN_0}\EE\left[\abs{\frac{1}{M_j}\sum_{\tau=1}^{M_j}(X_{j,\tau}^\sigma-E_{j,\tau})}\right]\lesssim \sum_{j\in\NN_0}\lr{2^{-\frac{jd}{6}}+2^{-\frac{jd\eps}{3}}}<\infty
\end{equation*}
and
\begin{align*}
\sum_{j\in\NN_0}\sum_{\ell\geq j}\EE\left[\abs{\frac{1}{N_\ell}\sum_{\tau=M_\ell+1}^{M_{\ell+1}}(X_{j,\tau}^\sigma-E_{j,\tau})}\right]&\lesssim \sum_{j\in\NN_0}\sum_{\ell\geq j} \lr{2^{-\frac{\ell d}{6}}+2^{-\frac{\ell d\eps}{3}}}\\
&\lesssim \sum_{j\in\NN_0}\lr{2^{-\frac{j d}{6}}+2^{-\frac{j d\eps}{3}}}<\infty.
\end{align*}
This completes the proof.
\end{proof}

The following lemma is the technical part of the proof of \cref{thm:mgf}.
\begin{Lem}\label{thm:1+eps-exp}
Let $(X_{j,\tau})_{j\in\NN_0,\tau\in\NN}$ be a family of non-negative i.i.d.\ random variables, $d\in\NN$, $C,\eps\in (0,\infty)$, and $\sigma\in [1,\infty)$.
Suppose that
\begin{equation*}
\EE[\exp(CX_{j,\tau}^\sigma)]<\infty.
\end{equation*}
Fix $d\in\NN$ and let
\begin{align*}
M_j&\defeq \#\setcond{m\in\ZZ^d}{\mnorm{m}_{\infty}\leq 2^j}=(2^{j+1}+1)^d\\
N_j&\defeq\#\setcond{m\in\ZZ^d}{2^j<\mnorm{m}_\infty\leq 2^{j+1}}=M_{j+1}-M_j\\
\end{align*}
for $j\in\NN_0$.
Then, for
\begin{equation*}
\Xi\defeq \sup_{j \in \NN_0} \lr{\frac{1}{M_j} \sum_{\tau=1}^{M_j} X_{j,\tau}+\sup_{\ell\geq j} \frac{1}{N_\ell} \sum_{\tau=M_\ell + 1}^{M_{\ell+1}} X_{j,\tau}},
\end{equation*} 
we have $\EE\left[\exp\lr{\frac{C}{2^{\sigma+1}(1+\eps)} \Xi^\sigma}\right]<\infty$.
\end{Lem}
\begin{proof}
First note that $(a+b)^\sigma\leq (2\max\setn{a,b})^\sigma\leq 2^\sigma(a^\sigma+b^\sigma)$ and
\begin{equation*}
\exp(a+b)=\exp(a)\exp(b)=\frac{\exp(a)^2}{2}+\frac{\exp(b)^2}{2}\leq \exp(2a)+\exp(2b)
\end{equation*}
for all  $a,b\in\RR$.
Together with Jensen's inequality, we obtain for 
\begin{equation*}
Y_j\defeq\frac{1}{M_j}\sum_{\tau=1}^{M_j}X_{j,\tau}\qquad\text{and}\qquad Y_{j,\ell}\defeq\frac{1}{N_\ell}\sum_{\tau=M_\ell+1}^{M_{\ell+1}}X_{j,\tau}
\end{equation*}
 that
\begin{align*}
&\exp\lr{\frac{C}{2^{\sigma+1}(1+\eps)}  \Xi^\sigma}\\
&\leq \exp\lr{\frac{C}{2^{\sigma+1}(1+\eps)} \lr{\sup_{j\in\NN_0}Y_j+\sup_{j\in\NN_0}\sup_{\ell\geq j}Y_{j,\ell}}^\sigma}\\
&\leq \exp\left(\frac{C}{2^{\sigma+1}(1+\eps)} 2^\sigma\lr{\sup_{j\in\NN_0}Y_j}^\sigma+\frac{C}{2^{\sigma+1}(1+\eps)} 2^\sigma\lr{\sup_{j\in\NN_0}\sup_{\ell\geq j}Y_{j,\ell}}^\sigma\right)\\
&\leq\exp\lr{\frac{C}{1+\eps}\sup_{j\in\NN_0}Y_j^\sigma}+\exp\lr{\frac{C}{1+\eps}\sup_{j\in\NN_0}\sup_{\ell\geq j}Y_{j,\ell}^\sigma}\\
&= \sup_{j\in\NN_0} \exp\lr{\frac{C}{1+\eps} Y_j^\sigma}+\sup_{j\in\NN_0}\sup_{\ell\geq j}\exp\lr{\frac{C}{1+\eps}Y_{j,\ell}^\sigma}\\
&\leq \sup_{j\in\NN_0} \frac{1}{M_j}\sum_{\tau=1}^{M_j}\exp\lr{\frac{C}{1+\eps} X_{j,\tau}^\sigma}+\sup_{j\in\NN_0}\sup_{\ell\geq j}\frac{1}{N_\ell}\sum_{\tau=M_\ell+1}^{M_{\ell+1}}\exp\lr{\frac{C}{1+\eps} X_{j,\tau}^\sigma}.
\end{align*}
Let $\tilde{X}_{j,\tau}\defeq \exp\lr{\frac{C}{1+\eps}  X_{j,\tau}^\sigma}$.
Then $\EE[\tilde{X}_{j,\tau}^{1+\eps}]=\EE[\exp\lr{C X_{j,\tau}^\sigma}]<\infty$, and the monotonicity of the expectation yields
\begin{equation*}
\EE\left[\exp\lr{\frac{C}{2^{\sigma+1}(1+\eps)} \Xi^\sigma}\right]\leq  \EE\left[\sup_{j\in\NN_0} \frac{1}{M_j}\sum_{\tau=1}^{M_j}\tilde{X}_{j,\tau}+\sup_{j\in\NN_0}\sup_{\ell\geq j}\frac{1}{N_\ell}\sum_{\tau=M_\ell+1}^{M_{\ell+1}}\tilde{X}_{j,\tau}\right].
\end{equation*}
Now \cref{thm:1+eps} applied to the $(\tilde{X}_{j,\tau})_{j\in\NN_0,\tau\in\NN}$ and $\sigma=1$ shows that the right-hand side is finite.
\end{proof}

The following lemma is the technical part of the proof of \cref{thm:powers-bernoulli}.
\begin{Lem}\label{thm:bernoulli-lemma}
Let $(X_{j,\tau})_{j\in\NN_0,\tau\in\NN}$ be a family of non-negative i.i.d.\ random variables, $\eps,\delta \in (0,\infty)$, $\mu,\nu\in (-\infty,0]$, and $\sigma\in [1,\infty)$.
Fix $d\in\NN$ and let
\begin{align*}
M_j&\defeq \#\setcond{m\in\ZZ^d}{\mnorm{m}_{\infty}\leq 2^j}=(2^{j+1}+1)^d,\\
N_j&\defeq\#\setcond{m\in\ZZ^d}{2^j<\mnorm{m}_\infty\leq 2^{j+1}}=M_{j+1}-M_j
\end{align*}
for $j\in\NN_0$.
Let further $(\lambda_{j,\tau})_{j\in\NN_0,\tau\in\NN}$ be a family of independent random variables, independent of $(X_{j,\tau})_{j\in\NN_0,\tau\in\NN}$ and such that $\lambda_{j,\tau}$ is Bernoulli distributed with success probability
\begin{equation*}
\varrho_{j,\tau}\defeq 2^{j\mu}\lr{1+\frac{\mnorm{\phi(\tau)}_\infty}{2^j }}^\nu,
\end{equation*}
where $\phi:\NN\to \ZZ^d$ is a bijection such that $\tau\mapsto\mnorm{\phi(\tau)}_\infty$ is nondecreasing.
If
\begin{equation}\label{eq:bernoulli-lemma-rv}
\EE[X_{j,\tau}^{\sigma(1+\eps)}]< \infty\qquad\text{and}\qquad \delta+\delta\eps-\eps\geq 0,
\end{equation}
then
\begin{equation*}
\tilde{\Xi}\defeq \sup_{j\in\NN_0}\lr{\frac{1}{M_j}\sum_{\tau=1}^{M_j}\varrho_{j,\tau}^{\frac{\delta-1}{\sigma}}\lambda_{j,\tau}X_{j,\tau}+\sup_{\ell\geq j}\frac{1}{N_\ell}\sum_{\tau=M_\ell+1}^{M_{\ell+1}}\varrho_{j,\tau}^{\frac{\delta-1}{\sigma}}\lambda_{j,\tau}X_{j,\tau}}
\end{equation*}
satisfies $\EE[\tilde{\Xi}^\sigma]<\infty$.
\end{Lem}
\begin{proof}
Let 
\begin{align*}
E&\defeq \EE[X_{0,1}^\sigma],\\
Y_{j,\tau}&\defeq \varrho_{j,\tau}^{\frac{\delta-1}{\sigma}}\lambda_{j,\tau}X_{j,\tau},\\
\kappa_{j,\tau}&\defeq\begin{cases}2^{\frac{jd}{3}}&\text{if }1\leq\tau\leq M_j,\\2^{\frac{\ell d}{3}}&\text{if }M_\ell +1\leq \tau\leq M_ {\ell+1}\text{ for some (unique) }\ell\geq j,\end{cases}\\
E_{j,\tau}&\defeq \EE[Y_{j,\tau}^\sigma\one_{Y_{j,\tau}^\sigma\leq \kappa_{j,\tau}}]
\end{align*}
for $j\in\NN_0$ and $\tau\in\NN$.

Let $j\in \NN_0$.
If $1\leq \tau\leq M_j$, then $0\leq\mnorm{\phi(\tau)}_\infty \leq 2^j$, i.e., $1\leq 1+\frac{\mnorm{\phi(\tau)}_\infty}{2^j}\leq 2$.
It follows that $\varrho_{j,\tau}\asymp 2^{j\mu}$,
and
\begin{align*}
\EE[Y_{j,\tau}^{\sigma(1+\eps)}]&=\varrho_{j,\tau}^{(1+\eps)(\delta-1)}\EE[\lambda_{j,\tau}]\EE[X_{j,\tau}^{\sigma(1+\eps)}]\lesssim \varrho_{j,\tau}^{(1+\eps)(\delta-1)+1}\\
&= \varrho_{j,\tau}^{\delta+\delta\eps+\eps}\asymp 2^{j\mu(\delta+\delta\eps-\eps)}\leq 1
\end{align*}
since $\mu\geq 0$ and $\delta+\delta\eps+\eps\geq 0$.

If, on the other hand, $\tau>M_j$, then there exists a unique natural number $\ell\geq j$ such that $M_\ell+1\leq \tau\leq M_{\ell+1}$.
Then $2^\ell< \mnorm{\phi(\tau)}_\infty\leq 2^{\ell+1}$ and
\begin{equation*}
2^{\ell-j}\leq 1+2^{\ell-j}\leq 1+\frac{\mnorm{\phi(\tau)}_\infty}{2^j}\leq 1+2^{\ell+1-j}\leq 3\cdot 2^{\ell-j},
\end{equation*}
i.e., $\varrho_{j,\tau}\asymp 2^{j\mu}\cdot 2^{(\ell-j)\nu}$.
It follows as above that 
\begin{equation*}
\EE[Y_{j,\tau}^{\sigma(1+\eps)}]\lesssim \varrho_{j,\tau}^{\delta+\delta\eps+\eps}\asymp (2^{j\mu}\cdot 2^{(\ell-j)\nu})^{\delta+\delta\eps-\eps}\leq 1.
\end{equation*}
Further note because of $\varrho_{j,t}\leq 1$ and $\delta\geq 0$ that
\begin{equation*}
E\geq \varrho_{j,\tau}^{\delta} \EE[X_{j,\tau}^\sigma] = \varrho_{j,\tau}^{\delta-1}\EE[\lambda_{j,\tau}]\EE[X_{j,\tau}^\sigma]=\EE[Y_{j,\tau}^\sigma]\geq E_{j,\tau}.
\end{equation*}
From
\begin{align*}
\tilde{\Xi}^\sigma&\leq\lr{\sup_{j\in\NN_0}\frac{1}{M_j}\sum_{\tau=1}^{M_j}Y_{j,\tau}+\sup_{j\in\NN_0}\sup_{\ell\geq j}\frac{1}{N_\ell}\sum_{\tau=M_\ell+1}^{M_{\ell+1}}Y_{j,\tau}}^\sigma\\
&\lesssim \lr{\sup_{j\in\NN_0}\frac{1}{M_j}\sum_{\tau=1}^{M_j}Y_{j,\tau}}^\sigma+\lr{\sup_{j\in\NN_0}\sup_{\ell\geq j}\frac{1}{N_\ell}\sum_{\tau=M_\ell+1}^{M_{\ell+1}}Y_{j,\tau}}^\sigma\\
&= \sup_{j\in\NN_0}\lr{\frac{1}{M_j}\sum_{\tau=1}^{M_j}Y_{j,\tau}}^\sigma+\sup_{j\in\NN_0}\sup_{\ell\geq j}\lr{\frac{1}{N_\ell}\sum_{\tau=M_\ell+1}^{M_{\ell+1}}Y_{j,\tau}}^\sigma\\
&\leq  \sup_{j\in\NN_0}\frac{1}{M_j}\sum_{\tau=1}^{M_j}Y_{j,\tau}^\sigma +\sup_{j\in\NN_0}\sup_{\ell\geq j}\frac{1}{N_\ell}\sum_{\tau=M_\ell+1}^{M_{\ell+1}}Y_{j,\tau}^\sigma\\
&\leq\sup_{j\in\NN_0}\frac{1}{M_j}\sum_{\tau=1}^{M_j}(Y_{j,\tau}^\sigma-E)+\sup_{j\in\NN_0}\sup_{\ell\geq j}\frac{1}{N_\ell}\sum_{\tau=M_\ell+1}^{M_{\ell+1}}(Y_{j,\tau}^\sigma-E)+2E\\
&\leq\sup_{j\in\NN_0}\frac{1}{M_j}\sum_{\tau=1}^{M_j}(Y_{j,\tau}^\sigma-E_{j,\tau})+\sup_{j\in\NN_0}\sup_{\ell\geq j}\frac{1}{N_\ell}\sum_{\tau=M_\ell+1}^{M_{\ell+1}}(Y_{j,\tau}^\sigma-E_{j,\tau})+2E\\
&\leq\sum_{j\in\NN_0}\abs{\frac{1}{M_j}\sum_{\tau=1}^{M_j}(Y_{j,\tau}^\sigma-E_{j,\tau})}+\sum_{j\in\NN_0}\sum_{\ell\geq j}\abs{\frac{1}{N_\ell}\sum_{\tau=M_\ell+1}^{M_{\ell+1}}(Y_{j,\tau}^\sigma-E_{j,\tau})}+2E,
\end{align*}
we obtain
\begin{align*}
&\EE[\tilde{\Xi}^\sigma]\\
&\lesssim 2E+\sum_{j\in\NN_0}\EE\left[\abs{\frac{1}{M_j}\sum_{\tau=1}^{M_j}(Y_{j,\tau}^\sigma-E_{j,\tau})}\right]+\sum_{j\in\NN_0}\sum_{\ell\geq j}\EE\left[\abs{\frac{1}{N_\ell}\sum_{\tau=M_\ell+1}^{M_{\ell+1}}(Y_{j,\tau}^\sigma-E_{j,\tau})}\right].
\end{align*}
In order to show that $\EE[\tilde{\Xi}^\sigma]<\infty$, it suffices to show that the two infinite sums in the previous sum are finite.
To this end, we first observe that
\begin{align*}
&\EE\left[\abs{\frac{1}{M_j}\sum_{\tau=1}^{M_j}(Y_{j,\tau}^\sigma-E_{j,\tau})}\right]\\
&=\EE\left[\abs{\frac{1}{M_j}\sum_{\tau=1}^{M_j}(Y_{j,\tau}^\sigma\one_{Y_{j,\tau}^\sigma\leq \kappa_{j,\tau}}-E_{j,\tau})+\frac{1}{M_j}\sum_{\tau=1}^{M_j}Y_{j,\tau}^\sigma\one_{Y_{j,\tau}^\sigma>\kappa_{j,\tau}}}\right]\\
&\leq\EE\left[\abs{ \frac{1}{M_j}\sum_{\tau=1}^{M_j}(Y_{j,\tau}^\sigma\one_{Y_{j,\tau}^\sigma\leq \kappa_{j,\tau}}-E_{j,\tau})}\right]+\frac{1}{M_j}\sum_{\tau=1}^{M_j}\EE\left[\abs{Y_{j,\tau}^\sigma\one_{Y_{j,\tau}^\sigma>\kappa_{j,\tau}}}\right]\\
&\leq \sqrt{\EE\left[\lr{\frac{1}{M_j}\sum_{\tau=1}^{M_j}(Y_{j,\tau}^\sigma\one_{Y_{j,\tau}^\sigma\leq \kappa_{j,\tau}}-E_{j,\tau})}^2\right]}+\frac{1}{M_j}\sum_{\tau=1}^{M_j}\kappa_{j,\tau}^{-\eps}\EE\left[\abs{Y_{j,\tau}^\sigma}^{1+\eps}\right]\\
&\leq \sqrt{\frac{1}{M_j^2}\sum_{\tau=1}^{M_j}\VV\left[Y_{j,\tau}^\sigma\one_{Y_{j,\tau}^\sigma\leq \kappa_{j,\tau}}\right]}+2^{-\frac{j d\eps}{3}}\\
&\leq \sqrt{\frac{2^{\frac{2jd}{3}}}{M_j}}+2^{-\frac{j d\eps}{3}}\\
&\lesssim 2^{-\frac{jd}{6}}+2^{-\frac{j d\eps}{3}}.
\end{align*}
Similarly, we conclude
\begin{align*}
&\EE\left[\abs{\frac{1}{N_\ell}\sum_{\tau=M_\ell+1}^{M_{\ell+1}}(Y_{j,\tau}^\sigma-E_{j,\tau})}\right]\\
&=\EE\left[\abs{\frac{1}{N_\ell}\sum_{\tau=M_\ell+1}^{M_{\ell+1}}(Y_{j,\tau}^\sigma\one_{Y_{j,\tau}^\sigma\leq \kappa_{j,\tau}}-E_{j,\tau})+\frac{1}{N_\ell}\sum_{\tau=M_\ell+1}^{M_{\ell+1}}Y_{j,\tau}^\sigma\one_{Y_{j,\tau}^\sigma>\kappa_{j,\tau}}}\right]\\
&\leq\EE\left[\abs{\frac{1}{N_\ell}\sum_{\tau=M_\ell+1}^{M_{\ell+1}}(Y_{j,\tau}^\sigma\one_{Y_{j,\tau}^\sigma\leq \kappa_{j,\tau}}-E_{j,\tau})}\right]+\frac{1}{N_\ell}\sum_{\tau=M_\ell+1}^{M_{\ell+1}}\EE\left[\abs{Y_{j,\tau}^\sigma\one_{Y_{j,\tau}^\sigma>\kappa_{j,\tau}}}\right]\\
&\leq \sqrt{\EE\left[\lr{\frac{1}{N_\ell}\sum_{\tau=M_\ell+1}^{M_{\ell+1}}(Y_{j,\tau}^\sigma\one_{Y_{j,\tau}^\sigma\leq \kappa_{j,\tau}}-E_{j,\tau})}^2\right]}+\frac{1}{N_\ell}\sum_{\tau=M_\ell+1}^{M_{\ell+1}}\kappa_{j,\tau}^{-\eps}\EE\left[\abs{Y_{j,\tau}^\sigma}^{1+\eps}\right]\\
&\leq \sqrt{\frac{1}{N_\ell^2}\sum_{\tau=M_\ell+1}^{M_{\ell+1}}\VV\left[Y_{j,\tau}^\sigma\one_{Y_{j,\tau}^\sigma\leq \kappa_{j,\tau}}\right]}+2^{-\frac{\ell d\eps}{3}}\\
&\leq \sqrt{\frac{2^{\frac{2\ell d}{3}}}{N_\ell}}+2^{-\frac{\ell d\eps}{3}}\\
&\lesssim  2^{-\frac{\ell d}{6}}+2^{-\frac{\ell d\eps}{3}}.
\end{align*}
It follows that 
\begin{align*}
\sum_{j\in\NN_0}\EE\left[\abs{\frac{1}{M_j}\sum_{\tau=1}^{M_j}(Y_{j,\tau}^\sigma-E_{j,\tau})}\right]&\lesssim \sum_{j\in\NN_0}\lr{2^{-\frac{jd}{6}}+2^{-\frac{j d\eps}{3}}}<\infty
\end{align*}
and
\begin{align*}
\sum_{j\in\NN_0}\sum_{\ell\geq j}\EE\left[\abs{\frac{1}{N_\ell}\sum_{\tau=M_\ell+1}^{M_{\ell+1}}(Y_{j,\tau}^\sigma-E_{j,\tau})}\right]&\lesssim \sum_{j\in\NN_0}\sum_{\ell\geq j} \lr{2^{-\frac{\ell d}{6}}+2^{-\frac{\ell d\eps}{3}}}\\
&\lesssim \sum_{j\in\NN_0}\lr{2^{-\frac{jd}{6}}+2^{-\frac{j d\eps}{3}}}<\infty.\qedhere
\end{align*}
\end{proof}

The following lemma is the technical part of the proof of \cref{thm:sufficiency-log}.

\begin{Lem}\label{thm:UniformLLN-new}
Let $(X_{j,\tau})_{j \in \NN_0, \tau \in \NN}$ be a family of non-negative i.i.d.\ random variables.
Suppose that
\begin{equation}\label{eq:xlogx}
\EE[X_{j,\tau} \cdot \log_2^+ (X_{j,\tau})] < \infty.
\end{equation}
Fix $d\in\NN$ and let
\begin{align*}
M_j&\defeq \#\setcond{m\in\ZZ^d}{\mnorm{m}_{\infty}\leq 2^j}=(2^{j+1}+1)^d,\\
N_j&\defeq\#\setcond{m\in\ZZ^d}{2^j<\mnorm{m}_\infty\leq 2^{j+1}}=M_{j+1}-M_j
\end{align*}
for $j\in\NN_0$.
Then the (random) quantity 
\begin{equation*}
\Xi\defeq \sup_{j \in \NN_0} \lr{\frac{1}{M_j} \sum_{\tau=1}^{M_j} X_{j,\tau}+\sup_{\ell\geq j} \frac{1}{N_\ell} \sum_{\tau=M_\ell + 1}^{M_{\ell+1}} X_{j,\tau}}
\end{equation*} 
is almost surely finite.
\end{Lem}

\begin{proof}

\emph{Step 1.} We prove the desired properties for the truncations
\begin{equation*}
Y_{j,\tau}=\begin{cases}X_{j,\tau}\one_{X_{j,\tau}\leq M_j}&\text{if }1\leq \tau \leq M_j,\\X_{j,\tau}\one_{X_{j,\tau}\leq M_\ell}&\text{if }M_\ell+1\leq \tau \leq M_{\ell+1}\text{ for some (unique) }\ell\geq j\end{cases}
\end{equation*}
in place of $X_{j,\tau}$, i.e., we show for $E\defeq \EE[X_{j,\tau}]$ and arbitrary $C>0$ that
\begin{equation*}
\sum_{j\in\NN_0} \PP(S_j^\ast\geq C+E)< \infty
\qquad\text{and}\qquad
\sum_{j\in\NN_0} \sum_{\ell\geq j}\PP(S_{j,\ell}^\ast \geq C+E)< \infty,
\end{equation*}
where 
\begin{equation*}
S_j^\ast\defeq\frac{1}{M_j}\sum_{\tau=1}^{M_j}Y_{j,\tau}\qquad\text{and}\qquad S_{j,\ell}^\ast\defeq\frac{1}{N_\ell}\sum_{\tau=M_\ell+1}^{M_{\ell+1}}Y_{j,\tau}.
\end{equation*}
This then implies by the Borel--Cantelli lemma
\begin{equation*}
\Xi^\ast\defeq \sup_{j\in\NN_0} \lr{S_j^\ast+\sup_{\ell\geq j}S_{j,\ell}^\ast}<\infty
\end{equation*}
almost surely.

Let $X_\ast\defeq X_{0,1}$ and recall that $X_{j,\tau}\stackrel{iid}{\sim}X_\ast$.
Using $E\geq \EE[Y_{j,\tau}]$, Chebyshev's inequality, and the independence of the $Y_{j,\tau}$, we obtain
\begin{align*}
&\PP(S_j^\ast\geq C+E)\\
&=\PP\lr{\frac{1}{M_j}\sum_{\tau=1}^{M_j}(Y_{j,\tau}-E)\geq C}\leq \PP\lr{\frac{1}{M_j}\sum_{\tau=1}^{M_j}(Y_{j,\tau}-\EE[Y_{j,\tau}])\geq C}\\
&\leq\frac{1}{C^2}\VV\left[\frac{1}{M_j}\sum_{\tau=1}^{M_j}(Y_{j,\tau}-\EE[Y_{j,\tau}])\right]=\frac{1}{C^2 M_j^2}\sum_{\tau=1}^{M_j}\VV[Y_{j,\tau}-\EE[Y_{j,\tau}]]\\
&=\frac{1}{C^2M_j^2}\sum_{\tau=1}^{M_j}\VV[Y_{j,\tau}]=\frac{1}{C^2 M_j^2}\sum_{\tau=1}^{M_j}\lr{\EE[Y_{j,\tau}^2]-\EE[Y_{j,\tau}]^2}\\
&\leq\frac{1}{C^2 M_j^2}\sum_{\tau=1}^{M_j}\EE[Y_{j,\tau}^2]=\frac{1}{C^2 M_j^2}\sum_{\tau=1}^{M_j}\EE[X_{j,\tau}^2\one_{X_{j,\tau}\leq M_j}]\\
&=\frac{1}{C^2 M_j^2}\sum_{\tau=1}^{M_j}\EE[X_\ast^2\one_{X_\ast\leq M_j}]=\frac{1}{C^2 M_j}\EE[X_\ast^2\one_{X_\ast\leq M_j}].
\end{align*}
Thus, by the monotone convergence theorem,
\begin{align}
\sum_{j\in\NN_0} \PP(S_j^\ast\geq C+E)&\leq  \frac{1}{C^2}\sum_{j\in\NN_0} \frac{1}{M_j}\EE[X_\ast^2\one_{X_\ast\leq M_j}]\nonumber\\
&=\frac{1}{C^2}\EE\left[X_\ast^2 \sum_{j\in\NN_0} \frac{1}{M_j}\one_{X_\ast\leq M_j}\right]\nonumber\\
&=\frac{1}{C^2}\EE\left[X_\ast^2\one_{X_\ast >0} \sum_{j\in\NN_0} \frac{1}{M_j}\one_{X_\ast\leq M_j}\right]\label{eq:insert-indicator}\\
&\leq\frac{1}{C^2}\EE\left[X_\ast^2\one_{X_\ast >0} \sum_{j\in\NN_0} 2^{-jd}\one_{X_\ast\leq 2^{(j+2)d}}\right]\label{eq:avoid-mj}\\
&\leq\frac{1}{C^2}\EE\left[X_\ast^2\one_{X_\ast >0} \sum_{j\geq\ceil{\frac{\log_2(X_\ast)}{d}-2}} {2^{-jd}}\right].\label{eq:lemma-step1-part2}
\end{align}
Recall that we are allowed to insert the indicator function $\one_{X_\ast >0}$ in \eqref{eq:insert-indicator} because by assumption $X_{j,\tau}$ are nonnegative random variables and because the factor $X_\ast^2$ vanishes if $X_\ast=0$.
In \eqref{eq:avoid-mj}, we used that $2^{jd}\leq M_j \leq 2^{(j+2)d}$, whence $\frac{1}{M_j}\one_{X_\ast\leq M_j}\leq 2^{-jd}\one_{X_\ast\leq 2^{(j+2)d}}$.
For \eqref{eq:lemma-step1-part2}, note that $X_\ast\leq 2^{(j+2)d}$ is equivalent to $j\geq \ceil{\frac{\log_2(X_\ast)}{d}-2}$. 

Next, perform the index shift $\ell=j+2$ to see
\begin{align*}
\sum_{j\geq\ceil{\frac{\log_2(X_\ast)}{d}-2}} {2^{-jd}}&=\sum_{\ell\geq\ceil{\frac{\log_2(X_\ast)}{d}}} {2^{-(\ell-2)d}}\lesssim \sum_{\ell\geq\ceil{\frac{\log_2(X_\ast)}{d}}} {2^{-\ell d}}\\
&=\frac{2^d}{2^d-1}2^{-d\ceil{\frac{\log_2(X_\ast)}{d}}}\lesssim 2^{-d\ceil{\frac{\log_2(X_\ast)}{d}}}.
\end{align*}
Thus, it suffices to prove that the following quantity is finite:
\begin{align}
&\EE\left[X_\ast^2\one_{X_\ast >0}\frac{1}{2^{d\ceil{\frac{\log_2(X_\ast)}{d}}}}\right]\leq\EE\left[X_\ast^2\one_{X_\ast >0}\frac{1}{X_\ast}\right]\leq \EE[X_\ast]<\infty.\label{eq:first-step}
\end{align}
Here we used that $x\log_2^+(x)\geq x$ for $x\geq 3$, which shows that \eqref{eq:xlogx} implies $\EE[X_\ast]<\infty$.

To estimate the second series, note that by $E\geq \EE[Y_{j,\tau}]$ and Chebyshev's inequality
\begin{align*}
\PP(S_{j,\ell}^\ast \geq C+E)&=\PP\lr{\frac{1}{N_\ell}\sum_{\tau=M_\ell+1}^{M_{\ell+1}}(Y_{j,\tau}-E)\geq C}\\
&\leq \PP\lr{\frac{1}{N_\ell}\sum_{\tau=M_\ell+1}^{M_{\ell+1}}(Y_{j,\tau}-\EE[Y_{j,\tau}])\geq C}\\
&\leq \frac{1}{C^2}\VV\left[\frac{1}{N_\ell}\sum_{\tau=M_\ell+1}^{M_{\ell+1}}(Y_{j,\tau}-\EE[Y_{j,\tau}])\right]\\
&= \frac{1}{C^2 N_\ell^2}\sum_{\tau=M_\ell+1}^{M_{\ell+1}}\VV[Y_{j,\tau}-\EE[Y_{j,\tau}]]\\
&\leq \frac{1}{C^2 N_\ell^2}\sum_{\tau=M_\ell+1}^{M_{\ell+1}}\EE[Y_{j,\tau}^2].
\end{align*}
Next, since $M_\ell\asymp N_\ell\asymp 2^{\ell d}$ (see the proof of \cref{thm:master-lemma}, there exist $\kappa_1,\kappa_2,\kappa_3>0$ only depending on $d$ such that if $X_\ast\leq M_\ell$ and $M_\ell+1\leq\tau\leq M_{\ell+1}$, then $\kappa_1 \tau\leq N_\ell \leq \kappa_2 \tau$, and $\kappa_3\tau\leq M_\ell\leq \tau$, as well as $X_\ast\leq\tau$.
We conclude
\begin{align*}
\sum_{\ell\geq j}\PP(S_{j,\ell}^\ast \geq C+E)&\leq \frac{1}{C^2} \sum_{\ell\geq j} \frac{1}{N_\ell^2} \sum_{\tau=M_\ell+1}^{M_{\ell+1}}\EE[Y_{j,\tau}^2]\\
&=\frac{1}{C^2} \sum_{\ell\geq j} \sum_{\tau=M_\ell+1}^{M_{\ell+1}}\EE\left[\frac{X_\ast^2}{N_\ell^2}\one_{X_\ast\leq M_\ell}\right]\\
&\leq\frac{1}{C^2} \sum_{\ell\geq j} \sum_{\tau=M_\ell+1}^{M_{\ell+1}}\EE\left[\frac{X_\ast^2}{\kappa_1^2\tau^2}\one_{X_\ast\leq \tau}\right]\\
&\lesssim \sum_{\tau\geq M_j+1}\EE\left[\frac{X_\ast^2}{\tau^2}\one_{X_\ast\leq \tau}\right]\\
&=\EE\left[X_\ast^2\sum_{\tau\geq M_j+1}\frac{\one_{X_\ast\leq \tau}}{\tau^2}\right]\eqdef a_j.
\end{align*}
For $N\geq 3$, there exists $n=n(N)\in\NN$ such that $2^n+1\leq N\leq 2^{n+1}$, whence
\begin{align}
\sum_{\tau\geq N} \frac{\log_2(\tau)}{\tau^2}&\leq \sum_{\ell\geq n} \sum_{\tau=2^\ell+1}^{2^{\ell+1}}\frac{\log_2(\tau)}{\tau^2}\leq \sum_{\ell\geq n} \sum_{\tau=2^\ell+1}^{2^{\ell+1}}\frac{\ell+1}{(2^\ell)^2}\nonumber\\
&= \sum_{\ell\geq n} \frac{\ell+1}{2^\ell}=\frac{n+1}{2^n}\sum_{\ell\geq n}\frac{\ell+1}{(n+1)2^{\ell-n}}\nonumber\\
&\leq\frac{n+1}{2^n}\sum_{s\in\NN_0}\frac{s+1}{2^s}\lesssim \frac{n+1}{2^n}\lesssim \frac{1+\log_2(N)}{N}.\label{eq:logaux}
\end{align}
Here we used $\frac{\ell+1}{n+1}=\frac{\ell-n+n+1}{n+1}\leq 1+\ell-n$.

Also, it is not hard to check that $M_j+1\leq \tau$ is equivalent to $\frac{\log_2(\tau-1)}{d}\geq \log_2(2^{j+1}+1)$.
Using the monotonicity of the logarithm, this implies $\frac{\log_2(\tau)}{d}\geq j+1\geq j$.
Hence
\begin{align*}
\sum_{j\in\NN_0}\sum_{\ell\geq j}\PP(S_{j,\ell}^\ast\geq C+E)&\lesssim\sum_{j\in\NN_0} a_j\\
&=\EE\left[X_\ast^2\sum_{j\in\NN_0}\sum_{\tau\geq M_j+1}\frac{\one_{X_\ast\leq\tau}}{\tau^2} \right]\\
&=\EE\left[X_\ast^2\sum_{\tau\geq 3^d}\sum_{j\in\NN_0} \one_{M_j+1\leq \tau}\frac{\one_{X_\ast\leq \tau}}{\tau^2} \right]\\
&\leq\EE\left[X_\ast^2\sum_{\tau\geq 3^d}\frac{\one_{X_\ast\leq \tau}}{\tau^2} \sum_{j=1}^{\floor{\frac{\log_2(\tau)}{d}}} 1\right]\\
&\leq\EE\left[X_\ast^2 \sum_{\tau\geq \max\setn{3^d,\ceil{X_\ast}}} \frac{\log_2(\tau)}{d\cdot\tau^2}\right] \\
&\stackrel{\eqref{eq:logaux}}{\lesssim} \EE\left[X_\ast^2 \frac{1+\log_2^+(\max\setn{3^d,\ceil{X_\ast}})}{\max\setn{3^d,\ceil{X_\ast}}}\right]\\
&\lesssim \EE\left[X_\ast^2 \frac{1+\log_2^+(X_\ast)}{X_\ast}\right]\\
&=\EE\left[X_\ast\cdot (1+\log_2^+(X_\ast))\right]\\
&=\EE\left[X_\ast]+\EE[X_\ast\log_2^+(X_\ast)\right]<\infty,
\end{align*}
as in \eqref{eq:first-step}.\\[\baselineskip]
\emph{Step 2.} Next we show $\sum_{j\in\NN_0} \sum_{\tau\in\NN} \PP(X_{j,\tau}\neq Y_{j,\tau})<\infty$.
The Borel--Cantelli lemma then implies $\PP(X_{j,\tau}\neq Y_{j,\tau}\text{ for infinitely many }(j,\tau))=0$.
That is, the event $A$ of all $\omega\in \Omega$ for which $X_{j,\tau}= Y_{j,\tau}$ happens for all but finitely many $(j,\tau)$ has probability $1$.
By Step 1, we know that 
\begin{equation*}
\Xi^\ast\defeq \sup_{j\in\NN_0} \lr{S_j^\ast+\sup_{\ell\geq j}S_{j,\ell}^\ast}<\infty
\end{equation*}
holds in an almost sure event $B$.
It is then easy to see that $\Xi<\infty$ in the almost sure event $A\cap B$.

In order to show that $\sum_{j\in \NN_0} \sum_{\tau\in \NN} \PP(X_{j,\tau}\neq Y_{j,\tau})<\infty$, first note that $X_{j,\tau}>M_\ell$ and $M_\ell+1\leq \tau\leq M_{\ell+1}$ imply $\kappa_4X_{j,\tau}\geq \frac{M_{\ell+1}}{M_\ell}X_{j,\tau}>M_{\ell+1}\geq \tau$ for some $\kappa_4>0$ which only depends on $d$ (which is irrelevant for the convergence of the series).
The random variables $\kappa_4 X_{j,\tau}$ all have the same distribution as the random variable $Y\defeq \kappa_4 X_\ast$.
Hence, 
\begin{align}
\sum_{j\in \NN}\sum_{\tau\in \NN_0} \PP(X_{j,\tau}\neq Y_{j,\tau})&\leq \sum_{j\in \NN_0}\lr{\sum_{\tau=1}^{M_j}\PP(X_{j,\tau}>M_j)+\sum_{\ell\geq j} \sum_{\tau=M_\ell+1}^{M_{\ell+1}}\PP(X_{j,\tau}>M_\ell)}\nonumber\\
&\leq \sum_{j\in \NN_0}\lr{\sum_{\tau=1}^{M_j}\PP(X_\ast>M_j)+\sum_{\ell\geq j} \sum_{\tau=M_\ell+1}^{M_{\ell+1}}\PP(Y>\tau)}\nonumber\\
&= \sum_{j\in \NN_0}\lr{M_j\PP(X_\ast>M_j)+\sum_{\tau\geq M_j+1}\PP(Y>\tau)}.\label{eq:step2}
\end{align}

Let us show that the two sums in \eqref{eq:step2} are finite.
First, observe thanks to the monotone convergence lemma that
\begin{align*}
\sum_{j\in \NN_0} M_j\PP(X_\ast>M_j)&=\sum_{j\in \NN_0} M_j \EE[\one_{X_\ast>M_j}]=\EE\left[\sum_{j\in \NN_0} M_j\one_{X_\ast>M_j}\right]\\
&\leq \EE \left[\sum_{j=1}^{\floor{\frac{\log_2(X_\ast)}{d}}}M_j\right]\lesssim \EE \left[\sum_{j=1}^{\floor{\frac{\log_2(X_\ast)}{d}}}2^{jd}\right]\\
&\lesssim \EE \left[ 2^{d\floor{\frac{\log_2(X_\ast)}{d}}}\right]\leq \EE[X_\ast]<\infty.
\end{align*}
Here we used that $X_\ast>M_j$ implies $X_\ast>2^{jd}$, which in turn yields $j\leq \frac{\log_2(X_\ast)}{d}$.
Second, we have
\begin{align*}
\sum_{j\in \NN_0}\sum_{\tau\geq M_j+1} \PP(Y>\tau)&=\sum_{j\in \NN_0}\sum_{\tau\geq 3} \one_{\tau\geq M_j+1}\PP(Y> \tau)\\
&\leq\sum_{\tau\geq 3} \PP(Y>\tau)\sum_{j=1}^{\floor{\frac{\log_2(\tau)}{d}}}1\\
&=\EE\left[\sum_{\tau\geq 3} \one_{Y>\tau}\floor{\frac{\log_2(\tau)}{d}}\right]\\
&\leq\EE\left[\one_{Y>3}\sum_{\tau=3}^{\floor{Y}}\log_2(\tau)\right]\\
&\leq\EE\left[\one_{Y>3}\sum_{\tau=3}^{\floor{Y}}\log_2^+(\floor{Y})\right]\\
&\leq\EE\left[\floor{Y}\log_2^+(\floor{Y})\right]\\
&\leq\EE [Y\log_2^+(Y)]\\
&\leq\kappa_4\EE [X_\ast(\log_2^+(\kappa_4)+\log_2^+(X_\ast))]\\
&=\kappa_4\log_2^+(\kappa_4)\EE[X_\ast]+\kappa_4\EE[X_\ast\log_2^+(X_\ast)]<\infty,
\end{align*}
as in \eqref{eq:first-step}.
\end{proof}

The following auxiliary statement is used in the proof of \cref{bem:essential-boundedness}.
\begin{Lem} \label{thm:SupSTDNorm}
For every $n\in\NN$, let $X^{(n)}=(X^{(n)}_1,\ldots,X^{(n)}_n)$ be a random vector whose entries $X^{(n)}_1,\ldots,X^{(n)}_n\sim\mathcal{N}(0,1)$ are i.i.d.\ standard normal random variables.
Then the following event has probability $1$: 
\begin{equation*}
\mnorm{X^{(n)}}_\infty\geq \sqrt{\log(n)}\text{ holds for all but finitely many }n\in\NN.
\end{equation*}
\end{Lem}

Let us emphasize that we do not assume that the random vectors $X^{(1)}, X^{(2)},\ldots$ are independent. 
It is only assumed that the components of $X^{(n)}$ for fixed $n\in\NN$ are jointly independent.

\begin{proof}
We start by estimating the tail probability of $\mnorm{X^{(n)}}_{\infty}$.
Let $(a_n)_{n\in\NN}$ be a sequence of positive numbers.
Then
\begin{equation*}
\PP\lr{\mnorm{X^{(n)}}_{\infty}\leq a_n}=\PP\lr{\abs{X^{(n)}_1}\leq a_n}^n = \erf\lr{\frac{a_n}{\sqrt{2}}}^n,
\end{equation*}
where $ \erf(x)=\frac{2}{\sqrt{\piup}}\int_0^x\ee^{-t^2}\dd t$ is the error function.
To estimate the quantity $\PP(\mnorm{X^{(n)}}_{\infty}\leq a_n)$ further, we first need to establish an inequality for the error function.
An inequality for the so called Mills ratio for standard normal random variables \cite[7.1.13]{Abramowitz1964} provides for $x\geq 0$ that
\begin{equation*}
0\leq\exp(x^2)\frac{\sqrt{\piup}}{2}(1-\erf(x))> \frac{1}{x+\sqrt{x^2+2}}.
\end{equation*}
This in turn implies that for $x\geq 0$, we have
\begin{equation*}
0\leq \erf(x)< 1-\frac{2}{\sqrt{\piup}}\cdot\frac{\exp(-x^2)}{x+\sqrt{x^2+2}}.
\end{equation*}
Using this inequality and the monotonicity of $x^n$ on the positive reals, we get 
\begin{align*}
\erf\lr{\frac{a_n}{\sqrt{2}}}&< \lr{1-\frac{2}{\sqrt{\piup}}\cdot\frac{\exp\lr{-\frac{a_n^2}{2}}}{\frac{a_n}{\sqrt{2}}+\sqrt{\frac{a_n^2}{2}+2}}}\\
&\leq\lr{1-\frac{2}{\sqrt{\piup}}\frac{\exp\lr{-\frac{a_n^2}{2}}}{\sqrt{2}(a_n+1)}},
\end{align*}
where the last inequality follows by the fact that $\lr{\frac{a_n}{\sqrt{2}}+\sqrt{2}}^2\geq \frac{a_n^2}{2}+2$.

We now choose $a_n\defeq \sqrt{2\log(2n)-\delta_n}$ with $0<\delta_n<2\log(2n)$, which gives gives the following estimate
\begin{align*}
\erf\lr{\frac{a_n}{\sqrt{2}}}&< 1-\frac{2\exp\lr{-\log(2n)+\frac{\delta_n}{2}}}{\sqrt{2\piup}\lr{\sqrt{2\log(2n)-\delta_n}+1}}\\
&=1-\frac{1}{\sqrt{2\piup}n\lr{\sqrt{2\log(2n)-\delta_n}+1}}\exp\lr{\frac{\delta_n}{2}}\\
&\leq 1-\frac{1}{\sqrt{2\piup}n\lr{\sqrt{2\log(2n)}+1}}\exp\lr{\frac{\delta_n}{2}}\eqdef 1+x
\end{align*}
Since $x\geq -1$, we conclude
\begin{equation*}
\erf\lr{\frac{a_n}{\sqrt{2}}}^{n}\leq (1+x)^n\leq \exp(xn)= \exp\lr{-\frac{\exp\lr{\frac{\delta_n}{2}}}{\sqrt{2\piup}\lr{\sqrt{2\log(2n)}+1}}},
\end{equation*}
where the last inequality follows from the elementary estimate .
Choosing $\delta_n = K\log(\log(n))$ with $K>0$ and noting that $\delta<2\log(2n)$ for $n\geq n_0=n_0(K)$ yields
\begin{equation} \label{eq:BoundSup}
\PP\lr{\mnorm{X^{(n)}}_\infty\leq a_n}< \exp\lr{-\frac{\log(n)^{\frac{K}{2}}}{\sqrt{2\piup}\lr{\sqrt{2\log(2n)}+1}}}.
\end{equation}
Note that $\sqrt{2\log(2n)}+1\leq \sqrt{2\log(n)}+\sqrt{2\log(2)}+1\leq C\sqrt{\log(n)}$ for some absolute constant $C>0$ and $n\geq 3$.
Therefore,
\begin{equation*}
\frac{\log(n)^{\frac{K}{2}}}{\sqrt{2\piup}\lr{\sqrt{2\log(2n)}+1}}\geq \frac{1}{\sqrt{2\piup}C}\log(n)^{\frac{K-1}{2}}\geq 2\log(n)
\end{equation*}
for fixed $K>3$ and $n\geq n_1=n_1(K)\geq n_0(K)$.
Combined with \eqref{eq:BoundSup} this implies
\begin{equation*}
\PP\lr{\mnorm{X^{(n)}}_\infty\leq a_n}\leq \exp\lr{-2\log(n)}=n^{-2}.
\end{equation*}
Thus
\begin{equation*}
\sum_{n=1}^\infty \PP\lr{\mnorm{X^{(n)}}_{\infty}\leq a_n}\leq \sum_{n=1}^{n_1-1}1+\sum_{n\geq n_1}\frac{1}{n^2}<\infty.
\end{equation*}
The Borel--Cantelli lemma thus shows that almost surely, there are only finitely many $n\in\NN$ with $\mnorm{X^{(n)}}_\infty\leq a_n$.
Further note that
\begin{equation*}
a_n=\sqrt{2\log(2n)-K\log(\log(n))}\geq \sqrt{\log(n)}
\end{equation*}
for all but finitely many $n\in\NN$.
Hence, we see that almost surely, we have $\mnorm{X^{(n)}}_\infty\geq \sqrt{\log(n)}$ for all but finitely many $n\in\NN$.
\end{proof}

\textbf{Acknowledgement.} The authors thank Jakob Lemvig for raising the question considered in this paper, and Andrei Caragea, Dorothee Haroske, and Dominik Stöger for their helpful comments.

\footnotesize
\bibliographystyle{plain}
\bibliography{references.bib}

\end{document}